\documentclass[reqno]{amsart}
\usepackage{textcomp}

\usepackage[dvipsnames]{xcolor}
\usepackage{amsmath}
\usepackage{amsthm}
\usepackage{mathrsfs}
\usepackage{amsfonts,amsthm,amsfonts,amscd,latexsym}
\usepackage{amssymb}
\usepackage{graphicx}
\usepackage{tikz-cd}
\usepackage{epsfig}
\usepackage{flafter}
\usepackage{longtable}
\usepackage{mathtools}
\usepackage{comment}
\usepackage{stmaryrd}
\usepackage{enumitem}
\usepackage{caption}
\usepackage{subcaption}
\usepackage{float}
\usepackage{hyperref}

\hypersetup{
    colorlinks=true,
    linkcolor=blue,
    citecolor=blue,
    filecolor=blue,
    urlcolor=blue
}
\usepackage{tikz}
\usetikzlibrary{graphs,positioning,arrows,shapes.misc,decorations.pathmorphing,matrix,calc}

\tikzset{
    >=stealth,
    every picture/.style={thick},
    graphs/every graph/.style={empty nodes},
}
\tikzstyle{vertex}=[draw,circle,fill=black,inner sep=1pt,minimum width=5pt]

\usepackage{color}

\calclayout

\newcommand{\pp}{\mathbb{P}}

\newcommand{\qq}{\mathbb{Q}}
\newcommand{\zz}{\mathbb{Z}}
\newcommand{\nn}{\mathbb{N}}
\newcommand{\rr}{\mathbb{R}}
\newcommand{\cc}{\mathbb{C}}

\newcommand{\ff}{\mathbb{F}}

\newcommand{\oo}{\mathcal{O}}

\DeclareMathOperator{\Cl}{Cl}

\DeclareMathOperator{\mult}{mult}

\DeclareMathOperator{\relint}{relint}

\newcommand{\drar}{\dashrightarrow}

\newtheorem{introthm}{Theorem}

\newtheorem{theorem}{Theorem}[section]
\newtheorem{lemma}[theorem]{Lemma}
\newtheorem{proposition}[theorem]{Proposition}
\newtheorem{corollary}[theorem]{Corollary}

\theoremstyle{definition}

\newtheorem{definition}[theorem]{Definition}
\newtheorem{example}[theorem]{Example}

\newtheorem{convention}[theorem]{Convention}

\theoremstyle{remark}
\newtheorem{remark}[theorem]{Remark}

\numberwithin{equation}{section}

\usepackage[all]{xy}

\begin{document}

\title[Orbifold Degenerations of Hirzebruch Surfaces]
{Orbifold Degenerations of Hirzebruch Surfaces}

\author[J.P.~Z\'u\~niga]{Juan Pablo Z\'u\~niga}
\address{UCLA Mathematics Department, Box 951555, Los Angeles, CA 90095-1555, USA}
\email{jpzuniga@math.ucla.edu}


\begin{abstract}
We study orbifold degenerations of Hirzebruch surfaces. Our main theorem shows that every such degeneration $X$ arises as a partial smoothing of a toric surface. We then give a combinatorial description of the singularities that arise when $-K_X$ is not nef. This complements previous joint work with G. Urzúa, which treated the case in which $-K_X$ is ample. For Hirzebruch surfaces $\mathbb{F}_k$ with $k\geq 3$, we obtain an explicit description of all possible central fibers.
For $k\leq 1$, we use the threefold minimal model program to reduce the problem to the case of central fibers whose anticanonical divisor is nef. The remaining surfaces are toric del Pezzo surfaces of degree $8$ with T-singularities. We classify these by studying the birational geometry arising from mutations of Fano polygons.
\end{abstract}

\maketitle

\setcounter{tocdepth}{1}
\tableofcontents

\section{Introduction}
\label{sec:intro}
Describing the mildly singular degenerations of a fixed smooth surface is a central problem in algebraic geometry. For instance, degenerations are relevant for constructing compactifications of moduli spaces. This approach was established in the work of Kollár and Shepherd-Barron \cite{KSB88} and Alexeev \cite{Ale94} to construct a compactification of the moduli space of surfaces of general type. Similarly, Hacking \cite{H04} constructed a compactification of the moduli space of degree $d$ plane curves, for which elements in the boundary are stable pairs $(X,D)$ where $X$ is a degeneration of $\pp^2$. For extensive literature in this topic, see for example \cite{DHL13}. 

In this article we study the orbifold degenerations of the Hirzebruch surfaces $\ff_k$.
The problem is to describe the surfaces that occur as central fibers, for a fixed general fiber. This is also a problem of threefold geometry. The total space of a degeneration over a curve germ is a $3$-fold.
Two degenerations of the same surface often have total spaces that are birational over the base.
The birational maps relating them are steps of a minimal model program, see \cite{M02} and \cite{HTU17}.
We will use this point of view to modify central fibers.

Orbifold degenerations of smooth rational surfaces have been studied by many authors and from different perspectives, see for example \cite{B85,Man91,HP10,UZ25,MZ25}.
The study of normal degenerations of rational and ruled surfaces was initiated by B\u{a}descu \cite{B85}.
Manetti proved that a normal degeneration of $\pp^2$ with quotient singularities carries at most three singular points, all of them Wahl singularities \cite{Man91}.
A Wahl singularity is a smoothable surface quotient singularity with vanishing Milnor number.
The complete answer for $\pp^2$ is due to Hacking and Prokhorov, who classified the smoothable del~Pezzo surfaces with quotient singularities and Picard rank one \cite[Theorem~1]{HP10}.
For del~Pezzo surfaces of degree $9$ their result says the following.
Every orbifold degeneration $\pp^2\rightsquigarrow X$ is a partial $\qq$-Gorenstein smoothing of a weighted projective plane $\pp(a^2,b^2,c^2)$, where $(a,b,c)$ solves the Markov equation.
Prokhorov later studied the degenerations of del~Pezzo surfaces with ample anticanonical divisor \cite{P15}.
Peng described the $\mathbb{G}_m$-degenerations of a del~Pezzo surface $X$ induced by the log canonical places of a pair $(X,C)$, where $C\in|-K_X|$ is a nodal curve \cite{P26}. However, a classification of orbifold degenerations of rational surfaces is still a wide open problem. 

Two previous works are directly relevant to this paper.
In \cite{UZ25} Urz\'ua and the author determined which Wahl singularities occur in a degeneration of a del~Pezzo surface of any degree.
The answer is a condition on the Hirzebruch--Jung continued fraction of the singularity.
Each continued fraction satisfying it determines a surface, which is unobstructed to deform.
In \cite{MZ25} Moraga and the author studied degenerations of toric and of cluster type pairs.
That paper gives criteria for the existence of nodal anticanonical elements in degenerations of singular toric surfaces of Picard rank one.
As a consequence, they classify the orbifold degenerations of $\pp(1,1,k)$ with $k\geq 3$.

We state our results in the language of $W$-surfaces.
A $W$-surface is a $\qq$-Gorenstein degeneration $\pi\colon(X\subset\mathcal{X})\to(0\in\mathbb{D})$ over a smooth curve germ, with smooth general fiber, whose central fiber $X$ is reduced and has only Wahl singularities, see Definition~\ref{def:W-surface}.
This entails no loss of generality here.
If $X$ is an orbifold degeneration of $\ff_k$, then either all singularities of $X$ are Wahl and $\rho(X)=2$, or $\rho(X)=1$ and such $X$ is classified by \cite[Theorem~1.1]{HP10}.

\subsection{Toric degenerations and nodal anticanonical elements}\label{ss:toric}

Our first theorem reduces the classification to the toric case.

\begin{introthm}\label{ithm:toric}
Let $\pi\colon (X \subset \mathcal{X}) \to (0 \in \mathbb{D})$ be a $W$-surface with general fiber $\mathcal{X}_t \cong \ff_k$.
Then there exists a $\qq$-Gorenstein deformation $\pi'\colon \mathcal{X}' \to \mathbb{D}$ such that the general fiber $\mathcal{X}'_t$ is isomorphic to $X$ and the central fiber $\mathcal{X}'_0$ is a toric surface.
\end{introthm}

In other words, every orbifold degeneration of a Hirzebruch surface is a partial $\qq$-Gorenstein smoothing of a toric surface with T-singularities.
Any description of the possible central fibers may therefore start from a list of toric surfaces.

The geometric input for Theorem~\ref{ithm:toric} is the existence of a nodal anticanonical element on $X$.
In Theorem~\ref{thm:1-compldegHz} we prove that every such $X$ carries a reduced divisor $B\in|-K_X|$ with $(X,B)$ log canonical.
When $X$ is singular and $\rho(X)=2$, $B$ is a cycle of rational curves passing through every singular point of $X$.
This is the Picard rank two analogue of \cite[Theorem~7.1]{HP10}, and it holds for every $k\geq 0$.
It already bounds the number of singular points of $X$ by four, with equality only for $X$ toric, see Proposition~\ref{prop:toricity}.
The divisor $B$ is what allows us to degenerate $X$.
We deform it through a sequence of local $\qq$-Gorenstein deformations, which we call \emph{slides} (see \cite[Definition~4.6]{UZ25}), until a toric surface is reached.

One class of central fibers admits no slide, namely the surfaces of type $(\dagger)$ of Definition~\ref{def:dagger-type}.
These are certain blow-ups of degenerations of $\pp^2$ at a point of the toric boundary.
For them the toric degeneration is obtained by an explicit construction, see Proposition~\ref{prop:E0-two-sings}.

\subsection{The Mori cone and the smoothing index}\label{ss:features}

Two features distinguish our setting from that of del~Pezzo surfaces of Picard rank one.

The first feature concerns the Mori cone of a singular central fiber.
It has exactly two extremal rays, generated by curves $\Gamma_0$ and $\Gamma_1$, see Proposition~\ref{prop:moriray}.
When $-K_X$ is not nef one has $K_X\cdot\Gamma_0<0$ and $K_X\cdot\Gamma_1>0$.
The curve $\Gamma_1$ is thus the only obstruction to the nefness of $-K_X$, and both classification results below are organized around it.

The second feature concerns that $k$ is not an invariant of $X$.
A smoothing to $\ff_k$ also gives a smoothing to $\ff_{k-2d}$ for every $d\geq0$ with $k-2d\geq0$, see Theorem~\ref{thm:consc-def}. Thus only $k\bmod 2$ is a deformation invariant, which agrees with the homeomorphism type of the Hirzebruch surface.
The quantity attached to $X$ is therefore the largest $k$ for which a smoothing $\ff_k\rightsquigarrow X$ exists, which we call the \emph{smoothing index} $m(X)$. In Subsection \ref{ss:sbounds} we provide bounds for $m(X)$.

\subsection{Degenerations of $\ff_k$ with $k\geq 3$}\label{ss:k3}

Theorem~\ref{ithm:toric} does not say which surfaces occur, nor which singularities they carry.
For $k\geq 3$ we answer both questions.
The two ranges $k\geq 3$ and $k\leq 1$ require different methods, and the case $k=2$ is not treated in this paper.

The case $k\geq 3$ is distinguished by the following fact.
The anticanonical divisor of a central fiber is never nef.
Hence the extremal curve $\Gamma_1$ with $K_X\cdot\Gamma_1>0$ is always present. The Wahl singularities lying on $\Gamma_1$ are constrained, and this is what makes an explicit answer possible.

We encode those singularities by a condition on their Hirzebruch--Jung continued fractions.
A Wahl chain satisfying it is called a \emph{Hirzebruch Wahl chain of index $k$}, see Definition~\ref{def:hirzebruch-wahl-chain}.
This condition was inspired by the analogous del~Pezzo Wahl chains of \cite[Definition~1.4]{UZ25}, with an extra entry recording the index $k$. A Hirzebruch Wahl chain not only records singularities, but also a surface.
We call this surface the associated \emph{marked surface}, see Definition~\ref{def:marked-surface}.
A marked surface carries one or two Wahl singularities, and we denote it by $X_\bullet$.
Marked surfaces are unobstructed, and every smoothing of $X_\bullet$ has a Hirzebruch surface as general fiber, see Proposition~\ref{prop:markedunobstructed}.
The next theorem shows that they are the models from which the central fibers are obtained. Assume that $X$ is not of $(\dagger)$ type, see Remark \ref{rem:not-dagger}.

\begin{introthm}\label{ithm:marked}
Let $\pi\colon (X \subset \mathcal{X}) \to (0 \in \mathbb{D})$ be a $W$-surface with general fiber $\mathcal{X}_t \cong \ff_k$. Assume that $X$ is singular and that $-K_X$ is not nef.
Then there exists a $\qq$-Gorenstein deformation $\pi'\colon \mathcal{X}' \to \mathbb{D}$ such that the general fiber $\mathcal{X}'_t$ is a marked surface $X_{\bullet}$ and $\mathcal{X}'_0\cong X$.
\end{introthm}

Theorem~\ref{ithm:marked} is proved in Theorem~\ref{thm:Hz-deformations}. Imposing $-K_X$ not nef is the case in which $\Gamma_1$ is available for extraction in families.
For $k\geq 3$ the hypothesis holds for every singular central fiber, so Theorem~\ref{ithm:marked} classifies the possible central fibers in that range.
The Hirzebruch Wahl chain of the marked surface determines the singularities of $X$.
Its index is governed by the smoothing index of $X$ rather than by $k$ itself.
For $k\leq 1$ the hypothesis is a genuine restriction, and the remaining central fibers are reached by Theorem~\ref{ithm:antiflip} below.

\subsection{Degenerations of $\ff_k$ with $k\leq 1$}\label{ss:k01}

For $k\leq 1$ the general fiber is a del~Pezzo surface, but the central fiber need not have $-K_X$ nef, see for example \cite{UZ24}.
Our strategy for a classification is therefore to remove the obstruction to nefness first.

That obstruction is the extremal curve $\Gamma_1$ with $K_X\cdot\Gamma_1>0$.
Here the $3$-fold viewpoint enters.
Removing $\Gamma_1$ is a small birational modification on the total space of the degeneration, namely a terminal $K$-antiflip.
An antiflip replaces the central fiber by a new one with the same general fiber.
Its existence is not automatic: it is controlled by the extremal $P$-resolutions of the cyclic quotient singularity to which $\Gamma_1$ contracts, and by the local data of the smoothing along $\Gamma_1$ \cite[Corollary~3.23]{HTU17}.
In Theorem~\ref{thm:admissible-smoothing} we construct smoothings for which the antiflip is available.
The resulting sequence of antiflips terminates, because these operations are log flips of a klt $3$-fold pair.

\begin{introthm}\label{ithm:antiflip}
Let $\pi\colon (X \subset \mathcal{X}) \to (0 \in \mathbb{D})$ be a $W$-surface with general fiber $\mathcal{X}_t \cong \ff_k$ and suppose that $-K_X$ is not nef.
Set $l_0\equiv k\bmod 2$.
Then $X$ admits a $\qq$-Gorenstein smoothing $\mathcal{X}_0\to\mathbb{D}$ with general fiber $\ff_{l_0}$ and a finite sequence of terminal antiflips over $\mathbb{D}$
\begin{equation}\label{eq:antiflips}
\mathcal{X}_0\drar\mathcal{X}_1\drar\cdots\drar\mathcal{X}_N,
\end{equation}
whose final central fiber $X_N$ satisfies that $-K_{X_N}$ is nef.
\end{introthm}

Together with Theorem~\ref{ithm:toric}, this places every degeneration of $\ff_0$ and of $\ff_1$ in a clear framework.
Slides degenerate the final central fiber $X_N$ to a toric surface with nef anticanonical divisor.
For $k=1$ the conclusion sharpens to $-K_{X_N}$ ample, see the remark following Corollary~\ref{cor:antiflip-termination}.
For $k=0$ the non-ample case does occur, and it is treated in Subsection \ref{subsec:F0}.

We are thus reduced to the toric central fibers with $-K_X$ ample.
Since $K_X^2=8$, these are the toric del~Pezzo surfaces of degree $8$ with T-singularities.
The remaining problem is combinatorial.

Such surfaces are described by Fano polygons.
A Fano polygon $P$ determines a toric del~Pezzo surface $X_P$. For Fano polytopes the notion of combinatorial mutation was defined in \cite{ACGK12} in the realm of mirror symmetry, see also \cite{ACC16} and \cite{KNP17}.
By \cite[Theorem~1.3]{I12}, a mutation $P\rightsquigarrow Q$ is realized by a pencil $\pi\colon\mathcal{X}\to\pp^1$ with fibers $X_P$ and $X_Q$ over $0$ and $\infty$.
We call it the \emph{Ilten pencil} of the mutation, see also \cite[Definition 15]{C25}.
By \cite[Theorem~6]{KNP17}, every toric degeneration $\ff_k\rightsquigarrow X_P$ with $-K_{X_P}$ ample satisfies $P\in[P_{\ff_k}]$, the mutation-equivalence class of the polygon of $\ff_k$.
It remains to describe the two classes $[P_{\ff_0}]$ and $[P_{\ff_1}]$ geometrically.

We study the Ilten pencil via birational geometry.
We work with the explicit form of the pencil due to Petracci \cite[Theorem~7.3]{P21}. We find two types of birational modifications that govern the whole geometric picture.
In the first one, the pencil becomes a mutation of Fano triangles after a divisorial contraction.
In the second one, the pencil fixes the cyclic quotient singularity obtained by contracting an extremal curve.
It then moves the surface one step along the universal family of extremal neighborhoods over that singularity, see \cite[Section~3.4]{HTU17}. In \cite[Section~5]{UZ24} this family was also denoted as \emph{Mori trains}. Thus, the mutated polygon describes the next term of such a family. We call a mutation of the second kind a \emph{Mori mutation}, see \cite[Definition 4.21]{UZ25}. 

\begin{introthm}\label{ithm:mutations}
Let $P\in[P_{\ff_k}]$ be a singular Fano quadrilateral with no pair of parallel edges, let $(w,F)$ be mutation data for an edge $E$ of $P$, and set $Q=\operatorname{mut}_w(P,F)$.
Let $\pi\colon\mathcal{X}\to\pp^1$ be the pencil with $\pi^{-1}(0)=X_P$ and $\pi^{-1}(\infty)=X_Q$ associated to the mutation.
Then, according to a numerical criterion on $E$, exactly one of the following holds:
\begin{enumerate}
\item[\textnormal{(i)}] \textnormal{(Mutations of Fano triangles.)}
A divisorial contraction identifies $\pi\colon\mathcal{X}\to\pp^1$ with a one-step mutation of a Fano triangle, and the singularities of $X_Q$ are determined by those of $X_P$.
\item[\textnormal{(ii)}] \textnormal{(Mori mutations.)}
The curves $\Gamma\subset X_P$ and $\Gamma'\subset X_Q$ contract to one and the same cyclic quotient singularity $\frac{1}{\Delta}(1,\Omega)$, and the resulting extremal neighborhoods are two consecutive $k2A$ in the universal family over it.
\end{enumerate}
If $P$ has a pair of parallel edges, then $P\in[P_{\ff_0}]$, and a mutation along either of these edges produces Fano triangles $Q_1$ and $Q_2$.
The two triangles induce the weighted projective planes $X_{Q_1}=\pp(a^2,b^2,2c^2)$ and $X_{Q_2}=\pp(a^2,b^2,2(c')^2)$, where $c'=2ab-c$ and both $(a,b,c)$ and $(a,b,c')$ are solutions of
\begin{equation}\label{eq:F_0}
x^2+y^2+2z^2=4xyz .
\end{equation}
\end{introthm}

The two cases of Theorem~\ref{ithm:mutations} are exhaustive for the edges of $P$.
The criterion separating them is numerical, and is stated in Theorems~\ref{thm:triangle-mutations} and~\ref{thm:Mori-mutations}.
The hypothesis that $P$ be a quadrilateral is where Picard rank two enters.
The parallel-edge case is exactly where our classification meets the Picard rank one surfaces of \cite[Theorem~4.1]{HP10}.
There the passage between the two triangles is not a single mutation using Definition \ref{def:mut}.

\subsection*{Organization of the paper}

The paper is organized as follows.
Section~\ref{sec:prelims} collects the preliminary material on singularities of pairs, complements, T-singularities and Wahl singularities, $W$-surfaces and their minimal model program, and Fano polygon mutations.
In Section~\ref{sec:complements} we prove the existence of $1$-complements for degenerations of Hirzebruch surfaces, Theorem~\ref{thm:1-compldegHz}, and deduce Theorem~\ref{ithm:toric}.
In Section~\ref{sec:degen-k3} we treat the range $k\geq 3$, introduce Hirzebruch Wahl chains and marked surfaces, and prove Theorem~\ref{ithm:marked}.
In Section~\ref{sec:degen-k01} we treat the range $k\leq 1$ and prove Theorem~\ref{ithm:antiflip}.
We then analyze the classes $[P_{\ff_0}]$ and $[P_{\ff_1}]$ through the Ilten pencils, proving Theorem~\ref{ithm:mutations} and describing the degenerations of $\ff_0$ in Subsection \ref{subsec:F0}.

\subsection*{Acknowledgements}
This work forms part of the author's PhD thesis. The author is grateful to Joaquín Moraga for his hospitality and for many useful conversations during visits to UCLA in Summer 2025 and Winter 2026. The author also thanks Giancarlo Urzúa, Vicente Monreal and Audric Lebovitz for their helpful comments, and in particular Audric Lebovitz for pointing out an error in an earlier version of the manuscript. The author was supported by the ANID National Doctoral Scholarship 2022--21221224.

\section{Preliminaries}\label{sec:prelims}

We work over the field of complex numbers $\cc$. All varieties are assumed to be projective and normal unless otherwise stated. The degenerations considered in this paper are $\qq$-Gorenstein and are taken over a smooth curve germ $(0\in\mathbb{D})$, with reduced central fiber. In this section we collect the preliminary material on singularities of pairs, T-singularities and Wahl singularities, degenerations of smooth rational surfaces, and Fano polygon mutations that will be used throughout.

\subsection{Singularities of pairs} 
In this subsection, we recall some basic definitions of singularities of pairs. 
For standard definitions of singularities of the MMP, for example, klt and lc pairs, we refer the reader to~\cite{KM98,Kol13}. In this paper, we consider pairs $(X,B)$ with rational coefficients, i.e., $B$ is a $\qq$-divisor. Furthermore, given a pair $(X,B)$ and a divisor $E$ over $X$, we write $a_E(X,B)$ for the log discrepancy of $(X,B)$ with respect to $E$. 

\begin{definition}
{\em 
A pair $(X,B)$ is said to be {\em purely log terminal} or {\em plt} for short if the pair $(X,B)$ is log canonical
and $a_E(X,B)>0$ for every divisor $E$ over $X$ which is exceptional over $X$. 
}
\end{definition}

Observe that klt pairs are plt. However, there are plt pairs which are not klt. For instance, $(\mathbb{A}^2, H_0)$ where $H_0$ is a hyperplane in $\mathbb{A}^2$.

\begin{definition}
{\em  
A {\em log Calabi--Yau pair} is a pair $(X,B)$ which has log canonical singularities
for which $K_X+B\sim_{\qq} 0$.
}
\end{definition}

\begin{definition}
The complexity of a $\log$ Calabi-Yau pair $(X,B)$ is the following value
$$c(X, B):=\operatorname{dim}X+\operatorname{dim} \mathrm{Cl}_{\qq}(X)-|B|,$$
where $|B|$ stands for the sum of the coefficients of $B$. 
\end{definition}
The following is a numerical criterion to detect when a log Calabi Yau pair $(X,B)$ corresponds exactly to a toric pair, i.e., $X$ is a toric variety and $B=\sum_{\rho\in\Delta(1)}D_\rho$, where $D_\rho$ is the torus invariant divisor associated to the ray $\rho$.
\begin{theorem}[{\cite[Theorem 1.2]{BMSZ18}}]
Let $(X, B)$ be a log Calabi-Yau pair. Then, the inequality $c(X, B)\geq 0$ holds. Furthermore, if $c(X,B)<1$, then the pair $(X,\lfloor B\rfloor)$ is toric.
\end{theorem}

\subsection{T-singularities and Wahl singularities}\label{subsec:T-sing}
 
\begin{definition}[{\cite[Definition 3.7, Proposition 3.10]{KSB88}}]\label{def:T-sing}
A \emph{T-singularity} is a 2-dimensional quotient singularity that admits a $\qq$-Gorenstein smoothing. It is either a Du Val singularity or a cyclic quotient singularity of the form $\frac{1}{dn^2}(1,dna-1)$, where $d\geq 1$ and $0<a<n$ are coprime integers.

\medskip
A \emph{Wahl singularity} is a T-singularity with $d=1$, i.e., a cyclic quotient singularity of the form $$\frac{1}{n^2}(1,na-1).$$
\end{definition}

For $P\in X$ a $T$-singularity the \emph{Milnor number} is $\mu_P:=b_2(M_P)$, where $M_P$ is the Milnor fiber of its $\qq$-Gorenstein smoothing.
One has $\mu_P=r$ for a Du Val singularity of type $A_r$, $D_r$ or $E_r$, and $\mu_P=d-1$ for $\frac{1}{dn^2}(1,dna-1)$; see \cite[Section~3]{Man91}.
In particular Wahl singularities have $\mu_P=0$. For normal projective surfaces with at most T-singularities, we have the following Noether's formula.
\begin{proposition}[{\cite[Proposition~3.6]{HP10}}]\label{prop:Noether}
Let $X$ be a normal projective surface with T-singularities.
Then
\[
K_X^2+\chi_{\operatorname{top}}(X)+\sum_{P\in\operatorname{Sing}(X)}\mu_P=12\chi(\mathcal{O}_X).
\]
\end{proposition}

\subsection{Degenerations of smooth rational surfaces}

\begin{definition}\label{def:W-surface}
A \textit{$W$-surface} is a proper flat morphism
$\pi:\mathcal{X}\to \mathbb{D}$
satisfying the following conditions:
\begin{enumerate}
\item $\mathcal{X}$ is a normal complex threefold and $K_{\mathcal{X}}$ is $\mathbb{Q}$-Cartier;
\item the central fiber $\mathcal{X}_0:=\pi^*\{0\}$ is a reduced normal surface with at worst Wahl singularities;
\item the general fiber $\mathcal{X}_t$ is smooth for every $t\neq 0$.
\end{enumerate}
\end{definition}

If $\mathcal{X}\rightarrow \mathbb{D}$ is a $W$-surface then the central fiber $\mathcal{X}_0$ is klt and the equality 
\[
(K_{\mathcal{X}}+\mathcal{X}_0)|_{\mathcal{X}_0}\sim_\qq K_{\mathcal{X}_0} 
\]
holds. Indeed, as $\pi^*\{0\}=\mathcal{X}_0$ the variety $\mathcal{X}$ does not have singularities of codimension two along $\mathcal{X}_0$. For such degenerations we denote $\mathcal{X}_t\rightsquigarrow X$.

Let $(X\subset\mathcal{X})\to (0\in\mathbb{D})$ be a $W$-surface where $X$ has Wahl singularities $P_1,\ldots,P_s$ of respective types $\frac{1}{n_i^2}(1,n_ia_i-1)$.
Let $L_i$ denote the link of $P_i$ and $M_i$ the corresponding Milnor fiber.
For Wahl singularities, the link $L_i$ is a lens space with $H_1(L_i,\mathbb{Z})\cong \mathbb{Z}/n_i^2$, and the Milnor fiber satisfies $H_1(M_i,\mathbb{Z})\cong \mathbb{Z}/n_i$ and $H_2(M_i,\mathbb{Z})=0$; see \cite[Section~4.1]{UCBM25}.
The smooth surgery $\mathcal{X}_t\rightsquigarrow X$ induces the following long exact sequence in integral homology, see \cite[Section~4.1]{UCBM25} and \cite[Lemma 2.2.3]{DHL13}:
\begin{equation}\label{eq:long-exact}
0\to H_2(\mathcal{X}_t,\mathbb{Z})\to H_2(X,\mathbb{Z})\to \bigoplus_{i=1}^s H_1(M_i,\mathbb{Z})\to H_1(\mathcal{X}_t,\mathbb{Z})\to H_1(X,\mathbb{Z})\to 0.
\end{equation}

We now restrict to the case of $\qq$-Gorenstein degenerations of smooth rational surfaces with $K^2\geq 1$.

\begin{theorem}[{\cite[Lemma~1.2]{Man91}; \cite[Theorem~1]{B85}}]\label{thm:rational-degeneration}
Let $\mathcal{X}\to\mathbb{D}$ be a normal projective degeneration with general fiber $\mathcal{X}_t$. Assume that
$p_g(\mathcal{X}_t)=q(\mathcal{X}_t)=0$. Then
$\rho(\mathcal{X}_0)\leq \rho(\mathcal{X}_t)$.

Moreover, if $\mathcal{X}_t$ is rational and $K^2_{\mathcal{X}_t}\geq 0$, and $\mathcal{X}_0$ has at most rational singularities, then $\mathcal{X}_0$ is rational.
\end{theorem}

Since $T$-singularities, and in particular Wahl singularities are rational, we always end up with a rational central fiber $X=\mathcal{X}_0$ by \cite[Theorem 1]{B85}. Consequently, the minimal resolution $Y\to X$ is obtained via a sequence of blow-ups over a Hirzebruch surface $\ff_k$.
We now combine the topological and algebraic information to reduce the long exact sequence~\eqref{eq:long-exact} to a short exact sequence.
Since $\mathcal{X}_t$ is a smooth rational surface, we have $H_1(\mathcal{X}_t,\mathbb{Z})=0$ and $H_2(\mathcal{X}_t,\mathbb{Z})\cong \operatorname{Pic}(\mathcal{X}_t)$. Exactness of (\ref{eq:long-exact}) implies $H_1(X,\mathbb{Z})=0$.
The conditions $H^1(X,\mathcal{O}_X)=H^2(X,\mathcal{O}_X)=0$ yield the identification $H_2(X,\mathbb{Z})\cong \operatorname{Cl}(X)$; see \cite[Proposition 4.11]{K05}. Substituting these identifications into \eqref{eq:long-exact}, we obtain the short exact sequence
\begin{equation}\label{eq:short-exact}
0\to \operatorname{Pic}(\mathcal{X}_t)\to \operatorname{Cl}(X)\to \bigoplus_{i=1}^s \mathbb{Z}/n_i\mathbb{Z}\to 0.
\end{equation}

On the other hand, the function $t\mapsto h^i(\mathcal{X}_t,nK_{\mathcal{X}_t})$ is upper semi-continuous for any $n\in\zz$, see \cite[Lemma 1.3]{B85}. Thus, that the central fiber $X=\mathcal{X}_0$ satisfies $h^0(X,-K_X)\geq 1$. The following result shows that $X$ is always unobstructed. See also \cite[Theorem 3.1]{HP10}.

\begin{proposition}[{\cite[Proposition 5.3]{M96}}]\label{prop:no-obstruction}
Let $X \subset \pp^n$ be a normal projective surface with $H^2(X,\oo_X)=H^1(X,\oo_X)=0$, $h^0(X,-K_X)\geq 1$ with at most rational singularities. Then $H^2\left(X,T_X\right)=H^1\left(X,\mathcal{O}_X(1)\right)=0$.
\end{proposition}
In the range $K_X^2\geq 1$, the minimal resolution $Y$ also carries a $\pp^1$-fibration $Y\to\pp^1$, which is the combinatorial framework for the exceptional divisor of $Y\to X$ used in \cite{Man91,HP10,UZ25,MZ25}.
For coprime integers $0<a<n$ we write $\frac{n}{a}=[e_1,\dots,e_r]$ for the Hirzebruch--Jung continued fraction of the cyclic quotient singularity $\frac1n(1,a)$, abbreviated c.q.s., whose entries are the negatives of the self-intersections of the exceptional curves of its minimal resolution; see \cite[Definition 2.1]{UZ24}.
 
\begin{definition}[{\cite[Definition 1.2]{UZ25}}]\label{def:zcf}
A chain $[f_1,\dots,f_r]$ with $f_i\geq 2$ \emph{admits a zero continued fraction of weight $\lambda$} if there are indices $i_1<\dots<i_v$ and integers $d_{i_j}\geq 1$ such that
\begin{equation}\label{eq:zcf-def}
[\,\dots,f_{i_1}-d_{i_1},\dots,f_{i_v}-d_{i_v},\dots\,]=0
\qquad\text{and}\qquad
\lambda+1=\sum_{j=1}^{v}d_{i_j}.
\end{equation}
\end{definition}
 
Zero continued fractions govern the versal deformation space of a c.q.s. $\frac1\Delta(1,\Omega)$.
By \cite[Theorem~3.9]{KSB88} the irreducible components of $\operatorname{Def}\bigl(\frac1\Delta(1,\Omega)\bigr)$ are in bijection with its \emph{P-resolutions}, that is, with the partial resolutions having only T-singularities and relatively ample canonical class, and these are enumerated by zero continued fractions in \cite{Chr91,Ste91}; see also \cite{BC94}.
This framework turns the study of deformations of algebraic surfaces into $3$-fold birational geometry.
 
Let $(X\subset\mathcal{X})\to(0\in\mathbb{D})$ be a $W$-surface with $K_X$ not nef, and let $\Gamma\subset X$ be a $K_X$-negative extremal curve with $\Gamma^2<0$.
The contraction of $\Gamma$ extends to an extremal neighborhood of $\mathcal{X}$ over a $cA_n$ germ, hence is semistable in the sense of \cite[Section~2.2]{KM92}, of type $k1A$ or $k2A$ according as $\Gamma$ passes through one or two Wahl singularities of $X$.
Its numerical data, and in particular whether it is divisorial or flipping, is computed by the algorithm of Mori \cite[Section~2]{M02}; see also \cite{HTU17,U16,UZ24}.
Both a divisorial contraction and a flip return a $W$-surface, see for example \cite[Section~5]{UZ24}.
 
The flip is a morphism $(\Gamma^+\subset X^+\subset\mathcal{X}^+)\to(P\in\bar{X}\subset\bar{\mathcal{X}})$ with $\mathcal{X}^+_t\cong\mathcal{X}_t$ for $t\neq 0$, whose restriction $(\Gamma^+\subset X^+)\to(P\in\bar{X})$ is a P-resolution with $X^+$ carrying at most two singular points, both Wahl, by \cite[Lemma 3.14]{KSB88}.
Such a P-resolution is called an \emph{extremal P-resolution}, see \cite[Section~4]{HTU17}.
We record these partial resolutions in the notation of \cite[Notation~1.7]{TU22}:
\[
\left[\binom{n_0}{a_0}\right]-(c_1)-\left[\binom{n_1}{a_1}\right]\longrightarrow\tfrac1\Delta(1,\Omega),
\]
where $\left[\binom{n_i}{a_i}\right]$ is the Wahl singularity $\frac{1}{n_i^2}(1,n_ia_i-1)$,  $\Gamma$ is a nonsingular rational curve with $P_0,P_1\in\Gamma$ forming a toric boundary at each $P_i$, and $(c_1)$ records that the proper transform of $\Gamma$ in the minimal resolution has self-intersection $-c_1$.
When $P_i$ is a smooth point we set $n_i=1$ and omit the bracket.
The possible outcomes of this semistable minimal model program for $W$-surfaces are described in \cite[Section 2]{Urz16a}.

\subsection{Deformation theory of degenerations of Hirzebruch surfaces}

The $\qq$-Gorenstein deformation theory of a normal projective surface $X$ with T-singularities is governed by the local-to-global exact sequence
\begin{equation}\label{eq:loc-glob-QG}
0 \to H^1(X, T_X) \to T^1_{\qq\mathrm{G},X} \to \bigoplus_{P \in \operatorname{Sing}(X)} T^1_{\qq\mathrm{G},P} \to H^2(X, T_X) \to T^2_{\qq\mathrm{G},X} \to 0,
\end{equation}
where $H^1(X,T_X)$ parametrizes equisingular deformations, $T^1_{\qq\mathrm{G},P}$ is the tangent space to the local $\qq$-Gorenstein deformation space at $P$, and $T^2_{\qq\mathrm{G},X}$ is the space of obstructions to $\qq$-Gorenstein deformations; see \cite[Section~3]{H04} and \cite[Section 3.2]{UCBM25}.

Suppose $X$ admits a $\qq$-Gorenstein smoothing to $\ff_k$.
Then $\rho(X)\leq 2$ by Theorem~\ref{thm:rational-degeneration}. Consequently, by Proposition~\ref{prop:Noether} it follows that $$\rho(X)+\sum_{P\in\operatorname{Sing}(X)}\mu_P=2.$$
Hence $\rho(X)=2$ if and only if every singularity of $X$ is Wahl, and $\rho(X)=1$ if and only if $X$ carries a single non-Wahl singularity $P$ with $\mu_P=1$.
The latter case occurs only for $k=0$ and is classified by \cite[Theorem~4.1]{HP10}.

\begin{theorem}\label{thm:consc-def}
Let $X$ be a normal projective surface with rational singularities and admitting a $\qq$-Gorenstein smoothing $\pi:(X\subset\mathcal{X})\to (0\in\mathbb{D})$ where $\mathcal{X}_t\cong\ff_k$ with $k\geq 2$. Then, for each $d\geq 0$ satisfying $k-2d\geq 0$ there exists another $\qq$-Gorenstein smoothing of $X$, $\pi^\prime:(X\subset\mathcal{X}^\prime)\to (0\in\mathbb{D})$ such that $\mathcal{X}^\prime_t\cong \ff_{k-2d}$ for $t\neq 0$. 
\end{theorem}
\begin{proof}
Suppose that there exists a $\qq$-Gorenstein smoothing $\pi:(X\subset\mathcal{X})\to (0\in\mathbb{D})$ with general fiber $\ff_k$ with $k\geq 2$. Since $S:=\operatorname{Def}^{\qq G}(X)$ is a miniversal $\qq$-Gorenstein deformation space of $X$, there exists a miniversal family $\mathcal{X}^{\text{univ}}\to S$ with distinguished origin $0_S$ corresponding to $X$. The smoothing $\pi$ is induced by a base change from the miniversal family. Indeed, there exists a morphism of germs $\phi:(\mathbb{D},0)\to (S,0_S)$ such that $\mathcal{X}\cong \mathcal{X}^{\text{univ}}\times_S \mathbb{D}$.
Let $p = \phi(t)$ for some $t \neq 0$ sufficiently small. Since $\mathcal{X}_t\cong\ff_k$, the fiber of the miniversal family over $p$ is given by $\mathcal{X}_p^{\text{univ}}\cong\ff_k$.
Since $\mathcal{X}^{\text{univ}}\to S$ is miniversal, it is versal at $0_S$. By openness of versality (see for example \cite{A74}), there exists an open subset $U\subset S$ of $0_S$ such that $\mathcal{X}^{\text{univ}}\to S$ is versal in every point of $U$. In particular, the family is versal at $p$. 
This implies the existence of a holomorphic map of germs:
$$ \Phi: (S, p) \longrightarrow (\operatorname{Def}(\ff_k), 0) $$
such that $\mathcal{X}|_{(S,p)} \cong \Phi^* \mathcal{F}$, where $\mathcal{F}$ is the Kuranishi family of $\ff_k$. The versality of $\mathcal{X}$ at $p$ corresponds analytically to the surjectivity of the induced differential of $\Phi$,
$$ d\Phi_p: T_p S \twoheadrightarrow T_0 \operatorname{Def}(\ff_k) \cong H^1(\ff_k, T_{\ff_k}). $$
Since $H^2(\ff_k, T_{\ff_k}) = 0$, this vanishing of obstructions implies that the base space $\operatorname{Def}(\ff_k)$ is smooth of dimension $h^1(\ff_k, T_{\ff_k})=k-1$, and particularly analytically isomorphic to $(\mathbb{C}^{k-1}, 0)$. 
A holomorphic map to a smooth target is a submersion if and only if it is surjective on tangent spaces. Therefore, the combination of versality and unobstructedness implies that $\Phi$ is a submersion around $p$. Consequently, $\Phi$ is an open map in a neighborhood of $p$.

We define $$C_d:=\{s\in S|\hspace{0.1cm} \mathcal{X}^{\text{univ}}_s\cong \ff_{k-2d}\}.$$ 
Since $X$ is projective, by \cite[Theorem~1.6]{Art69} the miniversal deformation algebraizes, i.e., there exist an algebraic scheme $S^{\mathrm{alg}}$ of finite type over $\cc$ and a projective flat family $\mathcal{X}^{\mathrm{alg}}\to S^{\mathrm{alg}}$ whose formal completion at a closed point $0\in S^{\mathrm{alg}}$ recovers the germ of $\mathcal{X}^{\mathrm{univ}}\to S$. We then identify $(S,0)\cong ((S^\mathrm{alg})^{\mathrm{an}},0)$. The Isom-scheme $\mathrm{Isom}_{S^{\mathrm{alg}}}(\ff_{k-2d}\times S^{\mathrm{alg}}, \mathcal{X}^{\mathrm{alg}})\to S^{\mathrm{alg}}$ is of finite type, so by Chevalley's theorem \cite[Th\'eor\`eme~1.8.4]{EGAIV} its image is constructible. Hence $C_d$ is constructible near $0_S$.
By the explicit description of the $\operatorname{Def}(\ff_k)$ in \cite[Proposition 1.5]{C06} (see also \cite[Section~II.3]{M04}), the locus $$A_d=\{t\in \operatorname{Def}(\ff_k)\,|\,\mathcal{F}_t\cong \mathbb{F}_{k-2d}\}$$ is a locally closed algebraic subset $B_d\setminus B_{d-1}$, where $B_d$ is a determinantal cone and $0\in\overline{A_d}$.
For some neighborhood $V$ of $p$, we have $C_d\cap V=\Phi^{-1}(A_d)\cap V$. As $\Phi$ is open, it follows that $p\in\overline{C_d}$. Since this holds for $p=\phi(t)$ with $t\neq 0$ arbitrarily close to $0_S$, we conclude that $0_S=\lim_{t\to 0}\phi(t)\in\overline{C_d}$.

Since $C_d$ is constructible, we choose an irreducible component $Z$ of the germ $(\overline{C_d},0_S)$ such that $C_d\cap Z$ is dense in $Z$. After shrinking, there is a possibly empty proper closed analytic subset $A\subsetneq Z$ such that $Z\setminus A\subset C_d$. Choose a holomorphic germ $h$ vanishing on $A$ but not identically on $Z$, taking $h=1$ if $A=\varnothing$. By the curve selection lemma \cite[Theorem~III.25]{M04}, there exists a holomorphic map $\gamma:(\mathbb{D},0)\to(Z,0_S)$ such that $h\circ\gamma$ is not identically zero. Hence, after shrinking $\mathbb{D}$, we have $\gamma(\mathbb{D}\setminus\{0\})\subset C_d$. Pulling back the miniversal $\qq$-Gorenstein family $\mathcal{X}^{\mathrm{univ}}\to S$ along $\gamma$ gives a $\qq$-Gorenstein smoothing $\pi':(X\subset\mathcal{X}')\to(0\in\mathbb{D})$ with $\mathcal{X}'_t\cong\ff_{k-2d}$ for every $t\neq 0$.
\end{proof}

As a corollary of Theorem \ref{thm:consc-def}, we construct a series of examples of degenerations of Hirzebruch surfaces. For the classification of degenerations of $\pp(1,1,k)$ with $k\geq 3$, see \cite[Theorem 3]{MZ25}.

\begin{corollary}\label{cor:1-1-n-cons}
Let $X$ be a normal projective klt surface admitting a $\qq$-Gorenstein deformation $(X\subset\mathcal{X})\to (0\in\mathbb{D})$ with general fiber $\mathcal{X}_t\cong \pp(1,1,k)$ where $k\geq 3$. Then the surface obtained by resolving the singularity $\frac{1}{k}(1,1)$ admits a $\qq$-Gorenstein smoothing to every Hirzebruch surface $\ff_{k-2d}$ with $k-2d\geq 0$.
\end{corollary}
\begin{proof}
By \cite[Theorem~5.7]{MZ25}, it follows that any orbifold degeneration of $\pp(1,1,k)$ contains a singular point $P=\frac{1}{k}(1,1)$, for which we obtain a birational projective morphism $Y\to X$ resolving $P$. By \cite[Lemma~4.3]{MZ25} it follows there exists a family $\mathcal{Y}\to\mathbb{D}$ that extracts a horizontal divisor  $E$ over $\mathcal{X}$ such that $E_t^2=-k$. Hence $\mathcal{Y}\to\mathbb{D}$ provides a $\qq$-Gorenstein smoothing of $Y$ to the Hirzebruch surface $\ff_k$. Indeed, by Theorem \ref{thm:consc-def} it follows that $Y$ admits a $\qq$-Gorenstein smoothing to every Hirzebruch surface $\ff_{k-2d}$ where $k-2d\geq 0$.
\end{proof}

\begin{definition}\label{def:Xkxy}
Let $k\geq 3$ and $(x,y)$ be positive integers satisfying
\begin{equation}
x^2+y^2+k=(k+2)xy.
\end{equation}
We denote by $X_{k,x,y}\to\pp(k,x^2,y^2)$ the toric morphism resolving the singularity $\frac{1}{k}(1,1)$.
\end{definition}

By Theorem~\ref{thm:consc-def} the set of indices $k$ for which $\ff_k\rightsquigarrow X$ exists is determined by its largest element, which we now define.

\begin{definition}\label{def:smoothing-index}
Let $X$ be a normal projective surface admitting a $\qq$-Gorenstein smoothing to some Hirzebruch surface.
The \emph{smoothing index} of $X$ is
\[
m(X):=\max\{k\geq 0\mid\text{there exists a $\qq$-Gorenstein smoothing }\ff_k\rightsquigarrow X\}.
\]
\end{definition}
\begin{remark}
The number $m(X)$ is well-defined. This is proved in Proposition \ref{prop:zariski-decomposition} and more generally by \cite[Theorem 3.4]{Urz16a}. The set is closed downwards mod $2$ by Theorem~\ref{thm:consc-def}. 
\end{remark}

\subsection{Fano polygon mutations}

In this subsection, we briefly review the notion of Fano polygons mutations. The following definitions are extracted from \cite{KNP17}. For extensive literature in this topic, see also \cite{ACGK12} and \cite{ACC16}.
\begin{definition}[{\cite[Section 1.1]{KNP17}}]
A Fano polygon $P$ is a convex polytope in $N_{\mathbb{R}}:=N \otimes_{\mathbb{Z}} \mathbb{R}$, where $N$ is a rank-two lattice, with primitive vertices $\mathcal{V}(P)$ in $N$ such that the origin is contained in its strict interior, $\mathbf{0} \in P^{\circ}$. A Fano polygon defines a toric surface $X_P$ given by the spanning fan of $P$.
\end{definition}
Let $P$ a Fano Polygon in $N_{\mathbb{R}}$. Let $w \in M:=\operatorname{Hom}(N, \mathbb{Z})$ be a primitive inner normal vector for an edge $E$ of $P$, so $w: N \rightarrow \mathbb{Z}$ induces a grading on $N_{\mathbb{R}}$ and $w(v)=$ $-r_E$ for all $v \in E$, where $r_E$ is the height of $E$. Define
$$
h_{\max }:=\max \{w(v) \mid v \in P\} \quad \text { and } \quad h_{\min }:=-r_E=\min \{w(v) \mid v \in P\} .
$$

We have that $h_{\max }>0$ and $h_{\min }<0$. For each $h \in \mathbb{Z}$ we define $w_h(P)$ to be the (possibly empty) convex hull of those lattice points in $P$ at height $h$,
$$
w_h(P):=\operatorname{conv}\{v \in P \cap N \mid w(v)=h\} .
$$

By definition $w_{h_{\text {min }}}(P)=E$ and $w_{h_{\text {max }}}(P)$ is either a vertex or an edge of $P$. Let $v_E \in N$ be a primitive lattice element of $N$ such that $w\left(v_E\right)=0$, and define $F:=\operatorname{conv}\left\{\mathbf{0}, v_E\right\}$, a line segment of unit length parallel to $E$ at height 0 . Notice that $v_E$, and hence $F$, is uniquely defined only up to sign.

\begin{definition}[{\cite[Section 2]{KNP17}}]\label{def:mut}
Suppose that for each negative height $h_{\min } \leqslant h<0$ there exists a (possibly empty) lattice polytope $G_h \subset N_{\mathbb{R}}$ satisfying
$$
\{v \in \mathcal{V}(P) \mid w(v)=h\} \subseteq G_h+|h| F \subseteq w_h(P),
$$
where ' + ' denotes the Minkowski sum, and we define $\varnothing+Q=\varnothing$ for any polytope $Q$. We call $F$ a factor of $P$ with respect to $w$, and define the mutation given by the primitive normal vector $w$, factor $F$, and polytopes $\left\{G_h\right\}$ to be:
$$
\operatorname{mut}_w(P, F):=\operatorname{conv}\left(\bigcup_{h=h_{\min }}^{-1} G_h \cup \bigcup_{h=0}^{h_{\max }}\left(w_h(P)+h F\right)\right) \subset N_{\mathbb{R}} .
$$
\end{definition}
\begin{remark}
$\operatorname{mut}_{n_E}{\left(P, F\right)}$ is independent of the choice for $G_h$. Alternatively if there is no possible choice of $G_h$, then the mutation with respect to $n_E$ does not exist.
\end{remark}
Let $E$ be an edge of a Fano polygon $P$ with primitive inner normal vector $n_E\in M$. $P$ not always admit a mutation with respect $n_E$. Indeed, by \cite[Lemma 1]{KNP17}, this is equivalent to $\ell(E):=|E\cap N|-1\geq r_E$. In particular, if $\operatorname{cone}(E)$ defines a c.q.s. written as $\frac{1}{kr}(1,kc-1)$ then $l(E)=k$ and $r_E=r$, so an edge whose cone is a $T$-singularity always mutates, and an edge whose cone is a Wahl singularity satisfies $l(E)=r_E$. Fano polygon mutation is reflected through $\qq$-Gorenstein deformations between toric del Pezzo surfaces.
\begin{theorem}[{\cite[Theorem 1.3]{I12} and \cite[Lemma 7]{ACC16}}]
Let $P, Q$ two Fano polygons associated to the toric varieties $X_P$ and $X_Q$. Suppose that there exists a mutation between the two polygons, then there exists a $\qq$-Gorenstein pencil $\pi: \mathcal{X} \rightarrow \pp^1$ with scheme-theoretic fibres $\pi^*(0)=X_P$ and $\pi^*(\infty)=X_Q$.
\label{Ildef}
\end{theorem}
We call the pencil of Theorem~\ref{Ildef} the \emph{Ilten pencil} of the mutation. In the context of classifying toric log del Pezzo surfaces with a given singularity basket, the following definition is used in \cite{KNP17}.
\begin{definition}[{\cite[Definition 2]{ACC16}}]
Let $P, Q \subset N_{\mathbb{R}}$ be two Fano polygons. Then $P$ and $Q$ are mutation-equivalent if there exists a finite sequence of polygons $P_0, P_1, \dots, P_n$ such that $P_0 \cong P, P_n \cong Q$ and, $P_{i+1}=\operatorname{mut}_{n_i}{\left(P_i, F_i\right)}$ for some appropriate choice of $n_i$ and $F_i$, for all $i \in\{0, \dots, n-1\}$.
\end{definition}
In \cite{KNP17} Kasprzyk, Nill and Prince classify the mutation equivalence classes for del Pezzo surfaces with at most T-singularities.
\begin{theorem}[{\cite[Theorem 6]{KNP17}}]   
 There are precisely ten mutation-equivalence classes of Fano polygons such that the toric del Pezzo surface $X_P$ has only T-singularities. They are in bijective correspondence with the ten families of smooth del Pezzo surfaces.
\end{theorem}

\section{Existence of 1-complements and toric degenerations}\label{sec:complements}
The first part of this section is devoted to proving the existence of $1$-complements in any normal projective surface $X$ with klt singularities and that admits a $\qq$-Gorenstein smoothing to $\ff_k$ for some $k\geq 0$. We use the following definition of 1-complement. 
\begin{definition}[{\cite[Section 3.2]{P15}}]
 Let $X$ be a normal variety and let $D$ be a boundary on $X$ (an effective $\qq$-divisor with coefficients $\leq 1$). Let $D=S+B$, where $S:=\lfloor D\rfloor$ (resp. $B:=\{D\}$ ) is the integral (resp. fractional) part of $D$. A \textit{$1$-complement} of $(X,D)$ is a divisor $D^{+} \in\left|-K_X\right|$ such that $(X, D^{+})$ is log canonical and $D^{+} \geq S+\lfloor 2 B\rfloor$. In particular, if $D=0$, then \textit{$1$-complement} of $K_X$ is a divisor $D^{+} \in\left|-K_X\right|$ such that $(X, D^{+})$ is $\log$ canonical.
\end{definition}

The existence of $1$-complements plays a central role in measuring how far a normal projective surface is from being toric. Indeed, toric surfaces admit torus--invariant anticanonical elements, yielding natural $1$-complements. More generally, the existence of a $1$-complement imposes strong restrictions on the singularities and the configuration of curves on the surface.

To achieve the $1$-complement existence, we first establish a couple of statements that let us understand the Mori cone $\overline{\operatorname{NE}}(X)$ of such surfaces. First, we present a result about the geometry of the exceptional $\operatorname{Exc}(\phi)$, where $\phi:Y\to X$ is the minimal resolution of $X$. For the definition of weight of $X$, see \cite[Definition 2]{Man91}.

\begin{lemma}[{\cite[Lemma 5.1]{MZ25}}]\label{lem:Hz-minres}
Let $\pi:\mathcal{X}\to\mathbb{D}$ be a singular $W$-surface where $\mathcal{X}_t\cong\ff_k$. Let $\phi:Y\to X$ be the minimal resolution of $\mathcal{X}_0$, and let $\mu:Y\to \ff_d$ be the associated minimal model of $Y$ where $d$ is the weight of $X$. Then, $\mu(Exc(\phi))$ consists of the section $\Delta_0$ with at most two fibers $F_1,F_2$.
\end{lemma}
There always exist a degenerate fiber $f_1\subset Y$ that contains exactly two irreducible curves non contracted by $\phi$, see the proof of \cite[Lemma 5.1]{MZ25}. We denote their images as $\Gamma_0$ and $\Gamma_1$. Throughout the paper we assume $\Gamma_0$ to be $K_X$-negative.

\begin{proposition}\label{prop:moriray}
Let $(X\subset\mathcal{X})\to(0\in\mathbb{D})$ be a W-surface with $\mathcal{X}_t\cong \ff_k$, $X$ singular and $\rho(X)=2$. Then, the Mori cone $\mathrm{NE}(X)$ is generated by $[\Gamma_0]$ and $[\Gamma_1]$.
\end{proposition}
\begin{proof}
Let $(X\subset\mathcal{X})\to(0\in\mathbb{D})$ be a W-surface such that $\mathcal{X}_t\cong \ff_k$, satisfying the hypothesis above. By upper semicontinuity, it follows that $h^0(X,-K_X)\geq h^0(\ff_k,-K_{\ff_k})>0$, so $|-K_X|\neq \emptyset$. Pick an effective divisor $D\in |-K_X|$. Since $K_X^2=8$, we have that $D\cdot K_X=-8$. So, there is a curve $C \subset \operatorname{Supp}(D)$ such that $C\cdot K_X<0$. 
Since $X$ is klt, the Cone Theorem \cite[Theorem 3.7]{KM98} yields a $K_X$-negative extremal curve $\Gamma\subset X$. Let $R$ be the ray spanned by $[\Gamma]$ in $\operatorname{NE}(X)$ and $f$ the associated contraction. 
If $f:X\to T$ were a Mori fiber space, then $T$ must be smooth since $f_*\mathcal{O}_X=\mathcal{O}_T$ and $X$ is normal. Thus, by \cite[Proposition 7.4]{HP10}, a degenerate fiber $f^*(q)$ contains exactly two Wahl singularities, a contradiction. Therefore $f$ must be birational.

Let $\phi:Y\to X$ denote the minimal resolution of $X$ and $\hat{\Gamma}$ the strict transform of $\Gamma$. By the discrepancy formula $K_Y=\phi^*K_X+\sum_{i}d_iE_i$, it follows that $\hat{\Gamma}\cdot K_Y\leq \Gamma\cdot K_X<0$. We compute that $\hat{\Gamma}^2=-1$, and consequently $\hat{\Gamma}$ lies inside a degenerate fiber of $Y\to\ff_d$ where $d$ is the weight of $X$, by \cite[Lemma 5.1]{MZ25}.
By the structure of the minimal resolution described in Lemma \ref{lem:Hz-minres}, then $\hat{\Gamma}$ belongs to the degenerate fiber containing the curves $\Gamma_0$ and $\Gamma_1$. Thus, $\hat{\Gamma}\in\{\hat\Gamma_0,\hat\Gamma_1\}$, we assert that $\Gamma_0=\phi_*(\hat{\Gamma})$ is a $K_X$-negative extremal curve satisfying $\Gamma_0^2<0$.

The previous argument implies that $N_1(X)_{\mathbb{R}}$ is generated by the numerical classes $[\Gamma_i]$. Indeed, if $[\Gamma_1]=\lambda[\Gamma_0]$, then it follows that $\lambda=\frac{\Gamma_0\cdot\Gamma_1}{\Gamma_0^2}<0$. This contradicts the effectiveness of $\Gamma_1$. Thus, both classes $[\Gamma_i]$ are linearly independent.

We proceed to prove that $\Gamma_1$ spans an extremal ray of $\mathrm{NE}(X)$. Suppose $K_X\cdot \Gamma_1 \geq 0$. By Lemma~\ref{lem:Hz-minres}, the strict transform $\hat{\Gamma}_1$ intersects $\operatorname{Exc}(\phi)$ in at most two components: either transversely at two components, or at exactly one component, in which case $\hat{\Gamma}_1$ is the final component of the degenerate fiber of $Y\to\mathbb{P}^1$. The classification of log canonical pairs \cite[Proposition 3.2.7]{Ale92}, implies that $(X,\Gamma_1)$ is log canonical, except when $X$ is obtained as the blow-up of a non-toric surface $X^\prime$ admitting a $\qq$-smoothing to $\mathbb{P}^2$, with the blow-up point lying generically on the unique non contractible curve in a degenerate fiber of the minimal resolution of $X^\prime$, see \cite[Lemma 7.4]{HP10}.

When $(X,\Gamma_1)$ is log canonical, from the adjunction formula we obtain
$$\Gamma_1^2+K_X\cdot\Gamma_1=\operatorname{deg}(K_{\Gamma_1})+\operatorname{deg}(\operatorname{Diff}_{\Gamma_1}(0))=-\frac{1}{n_0^2}-\frac{1}{n_1^2},$$
for some $n_i^2\geq 1$ corresponding to the singular points of $X$ that $\Gamma_1$ contains. Implying that $\Gamma_1^2<0$ and hence $\Gamma_1$ spans an extremal ray of $\operatorname{NE}(X)$. In the exceptional case, $\Gamma_1^2<0$ holds as well.

In the situation where $\Gamma_1\cdot K_X<0$, it then follows that $\operatorname{NE}(X)$ is spanned by $[\Gamma_0]$ and some other extremal curve $[\Gamma]$ that satisfies $\Gamma\cdot K_X<0$. From the same argument, $\hat{\Gamma}^2=-1$, so $\hat{\Gamma}\in\{\hat{\Gamma}_0,\hat{\Gamma}_1\}$ by Lemma~\ref{lem:Hz-minres}. Since $\hat{\Gamma}_0$ is already accounted for, we conclude $\hat{\Gamma}=\hat{\Gamma}_1$, and hence $\Gamma=\Gamma_1$.
\end{proof}
\begin{remark}
In the same vein, we observe that $K_X\cdot \Gamma_1=0$ is only possible for $W$-surfaces that admit a $\qq$-Gorenstein smoothing to $\ff_0$ or $\ff_2$. If it is the case $K_X\cdot \Gamma_1=0$, we obtain that $0=\hat{\Gamma}_1\cdot K_Y+1-\frac{a_1}{n_1}+\frac{a_2}{n_2}$ for which $P_i=\frac{1}{n_i^2}(1,n_ia_i-1)$ are the singular points that $\Gamma_1$ must intersect. Therefore $n_1=n_2$ and $a_1=a_2$, which tells us that $\Gamma_1$ determines a crepant partial resolution of the T-singularity $\frac{1}{2n_1^2}(1,2n_1a_1-1)$ or $A_1$. If the $\Gamma_1$ curve lifts, this implies that the general fiber is $\ff_2$ and the contraction of such family $\mathcal{E}$ dominating $\mathbb{D}$ gives a degeneration of $\mathbb{P}(1,1,2)$, where the $A_1$ point degenerates to $A_1$ or $\frac{1}{2n_1^2}(1,2n_1a_1-1)$. If $\Gamma_1$ does not lift, by \cite[Theorem 1]{BC94} the blow-down deformation induces a $\qq$-Gorenstein degeneration $X^\prime$ of $\ff_k$ with $\rho(X^\prime)=1$, which is only possible if $X$ smooths to $\ff_0$.
\end{remark}

\begin{theorem}\label{thm:1-compldegHz}
Let $\pi:(X\subset\mathcal{X})\to (0\in\mathbb{D})$ be a $W$-surface with general fiber $\mathcal{X}_t\cong\ff_k$. Then, $X$ admits a $1$-complement.
\end{theorem}
\begin{proof}
If $-K_X$ is ample the result follows by \cite[Theorem 4.1]{P15}. In the notation of Proposition \ref{prop:moriray}, we assume that $\Gamma_1\cdot K_X\geq 0$. If $X$ is smooth, then $X\cong \ff_{k+2d}$ for some $d\geq 0$ and the result follows trivially. Then, assume $X$ is singular. 

By Proposition \ref{prop:moriray} there exists a birational contraction $\varphi:(\Gamma_0\subset X)\to (P\in \bar{X})$ that contracts the extremal $K_X$-negative curve $\Gamma_0$. Let $\overline{\Gamma}_1=\varphi(\Gamma_1)$, following \cite[Proposition 7.3]{HP10} we observe that the pair $(\bar{X},\overline{\Gamma}_1)$ is plt. We observe that $\overline{\Gamma}_1$ contains at most $2$ singular points $P_i$ of $\bar{X}$, with different $$\operatorname{Diff}_{\bar{\Gamma}_1}(0)=\left(1-\frac{1}{m_0}\right)P_0+\left(1-\frac{1}{m_1}\right)P_1.$$

Hence, the divisor $B=\{P_0\}+\{P_1\}$ forms a 1-complement for the pair $(\overline{\Gamma}_1,\operatorname{Diff}_{\overline{\Gamma}_1}(0))$. 
Since $\rho(\bar{X})=1$, then $-(K_{\bar{X}}+\overline{\Gamma}_1)$ is
either ample or numerically trivial. By adjunction formula, we observe that $\overline{\Gamma}_1\cdot (K_{\bar{X}}+\overline{\Gamma}_1)<0$, and consequently $-(K_{\bar{X}}+\overline{\Gamma}_1)$ is an ample divisor.
By the Kawamata--Viehweg vanishing theorem the restriction map $$H^0(\bar X,-(K_{\bar {X}}+\overline\Gamma_1))\to H^0(\overline\Gamma_1,-(K_{\bar{X}}+\overline\Gamma_1)|_{\overline\Gamma_1})$$ is surjective, so the $1$-complement lifts to a divisor $D\in|-(K_{\bar{X}}+\overline\Gamma_1)|$, where $\operatorname{Diff}_{\bar{\Gamma_1}}(D)=\{P_0\}+\{P_1\}$. Following \cite[Proposition~3.7]{PS09}, we may choose $D$ so that $(\bar X,\overline\Gamma_1+D)$ is log canonical.

Now we proceed to show that $\Gamma_0$ is a log canonical center of the pair $(\bar{X},\overline{\Gamma}_1+D)$, i.e., $a_{\Gamma_0}(\bar{X},\overline{\Gamma}_1+D)=0$. 

Since $\operatorname{Diff}_{\overline{\Gamma}_1}(D) = \{P_0\}+\{P_1\}$ by construction, the point
$P = P_0$ is a log canonical center of $(\overline{\Gamma}_1,
\operatorname{Diff}_{\overline{\Gamma}_1}(D))$. By inversion of adjunction applied to the inclusion $\overline{\Gamma}_1 \subset\bar{X}$ \cite[Theorem~5.50]{KM98}, the point $P$ is a log canonical center of $(\bar{X},\overline{\Gamma}_1+D)$.
We apply the connectedness lemma \cite[Theorem~5.48]{KM98} to the contraction $\varphi\colon X\to\bar{X}$. Define the pullback pair on $X$ by
$$K_X + B_X :=\varphi^*(K_{\bar{X}}+\overline{\Gamma}_1+D),$$
so that $-(K_X+B_X)\sim 0$ is trivially $\varphi$-nef. 
The connectedness lemma then asserts that the non-klt locus $\operatorname{Nklt}(X,B_X)$ is connected over every point of $\bar{X}$, and in particular over $P$. We now identify two components of $\operatorname{Nklt}(X,B_X)$ that lie over $P$ on opposite sides of $\Gamma_0$. The boundary $B_X$  contains the strict transforms $\Gamma_1$ and $\varphi^{-1}_*D$ each with coefficient $1$, so both belong to $\operatorname{Nklt}(X,B_X)$. 
Since $\overline{\Gamma}_1$ passes through $P$, the curve $\Gamma_1$ meets $\Gamma_0$ in $X$. Since $\operatorname{Diff}_{\overline{\Gamma}_1}(D)=\{P_0\}+\{P_1\}$ and $\operatorname{Diff}_{\bar{\Gamma_1}}(0)$ has coefficient $1-\frac{1}{m_0}$ at $P_0$, the divisor $D$ also passes through $P=P_0$, so $\varphi^{-1}_*D$ meets
$\Gamma_0$ as well. Since $(\bar X,\overline{\Gamma}_1+D)$ is log canonical at $P$, the curves $\overline{\Gamma}_1$ and $D$ are not tangent there, so $\Gamma_1$ and $\varphi^{-1}_*D$ are two distinct components of $\operatorname{Nklt}(X,B_X)$ meeting $\Gamma_0$ at two distinct points of $\varphi^{-1}(P)$. By connectedness of $\operatorname{Nklt}(X,B_X)$ over $P$, the curve $\Gamma_0$ itself must belong to $\operatorname{Nklt}(X,B_X)$. Otherwise $\Gamma_1$ and $\varphi^{-1}_*D$ would be disconnected from each other in the fiber over $P$. Hence
$a_{\Gamma_0}(\bar{X},\overline{\Gamma}_1+D)\leq 0$, and since the pair is log
canonical equality holds, $a_{\Gamma_0}(\bar{X},\overline{\Gamma}_1+D)=0$.
Consequently, we observe that $\varphi^{-1}_*D+\Gamma_1+\Gamma_0\sim-K_{X}$. By the crepant invariance of discrepancies, \cite[Lemma 2.30]{KM98}, we conclude that if $B=\varphi^{-1}_*D+\Gamma_1+\Gamma_0$  then $(X,B)$ is a log canonical pair and consequently $B$ is a 1-complement for $X$.
\end{proof}

\begin{proposition}\label{prop:toricity}
Let $\pi:(X\subset\mathcal{X})\to(0\in\mathbb{D})$ be a $W$-surface with general fiber $\mathcal{X}_t\cong \ff_k$. Then, $|\operatorname{Sing}(X)|\leq 4$, and if the equality holds, then $X$ must be toric.    
\end{proposition}
\begin{proof}
By Theorem \ref{thm:1-compldegHz}, it follows that there exists a 1-complement $B\in |-K_X|$ for $X$, then for the boundary $B$, we obtain the log Calabi-Yau pair $(X,B)$. First assume that $\rho(X)=2$. Since $X$ contains only Wahl singularities $P_i$ and consequently non-Gorenstein points, by the classification of the log canonical pairs \cite[Proposition 3.2.7]{Ale92} then the 1-complement must pass over them. Moreover, $B$ must be nodal near each $P_i$, i.e., $$(P_i,B)=\left(\frac{1}{n_i^2}(1,n_ia_i-1),uv=0\right)$$ and consequently, $B$ must have two analytic branches near $P_i$. We observe that $|B|\leq 4$, otherwise $c(X,B)<0$ that contradicts \cite[Theorem 1.2]{BMSZ18}. Also, since $B$ is reduced, by the adjunction formula we get  $$2p_a(B)-2=(B+K_X)\cdot B=0.$$
We then observe that $B$ must be a cycle of smooth rational curves. Indeed, from $B\sim -K_X$ and $h^1(X,\oo_X)=0$, the sequence $$0\to\oo_X(K_X)\to\oo_X\to\oo_B\to 0$$
gives $h^0(\oo_B)=1$, so $B$ is connected. Hence, the number of nodes of $B$ equals $|B|$ and consequently $|\operatorname{Sing}(X)| \leq |B|\leq 4$. In the situation where we have that $|\operatorname{Sing}(X)|=4$ it follows that $c(X,B)=0$. By \cite[Theorem 1.2]{BMSZ18}, this implies that $X$ must be a toric surface. The case $\rho(X)=1$ follows directly from \cite[Theorem 1.1]{HP10}.
\end{proof}

\subsection{Toric Degenerations}
In this section we use the constructions of \cite{UZ25}, to show that for any normal projective surface $X$ with Wahl singularities admitting a $\qq$-Gorenstein smoothing to a Hirzebruch surface $\ff_k$, it further degenerates to a toric surface. This is a $\qq$-Gorenstein family $\mathcal{X}\to\mathbb{D}$, where $\mathcal{X}_t\cong X$ and $\mathcal{X}_0$ is a toric surface with Wahl singularities. 

\begin{definition}[{\cite[Definition 4.16]{UZ25}}]
Let $Y$ be a nonsingular surface. Let $[e_1,\ldots,e_r]$ be a Wahl chain in $Y$ with corresponding exceptional curves $E_1,\ldots,E_r$. Let $X$ be its contraction. Assume that there is a $(-1)$-curve $\Gamma$ in $Y$ intersecting $E_i$ with $i>1$ transversally at one point. A \textit{left slide} of $\Gamma$ is a surface $Y'$ together with a Wahl chain $[f_1,\dots,f_{r'}]$, the Wahl chain $[e_1,\ldots,e_r]$, and a $(-1)$-curve $\Gamma'$ in between, so that they form the chain $$[f_1,\ldots,f_{r'},1,e_1,\ldots, e_r]$$ which contracts to $[e_1,\ldots,e_{i-1},e_i-1,e_{i+1},\ldots,e_r]$ and is the contraction of $\Gamma$ in $Y$. We denote the contraction of both Wahl chains in $Y'$ by $X'$, and the image of $\Gamma'$ in $X'$ and of $\Gamma$ in $X$ again by $\Gamma',\Gamma$. Similarly, we define the \textit{right slide} of $\Gamma$ with the notation $\Gamma''$, $X''$. 
\label{slide}    
\end{definition}

Slides always exist and are unique, see \cite[Lemma 4.18]{UZ25}. We now observe that the construction of slides provide a $\qq$-Gorenstein degeneration of a non log-canonical pairs $(X,\Gamma)$ into a plt pair $(X^\prime,\Gamma^\prime)$, as presented in the next lemma.

\begin{lemma}[{\cite[Lemma 5.11]{UZ25}}]
Let $X$ be a normal projective surface with at most Wahl singularities such that $H^2(X,T_X)=0$. Suppose $X$ has a singularity $P$ together with a curve $\Gamma$ defining a slide. Let $X'$ be the left (or right) slide of $\Gamma$. Then there is a $\qq$-Gorenstein degeneration of $X\rightsquigarrow X'$ which smooths the new singularity appearing in $X'$. 
\label{defslides}
\end{lemma}

Now, we gather the existence of $1$-complements and the deformations coming from performing slides, to study complexity of pairs.

\begin{theorem}\label{thm:slide-complement}
Let $X$ be a normal projective surface with Wahl singularities admitting a $1$-complement $B \in |-K_X|$. Suppose that $X$ has a Wahl singularity $P$ together with a curve $\Gamma$ defining a slide, $H^1(X,\mathcal{O}_X)=H^2(X,\mathcal{O}_X)=0$ and that $H^2(X,T_X)=0$. Let $X'$ be the left (resp., right) slide of $\Gamma$. Then $X'$ admits a $1$-complement $B' \in |-K_{X'}|$ satisfying
$$c(X',B^\prime)=c(X,B)-1.$$
\end{theorem}

\begin{proof}
By Theorem~\ref{defslides}, the slide produces a $\qq$-Gorenstein 
deformation $\pi\colon\mathcal{X}\to\mathbb{D}$ with $\mathcal{X}_t \cong X$ 
for $t \neq 0$ and $\mathcal{X}_0 = X'$, which smooths the new Wahl singularity $P'$ appearing in $X'$ and keeps all original singularities unchanged. 
Since $H^2(X', T_{X'}) = 0$ (by the adding/deleting criteria, as in the proof of Lemma~\ref{defslides}), the local $\qq$-Gorenstein smoothing of $P'$ globalizes to a deformation of $X'$ preserving the remaining singularities. The deformation is isotrivial outside the 
singularity germs involved in the slide.

\medskip
We now proceed to construct a 1-complement in $X'$. Since $P=\frac{1}{n^2}(1,na-1)$ is non-Gorenstein, then $B$ passes through $P$. Since $(X,B)$ is log canonical with $B$ reduced, by \cite[Theorem~4.15(1)]{KM98} the pair  $(P\in X, B)$ has exactly two analytic branches at $P$, forming the toric boundary $(P,B)\cong (\frac{1}{n^2}(1,na-1), uv=0)$. In the minimal resolution $\phi\colon Y \to X$, the strict transforms $\tilde{B}_1$ and  $\tilde{B}_2$ meet the two ends $E_1$ and $E_r$ of the Wahl chain  $[e_1,\ldots,e_r]$, respectively, with $\Gamma$ the $(-1)$-curve meeting $E_i$ ($i>1$).

Both $Y$ and the minimal resolution $Y'$ of $X'$ share a common blow-down  $\bar{Y}$, obtained by contracting $\Gamma$ in $Y$ or by contracting  $[f_1,\ldots,f_{r'},1]$ in $Y'$, whose chain is  $[e_1,\ldots,e_{i-1},e_i-1,e_{i+1},\ldots,e_r]$. The passage from  $\bar{Y}$ to $Y'$ involves a sequence of blow-ups that create $\Gamma'$  and $F_1,\ldots,F_{r'}$. Tracing through the contraction algorithm, these  blow-ups occur at points on the chain between $E_1$ and $E_i$; the curve  $E_1$ in $\bar{Y}$ becomes $F_1$, while  the blow-ups create the remaining curves $F_2,\ldots,F_{r'}$ and $\Gamma'$  to its right. No blow-up occurs near $E_r$.  Consequently, in $Y'$ the full chain is 
\[ \hat{B}_1 - F_1 - \cdots - F_{r'} - \Gamma' - E_1' - \cdots - E_r'  - \hat{B}_2, \] 
where $\tilde{B}_1$ meets $F_1$ (the left end of $[f_1,\ldots,f_{r'}]$)  and $\tilde{B}_2$ meets $E_r'$ (the right end of $[e_1,\ldots,e_r]$).  
We then observe that $B_1':= \phi'(\tilde{B}_1)$ is a toric branch at $P'$ and $B_2' := \phi'(\tilde{B}_2)$ is a toric branch at $P$.

\medskip
Since away from $P$, the deformation is isotrivial by construction of slides, then all other components of $B$ are unchanged. The node between $B_1$ and $B_2$ at $P$ in the  cycle $B$ is replaced by the chain of curves $B_1' - \Gamma' - B_2'$ with nodes at $P'$ and $P$. By defining $$B^\prime:=B_1^\prime+B_2^\prime+\Gamma^\prime+\sum_{B_i\in\operatorname{Supp}(B)\setminus\{B_1,B_2\}}B_i\in \operatorname{WDiv(X^\prime)},$$ we observe that $|B'|=|B|+1$. Note that $(X',B')$ is log canonical, indeed at each of the points $P,P^\prime$, the germs $(P\in X^\prime, B^\prime)$ and $(P^\prime\in  X^\prime,B^\prime)$ we have the structure of \cite[Theorem~4.15(1)]{KM98}.

\medskip
We show that $B'+K_{X'}\sim 0$. At each singularity $Q$ of $X'$, the boundary $B'$ forms the toric boundary of the cyclic quotient singularity, which satisfies $D_1+D_2+K\sim 0$ locally. On the smooth locus $(X')^\circ$, the isotriviality of the deformation gives $(B'+K_{X'})|_{(X')^\circ} \sim 0$. Thus $B'+K_{X'}$ is Cartier and linearly trivial on $(X')^\circ$ and in a neighborhood of each isolated singularity. Since $X'$ is normal, it follows that $B'+K_{X'}\sim 0$ globally. Therefore, $B^\prime$ forms a 1-complement for $X^\prime$.

Now, we proceed to prove the complexity statement. Since the degeneration only adds Wahl singularities, by the Noether formula we have that $\rho(X)=\rho(X^\prime)$ and since $H^1(X,\mathcal{O}_X)=H^2(X,\mathcal{O}_X)=0$, we observe that $\dim\Cl_\qq(X)=\dim\Cl_\qq(X^\prime)$. Therefore, $$c(X',B') = 2+\dim\operatorname{Cl}_\qq(X')-|B'| 
= 2+\dim\operatorname{Cl}_\qq(X)-(|B|+1) = c(X,B)-1.$$
\end{proof}

In the following definition we define the unique class of surfaces arising as degenerations of $\ff_1$ that does not admit a toric degeneration by successive slides.

\begin{definition}\label{def:dagger-type}
Let $\bar{X}$ be a degeneration of $\pp^2$ with quotient singularities and let $\phi:\overline{Y}\to \bar{X}$ be its minimal resolution. Let $\bar{X}^0$ denote the smooth locus of $\bar{X}$.  

We say that a surface $X$ is of type $(\dagger)$ if $X$ is isomorphic to some $\operatorname{Bl}_p \bar{X}$ for some point $p\in \bar{X}^0\cap \Delta$, where $\Delta=\Gamma_0+\Gamma_1$ and $\Gamma_i$ denotes the image of a non-contractible curve contained in a degenerate fiber of $\overline{Y}\to\pp^1$. 
\end{definition}

\begin{lemma}\label{lem:Hz-toric-deg}
Let $\pi:(X\subset\mathcal{X})\to(0\in\mathbb{D})$ be a $W$-surface with general fiber $\mathcal{X}_t\cong \ff_k$. Additionally, assume that $X$ is not of $(\dagger)$ type. If no curve in the minimal resolution of $X$ defines a slide, then $X$ is a toric surface. 
\end{lemma}
\begin{proof}
Assume $X$ is singular and $\rho(X)=2$. Let $\phi:Y\to X$ be the minimal resolution of $X$. By Lemma~\ref{lem:Hz-minres}, the $\mathbb{P}^1$-fibration $Y\to\mathbb{P}^1$ has a degenerate $f_1$. Again, let $\Gamma_0,\Gamma_1$ denote the images under $\phi$ of the non-contractible curves in $f_1$, and let $\Gamma = \phi_*(f_2)$ be the image of a possible degenerate fiber $f_2$. Assume without loss of generality that $\Gamma_0\cap\Gamma\neq\emptyset$. By hypothesis of $X$ not defining a slide nor $(\dagger)$, the curves $\Gamma_0,\Gamma_1$ and $\Gamma_0,\Gamma$ intersect torically. We consider a 1-complement $B\in|-K_X|$ by Theorem \ref{thm:1-compldegHz}. Since for each singular point $P\in X$, the pair $(P\in X,B)$ is lc, by \cite[Theorem 4.15]{KM98} and Lemma \ref{lem:Hz-minres}, it follows that the curves $\Gamma_0,\Gamma_1,\Gamma$ are contained in $\operatorname{Supp}(B)$. As $p_a(B)=1$, it follows that $B$ is a cycle of rational curves implying that $|B|=4$. Thus $c(X,B)=0$ and consequently $(X,B)$ is toric by \cite[Theorem 1.2]{BMSZ18}.
\end{proof}

As a consequence of Theorem \ref{thm:slide-complement}  we are able to show the main result of the present section. Following the same algorithm introduced in \cite[Section 5]{UZ25}, we start with a $\qq$-Gorenstein degeneration $\ff_k\rightsquigarrow X$. Then, we construct a toric surface $\widehat{X}$ via successive slides over $X$.

\begin{theorem}\label{thm:toric-degeneration}
Let $\pi:(X\subset\mathcal{X})\to(0\in\mathbb{D})$ be a $W$-surface with general fiber $\mathcal{X}_t\cong \ff_k$ and $X$ is not type $(\dagger)$. Then after a sequence of slides, there exists a $\qq$-Gorenstein family $\pi^\prime:\mathcal{X}^\prime\to\mathbb{D}$ such that the general fiber $\mathcal{X}_t^\prime$ is isomorphic to $X$ and the central fiber $\mathcal{X}'_0$ is a toric surface.
\end{theorem}
\begin{proof}
We proceed by induction on $c(X,B)$, where $B \in |-K_X|$ is a $1$-complement of $X$. If no curve in the minimal resolution $Y \to X$ defines a slide, then by 
Lemma~\ref{lem:Hz-toric-deg}, $c(X,B)=0$ and $X$ is toric, so the trivial family suffices.

Otherwise, there exists a $(-1)$-curve $\Gamma$ in $Y$ defining a slide. Let $X_1$ be the corresponding slide. By Theorem~\ref{thm:slide-complement}, $X_1$ admits a $1$-complement $B_1 \in |-K_{X_1}|$ with $c(X_1,B_1)=c(X,B)-1$. Consequently, by Proposition~\ref{prop:no-obstruction}, $H^2(X_1,T_{X_1})=0$. Hence, by Lemma \ref{defslides} we
repeat this process, obtaining a sequence 
$X=X_0, X_1, \ldots, X_r$ where each $X_{j+1}$ is a slide of $X_j$ and $c(X_r,B_r)=0$. Hence, $\widehat{X}:=X_r$ is a toric surface.

The surface $\widehat{X}$ has singularities 
$P_1,\ldots,P_s,P_1^\prime,\ldots,P_{r'}'$, where $P_1,\ldots,P_s$ are the original singularities of $X$ and $P_1',\ldots,P_{r'}'$ are the new Wahl singularities created by the $r$ successive slides. It follows that $r'\leq r$. Let $e=s+r'$ denote the total number of singularities of $\widehat{X}$. By Proposition \ref{prop:no-obstruction} it follows that $H^2(\widehat{X},T_{\widehat{X}})=0$ and consequently $\hat{X}$ does not have local-to-global obstructions to deform.
Consider the local-to-global exact sequence for $\qq$-Gorenstein deformations of $\widehat{X}$:
\[
0 \to H^1(\widehat{X},T_{\widehat{X}}) \to T^1_{\qq\text{G},\widehat{X}} \to \bigoplus_{i=1}^{e} T^1_{\qq\text{G},Q_i}
\to H^2(\widehat{X},T_{\widehat{X}}) \to T^2_{\qq\text{G},\widehat{X}} \to \bigoplus_{i=1}^{e} T^2_{\qq\text{G},Q_i} \to 0,
\]
where $Q_1,\ldots,Q_e$ denotes the full list of singularities. Since each $Q_i$ is a Wahl singularity, $T^2_{\qq\text{G},Q_i}=0$. Combined with $H^2(\widehat{X},T_{\widehat{X}})=0$, the exact sequence yields $T^2_{\qq\text{G},\widehat{X}}=0$ and a short exact sequence
\[
0 \to H^1(\widehat{X},T_{\widehat{X}}) \to T^1_{\qq\text{G},\widehat{X}} \to \bigoplus_{i=1}^{e} T^1_{\qq\text{G},Q_i} \to 0.
\]
Let $m:=h^1(\widehat{X},T_{\widehat{X}})$. The vanishing $T^2_{\qq\text{G},\widehat{X}}=0$ implies that the miniversal $\qq$-Gorenstein deformation space $(S,0):=(\operatorname{Def}^{\qq\text{G}}(\widehat{X}),0)$ is smooth of dimension $e+m$. Let $\mathcal{X}^{\mathrm{univ}}\to S$ denote the miniversal family. By \cite[Lemma~1]{Man91}, every collection of local $\qq$-Gorenstein deformations of the singularities globalizes, and so the natural map
\[
\Phi\colon(S,0)\longrightarrow \prod_{i=1}^{s} \operatorname{Def}^{\qq\text{G}}(P_i) \times \prod_{j=1}^{r'} \operatorname{Def}^{\qq\text{G}}(P_j')\cong(\cc^e,0)
\]
is a submersion. Choose local coordinates $(t_1,\ldots,t_s,u_1,\ldots,u_{r'})$ on the target so that $t_i$ corresponds to the deformation parameter of $P_i$ and $u_j$ to that of $P_j'$.

Since $\widehat{X}$ is projective, by \cite[Theorem~1.6]{Art69} the formal miniversal deformation algebraizes. As in Theorem \ref{thm:consc-def}, there exist an algebraic scheme $S^{\mathrm{alg}}$ of finite type over $\cc$ and a projective flat family $\mathcal{X}^{\mathrm{alg}}\to S^{\mathrm{alg}}$ whose formal completion at a closed point $0\in S^{\mathrm{alg}}$ recovers the formal germ of $\mathcal{X}^{\mathrm{univ}}\to S$. We identify $(S,0)\cong ((S^\mathrm{alg})^{\mathrm{an}},0)$ and define 
\[
V:=\bigl\{p\in S^{\mathrm{alg}}\,|\, t_1(p)=\cdots=t_s(p)=0,\; u_1(p)\cdots u_{r'}(p)\neq 0,\; (\mathcal{X}^{\mathrm{alg}})_p\cong X\bigr\}.
\]

The Isom-scheme $\mathrm{Isom}_{S^{\mathrm{alg}}}(X\times S^{\mathrm{alg}},\mathcal{X}^{\mathrm{alg}})\to S^{\mathrm{alg}}$ is of finite type, so by Chevalley's theorem \cite[Th\'eor\`eme~1.8.4]{EGAIV} its image is constructible.
Since the zero set $\{t_1=\cdots=t_s,u_1=\cdots u_{r'}\neq0\}$ is constructible, then $V$ is constructible.

We claim that $0\in\overline{V}$, where $\overline{V}$ denotes the Zariski closure. Consider the last slide $X_{r-1}\rightsquigarrow\widehat{X}$. By Lemma~\ref{defslides}, there exists a $\qq$-Gorenstein family $\mathcal{Y}_r\to\mathbb{D}$ with $(\mathcal{Y}_r)_0\cong\widehat{X}$ and $(\mathcal{Y}_r)_t\cong X_{r-1}$ for $t\neq 0$. By versality, this family is defined by a holomorphic map $\gamma_r\colon(\mathbb{D},0)\to(S,0)$. For any sequence $t_n\to 0$ with $t_n\neq 0$, the points $p_n^{(r)}:=\gamma_r(t_n)\to 0$ satisfy $(\mathcal{X}^{\mathrm{univ}})_{p_n^{(r)}}\cong X_{r-1}$.
Now fix $n$. Since versality is an open condition for proper flat morphisms, the family $\mathcal{X}^{\mathrm{univ}}\to S$ is versal at $p_n^{(r)}$ after shrinking $S$ if necessary. Applying Lemma~\ref{defslides} to the penultimate slide $X_{r-2}\rightsquigarrow X_{r-1}$, we obtain a $\qq$-Gorenstein deformation of $X_{r-1}$ with general fiber $X_{r-2}$. By versality at $p_n^{(r)}$, this deformation is classified by a map through $p_n^{(r)}$, yielding a point $p_n^{(r-1)}$ arbitrarily close to $p_n^{(r)}$ with $(\mathcal{X}^{\mathrm{univ}})_{p_n^{(r-1)}}\cong X_{r-2}$.
Repeating for each remaining slide, we obtain $p_n:=p_n^{(1)}\in S$ satisfying $p_n\to 0$ and $(\mathcal{X}^{\mathrm{univ}})_{p_n}\cong X$. Each step smooths exactly one singularity $P_j'$ while preserving $P_1,\ldots,P_s$ and the previously smoothed singularities, hence $p_n\in V$ for $n\gg 0$.
Since the $p_n$ are closed points of $S^{\mathrm{alg}}$ accumulating analytically at $0$, every Zariski-open neighborhood of $0$ contains infinitely many of them, and so $0\in\overline{V}$.

Since $V$ is constructible and $0\in\overline V$, the same curve-selection argument as in Theorem~\ref{thm:consc-def}, using \cite[Theorem~III.25]{M04}, gives a holomorphic map $\gamma:(\mathbb D,0)\to(S,0)$ such that $\gamma(\mathbb D\setminus\{0\})\subset V$.
Pulling back the the miniversal $\qq$-Gorenstein family $\mathcal{X}^{\text{univ}}\to S$ along $\gamma$ gives a $\mathbb Q$-Gorenstein deformation $$\mathcal{X}'\to\mathbb{D}$$ with $\mathcal{X}_0'\cong\widehat{X}$ and $\mathcal{X}'_t\cong X$ for all $t\neq 0$.

\end{proof}

As a consequence of Theorem \ref{thm:slide-complement} and \ref{thm:toric-degeneration}, we can provide a bridge between the number of sliding curves and the pair complexity of surfaces $X$ admitting a $\qq$-Gorenstein smoothing to $\ff_k$.

\begin{corollary}
Let $\pi:(X\subset\mathcal{X})\to(0\in\mathbb{D})$ be a $W$-surface with general fiber $\mathcal{X}_t\cong\ff_k$, and assume that $X$ is not of type $(\dagger)$. Let $s$ denote the number of curves on $X$ defining a slide, and let $B\in|-K_X|$ be a $1$-complement. Then, $$c(X,B)=s.$$
In particular, if $k\leq 1$, then $c(X,B)\in\{0,1,2,3\}$. While for $k\geq 2$ it follows that $c(X,B)\in\{0,1,2\}$. Furthermore, every value in the corresponding range occurs.
\end{corollary}
\begin{proof}
The identity $c(X,B)=s$ follows directly from Theorem \ref{thm:slide-complement} together with the sliding process described in Theorem \ref{thm:toric-degeneration}.

We now determine the possible values of $c(X,B)$. First, suppose that $k\leq 1$. Consider surfaces $X$ with a unique Wahl singularity and with $-K_X$ ample. By \cite[Remark 5.3 and Theorem 5.8]{UZ25}, there exist many such examples whose minimal resolution $Y\to X$ admits two degenerate fibers. In this case, there are exactly three curves on $X$ inducing slides, and hence $c(X,B)=3$. Performing successive slides, we obtain examples realizing all smaller values, so that $c(X,B)\in\{0,1,2,3\}$.

Next, suppose that $k\geq 2$. By Proposition \ref{prop:moriray}, we have the extremal curve $\Gamma_1$ such that $\Gamma_1\cdot K_X\geq 0$. It cannot induce a sliding since $(X,\Gamma_1)$ is lc. Thus, at most two curves on $X$ induce slides, and hence $c(X,B)\leq 2$. This shows that $c(X,B)\in\{0,1,2\}$.
\end{proof}

The surfaces constructed in Definition \ref{def:dagger-type} also admit toric degenerations. Rather than relying on the sliding process, the method exploits the extension of $\mathbb{G}_m$ equivariant actions on families $\mathcal{X}\to\mathbb{D}^*$. Let $X$ be a surface of $(\dagger)$ type. For simplicity let us assume that it is constructed by a blow-up of a toric surface $X\to\bar{X}=\pp(a^2,b^2,c^2)$. For the non-toric case the same method as for $\bar{X}$ in Proposition \ref{prop:E0-two-sings} applies. By performing subsequent slides as in Theorem \ref{thm:toric-degeneration}, a toric degeneration follows.

Let $(a,b<c)$ be a Markov triple, we consider homogeneous coordinates $[x:y:z]$ in $\bar{X}$ of degrees $(a^2,b^2,c^2)$, and let $P_z=[0:0:1]$ denote the singular point $\frac{1}{c^2}(1,cw_c-1)$.
Consider the point $p=[1:0:1]\in\{y=0\}$ and the $\mathbb{G}_m$-action $\lambda(t)\cdot[x:y:z]=[tx:y:z]$, so that the orbit of $p$ is $\{[t:0:1]\,|\,t\in\mathbb{C}^*\}$, which is dense in $\{y=0\}$.
Since $\gcd(a,c)=1$, the invariants of $\mu_{c^2}$ restricted to $\{y=0\}$ are generated by $u=x^{c^2}/z^{a^2}$, which is therefore a coordinate of $\{y=0\}\cap U_z$ centered at $P_z$, and $u(\lambda(t)\cdot p)=t^{c^2}$.
Accordingly we set $$S=\{([t:0:1],\,t^{c^2})\,|\,t\in\mathbb{D}\}\subset \pp(a^2,b^2,c^2)\times \mathbb{D},$$ which is a well-defined section of the projection, cut out by $\{y=0,\ u=t\}$, and we let $\mathcal{X}=\operatorname{Bl}_S(\bar{X}\times \mathbb{D})$.
On the orbifold cover $\mathbb{C}^2_{x,y}\to U_z$ with action $\frac{1}{c^2}(a^2,b^2)$, the reduced preimage of $S$ is then the smooth invariant curve $\tilde{S}=\{y=0,\ t=x^{c^2}\}$, which becomes the axis $\{y=t'=0\}$ after the change of coordinates $t^\prime=t-x^{c^2}$.
Blowing up $\tilde S$ we obtain coordinates $[Y:T']$ on the exceptional divisor satisfying $yT'=t'Y$, and the two charts $y=t'w$ and $t'=ys$ of the blow-up.
With this change of coordinates in mind, a local computation in these two charts produces the following result about the singularities of $\mathcal{X}_0$ lying on the exceptional curve $E_0$.
\begin{proposition}\label{prop:E0-two-sings}
The limit curve $E_0\subset\mathcal{X}_0$ passes torically through exactly two singular points, of types
\[
\tfrac1{c^2}(1,\,cw_c-1)\qquad\text{and}\qquad \tfrac1{c^4}(1,\,c^2\bar a-1),\quad \bar a\equiv(cw_c-1)^{-1}\!\!\pmod{c^2}.
\]
\end{proposition}

\begin{proof}
Write $r=c^2$, and recall $\gcd(a,c)=\gcd(b,c)=1$, so $a^2,b^2$ are units modulo $r$.
On the cover $\mathbb{C}^3_{x,y,t}$ of $U_z\times\mathbb{D}$ with action $\tfrac1r(a^2,b^2,0)$, the function $t'=t-x^r$ is invariant, so $(x,y,t')$ have weights $\tfrac1r(a^2,b^2,0)$ and $\widetilde{S}=\{y=t'=0\}$ is a smooth invariant axis.

Let $E$ be the exceptional divisor of its blow-up and $E_0=(E\cap\mathcal{X}_0)_{\mathrm{red}}$, computed in the two charts.

In the chart $y=t'w$, of weights $\tfrac1r(a^2,0,b^2)$ with $E=\{t'=0\}$, the fiber $t=0$ eliminates $t'=-x^r$ to the smooth surface $\mathbb{C}^2_{x,w}$ of weights $\tfrac1r(a^2,b^2)$, with $E_0=\{x=0\}$.
This action is free off the origin, which is thus the only singular point of $E_0$ here, of type $\tfrac1{c^2}(1,b^2(a^2)^{-1})=\tfrac1{c^2}(1,cw_c-1)$.

In the chart $t'=ys$, of weights $\tfrac1r(a^2,b^2,-b^2)$ with $E=\{y=0\}$, the fiber $t=0$ gives the $A_{r-1}$-singularity $\{ys=-x^r\}$, with $E_0=\{x=y=0\}$.
Uniformizing it by $(X,Y)\mapsto(x,y,s)=(XY,X^r,-Y^r)$, with deck group $g_0=\tfrac1r(1,-1)$ and fixing $\zeta_{r^2}^{\,r}=\zeta_r$, the action lifts to $\hat g=(\eta,\theta)$ solving $\eta\theta=\zeta_r^{a^2},\ \eta^r=\zeta_r^{b^2}$, e.g. $\hat g=(\zeta_{r^2}^{\,b^2},\zeta_{r^2}^{\,a^2r-b^2})$.
Then $\hat g^{\,r}=g_0^{\,b^2}$ with $\gcd(b,c)=1$ gives $\langle g_0,\hat g\rangle=\langle\hat g\rangle\cong\mathbb{Z}/r^2$, acting freely off the origin since $\gcd(b^2,r^2)=\gcd(a^2r-b^2,r^2)=1$, so the quotient is $\tfrac1{r^2}(b^2,a^2r-b^2)$.
\end{proof}
Both Theorem \ref{thm:toric-degeneration} and Proposition \ref{prop:E0-two-sings} establish Theorem \ref{ithm:toric}.

\section{Marked Degenerations}\label{sec:degen-k3}
Throughout this section we consider $W$-surfaces $\pi:(X\subset \mathcal{X})\to (0\in\mathbb{D})$ with general fiber $\mathcal{X}_t\cong \ff_k$. We write $\phi:Y\to X$ for the minimal resolution and $\mu:Y\to \pp^1$ the ruling of Lemma \ref{lem:Hz-minres}. Following Proposition \ref{prop:moriray}, let $f_1$ be the degenerate fiber containing the strict transforms $\hat{\Gamma}_0$, $\hat{\Gamma}_1$ of the generators of $\operatorname{NE}(X)$. Our goal is classifying such $X$ when $\Gamma_1\cdot K_X>0$. The cases where $-K_X$ is nef are devoted to the next section and were also studied in a similar manner in \cite{UZ25}. We proceed defining marked surfaces. 

\begin{definition}\label{def:hirzebruch-wahl-chain}
Let $k \geq 1$. A Wahl chain $[e_1,\ldots,e_r]$ is 
called \emph{Hirzebruch of index~$k$} of type~$\mathrm{(I)}$ or~$\mathrm{(II)}$ 
if one of the following holds:
\begin{itemize}
    \item[$\mathrm{(I)}$] There exists a Wahl chain $[f_1,\dots,f_t]$, possibly empty, such that
    \[
    [f_1,\dots,f_t,k,e_1,\ldots,e_{r-1}]
    \quad \text{or} \quad
    [e_2,\ldots,e_r,k,f_1,\dots,f_t]
    \]
   admits a zero continued fraction of weight~$0$. In the first case the \emph{central mark} is $e_r$, in the second it is $e_1$.
    
    \item[$\mathrm{(II)}$] There exist $i \in \{2,\ldots,r-1\}$ and a Wahl chain 
    $[f_1,\dots,f_t]$, possibly empty, such that
    \[
    [f_1,\dots,f_t,\,k,\,e_1,\ldots,e_{i-1}]
    \quad \text{and} \quad
    [e_{i+1},\ldots,e_r],
    \]
    or symmetrically
    \[
    [e_1,\ldots,e_{i-1}]
    \quad \text{and} \quad
    [e_{i+1},\ldots,e_r,\,k,\,f_1,\dots,f_t],
    \]
    admit zero continued fractions of weight~$0$. In both cases, the \emph{central mark} is $e_i$.
\end{itemize}
Its \emph{marking} is the data of the zero continued fraction(s) involved in its type. When every auxiliary Wahl chain is empty, we say that $[e_1, \ldots, e_r]$ is a \emph{primitive}; otherwise it is \emph{non-primitive}.
\end{definition}
\begin{remark}\label{rem:Hz-index}
For a Hirzebruch Wahl chain of index $k\geq 2$, a similar approach to \cite[Proposition 4.1]{UZ24} shows the following values for the central mark $e_i$:
\begin{center}
\begin{tabular}{c|c|c}
 & Primitive & Non-primitive \\
\hline
Type~I & $k+4$ & $k+5$ \\
Type~II & $k+7$ & $k+8$
\end{tabular}
\end{center}
\end{remark}
We now illustrate the definition with examples.

\begin{example}[Primitive, from $X_{k,x,y}$]\label{ex:primitive-Hz}
Let $\frac{n^2}{na-1}=[e_1,\ldots,e_r]$ be a Wahl chain arising from the surface $X_{k,x,y}$. By \cite[Proposition~5.6]{MZ25}, the degenerate fiber $f_1$ in the minimal resolution contains a subchain $[2,\ldots,2]$ of length~$k-1$.
Then $[k-1,1,2,\ldots,2]=0$ (where $[2,\ldots,2]$ has length $k-2$) is a zero continued fraction of weight~$0$, so the chain is primitive Hirzebruch of index~$k-1$.
\end{example}

\begin{example}[Primitive, not coming from $X_{k,x,y}$]\label{ex:primitive-Hz-2} 
Consider the Wahl chain with $n=246$ and $a=169$:
$$\frac{n^2}{na-1}=[2, 2, 7, 2, 2, 2, 2, 2, 2, 5, \mathbf{9}, 2, 2, 2, 2, 4].$$
This is primitive of Type (II) with $k=2$. Indeed, one verifies that $[\mathbf{2},2, 2, 7, 1, 2, 2, 2, 2, 2, 5]$ is a zero continued fraction of weight $0$. Similarly, $[2, 2, 2, 1, 4]=0$
\end{example}

\begin{example}[Non-primitive Hirzebruch Wahl chain of Type I]\label{ex:nonprimitive-Hz}
Consider the Wahl chain with $n=1799$, $a=246$:
$$\frac{n^2}{na-1}=[\mathbf{8}, 2, 2, 7, 2, 2, 2, 2, 2, 2, 5, 9, 2, 2, 2, 2, 5, 2, 2, 2, 2, 2, 2].$$
This is non-primitive of type~$\mathrm{(I)}$ with $k=3$.
The auxiliary Wahl chain has $n_1=19$, $a_1=13$ and continued fraction $[2,2,9,2,2,2,2,4]$.
Indeed, one verifies that $$[2, 2, 7, 2, 2, 2, 2, 2, 2, 5, 9, 2, 2, 2, 1, 5, 2, 2, 2, 2, 2, 2,\,\mathbf{3},\,2, 2, 9, 2, 2, 2, 2, 4]=0$$
\end{example}

\begin{example}[Non-primitive Hirzebruch Wahl chain of Type II]\label{ex:nonprimitive-Hz-II}
Consider the Wahl chain with $n=956$, $a=823$:
$$\frac{n^2}{na-1}=[2, 2, 2, 2, 2, 2, 7, 2, 2, \mathbf{11}, 2, 2, 2, 2, 2, 2, 5, 2, 2, 2, 2, 8].$$
This is non-primitive of type~$\mathrm{(II)}$ with $k=3$.
The auxiliary Wahl chain has $n_1=246$, $a_1=77$ and continued fraction $[4, 2, 2, 2, 2, 9, 5, 2, 2, 2, 2, 2, 2, 7, 2, 2]$.
Indeed, one verifies that $$[4, 2, 2, 2, 2, 9, 5, 2, 2, 2, 2, 2, 1, 7, 2, 2,\mathbf{3},2, 2, 2, 2, 2, 2, 7, 2, 2]=0$$
and 
$$[2, 2, 2, 2, 2, 2, 5, 1, 2, 2, 2, 8]=0.$$
\end{example}
We now describe the geometric procedure to construct a normal projective surface associated to a Hirzebruch Wahl chain of index~$k$. The algorithm is analogous to \cite[Section~5]{UZ25}.
Let $e_i$ be its central mark, with $i\in\{2,\ldots,r\}$, where $i=r$ is type~$\mathrm{(I)}$ and $i<r$ is type~$\mathrm{(II)}$, and let $[f_1,\ldots,f_t]$ be the auxiliary Wahl chain of the marking, with $t=0$ permitted.
When $t=0$ we set $(n_1,a_1)=(1,0)$ and read $\left[{1 \choose 0}\right]$ as the empty chain, that is, as a smooth germ.
\begin{itemize}
    \item Consider the Hirzebruch surface $\ff_{e_i}$, with its ruling $\pi \colon \ff_{e_i} \to \pp^1$ and negative section $\Delta_0$ satisfying $\Delta_0^2 = -e_i$.
    \item Choose a single fiber $F_1'$ of $\pi$ if $i=r$, and two distinct fibers $F_1'$ and $F_2'$ if $i<r$.
    Blow up along each chosen fiber, away from $\Delta_0$, to produce the chains
    \[
    [f_1,\dots,f_t,\,k,\,e_1,\ldots,e_{i-1}]
    \quad\text{and}\quad
    [e_{i+1},\ldots,e_r]
    \]
    corresponding to the zero continued fractions, starting at the intersections with $\Delta_0$; the second chain is empty when $i=r$.
    We write $\Gamma_1$ for the curve corresponding to the entry $k$, so that $\hat\Gamma_1^{\,2}=-k$.
    \item Since the zero continued fractions have weight~$0$, a unique blow-up at a general point of a single marked component is required over each chosen fiber.
    We write $\Gamma_0$, and $\Gamma$ when $i<r$, for the $(-1)$-curves so produced over $F_1'$ and over $F_2'$.
    \item Contract all the Wahl chains so produced.
    The resulting surface carries the Wahl singularities $\frac{1}{n^2}(1,na-1)$ from $[e_1,\ldots,e_r]$ and $\frac{1}{n_1^2}(1,n_1a_1-1)$ from $[f_1,\ldots,f_t]$, the latter being a smooth point exactly when $t=0$, and contains a partial resolution of the cyclic quotient singularity obtained by contracting $\Gamma_1$, with data
    \[
    \left[{n_1 \choose a_1}\right]-(k)-\left[{n \choose a}\right],
    \]
    up to ordering.
\end{itemize}

\begin{definition}\label{def:marked-surface}
The surface constructed above is the \emph{marked surface} associated to the given marking, denoted $X_*$ when $t=0$, $X_{**}$ when $t\geq 1$.
\end{definition}
\begin{proposition}\label{prop:markedunobstructed}
Let $X_\bullet$ be a marked surface associated to a Hirzebruch Wahl chain.
Then $H^2(X_\bullet,T_{X_\bullet})=0$, so $X_\bullet$ is unobstructed.
In particular $X_\bullet$ admits a $\qq$-Gorenstein smoothing, and the general fiber of every $\qq$-Gorenstein smoothing of $X_\bullet$ is a Hirzebruch surface.
\end{proposition}

\begin{proof}
Following the same method as in Theorem \ref{thm:1-compldegHz} it follows that $X_{\bullet}$ admits a $1$-complement and consequently $h^0(X_{\bullet},-K_{\bullet})\geq 1$. Furthermore, since $X_{\bullet}$ is a rational surface with at most rational singularities, then by Proposition \ref{prop:no-obstruction} it follows that $H^2(X_\bullet,T_{X_\bullet})=0$ and consequently $X_{\bullet}$ does not have local-to-global obstructions to deform. Counting, the construction performs $t+i$ blow-ups over $F_1'$ and $r-i$ over $F_2'$, so $N=r+t$, while $\phi$ contracts the $r+t$ curves of the two Wahl chains. Hence $\rho(X_\bullet)=\rho(\ff_{e_i})=2$, and $K_{X_\bullet}^2=8$ by Proposition~\ref{prop:Noether}.

Let $\mathcal X_t$ be the general fiber of any $\qq$-Gorenstein smoothing of $X_\bullet$.
Since $X_\bullet$ is rational with klt singularities, $P_m(X_\bullet)=0$ for every $m\geq1$, so upper semicontinuity of $t\mapsto h^0(mK_{\mathcal X_t})$ (see \cite[Lemma~1.3]{B85}) gives $P_m(\mathcal X_t)=0$, implying that $\mathcal X_t$ is a rational surface. Since $K_{\mathcal X_t}^2=K_{X_\bullet}^2=8$, it follows that $\mathcal{X}_t\cong \ff_k$ for some $k\geq 0$.
\end{proof}
We now proceed to prove that the marked surface characterize the central fibers in $W$-surfaces $\pi:(X\subset\mathcal{X})\to(0\in\mathbb{D})$ where $\mathcal{X}_t\cong \ff_k$ and $-K_X$ is not nef.

\begin{theorem}\label{thm:Hz-deformations}
Let $X$ be a singular normal projective surface with only Wahl singularities admitting a $\qq$-Gorenstein smoothing to a Hirzebruch surface.
Assume that $-K_X$ is not nef and that $X$ is not of type~$(\dagger)$.
Then there is a $\qq$-Gorenstein deformation $X_{\bullet}\rightsquigarrow X$, where $X_{\bullet}$ is a marked surface associated to a Hirzebruch--Wahl chain.
\end{theorem}

\begin{proof}
By Proposition \ref{prop:moriray} it follows that $\operatorname{NE}(X)$ is generated by $[\Gamma_0]$ and $[\Gamma_1]$, the images of the non-contractible curves lying in the degenerate fiber $f_1$ of $\mu$. By assumption, we have that $K_X\cdot\Gamma_0<0$ and $K_X\cdot\Gamma_1>0$.

First, let us assume that $\Gamma_1 \not\subset X^0$, the smooth locus of $X$. Let $P_0$ be the singularity whose exceptional chain contains $\hat{\Delta}_0$. As in Proposition \ref{prop:moriray}, the pair $(X,\Gamma_1)$ is log-canonical. Moreover, we have the birational morphism $\psi\colon X\to\bar{X}$ contracting  $\Gamma_1$, with $\rho(\bar{X})=1$ and $P:=\psi(\Gamma_1)=\frac{1}{\Delta}(1,\Omega)$. By Theorem~\ref{thm:1-compldegHz} there is a $1$-complement $B\in|-K_X|$, so $\psi_*B\in|-K_{\bar{X}}|$. Consequently, since $H^2(\bar{X},T_{\bar{X}})=0$, there exists a $\qq$-Gorenstein deformation $\bar{\mathcal{X}}_t\rightsquigarrow \bar{X}$ smoothing every singularity in $\operatorname{Sing}(\bar{X})\setminus\{P\}$ and is trivial near $P$. We observe that $a_{\Gamma_1}(\bar{\mathcal{X}}_t)<1$ if and only if $\Gamma_1\cdot K_X>0$. Hence by \cite[Lemma 3.1]{MZ25}, we extract $\Gamma_1$ in families over $P$ yielding a $\qq$-Gorenstein deformation $\mathcal{X}\to\mathbb{D}$ with $\mathcal{X}_0=X$; set $X^\prime:=\mathcal{X}_t$ for $t\neq 0$.

It remains to identify $X^\prime$ with a marked surface. First, we construct a 1-complement to the nearby fibers of $\mathcal{X}\to \mathbb{D}$.
Since $\pi^*\{0\}=\bar{\mathcal{X}}_0$, the threefold $\bar{\mathcal{X}}$ has no singularities in codimension two along $\bar{\mathcal{X}}_0$, so $(K_{\bar{\mathcal{X}}}+\bar{\mathcal{X}}_0)|_{\bar{\mathcal{X}}_0}\sim K_{\bar{X}}$ and $\operatorname{Diff}_{\bar{\mathcal{X}}_0}(0)=0$ by \cite[Lemma~2.1]{EV85}, and $(\bar{\mathcal{X}},\bar{\mathcal{X}}_0)$ is plt by inversion of adjunction \cite[Theorem~5.50]{KM98}.
Write $L:=-(K_{\bar{\mathcal{X}}}+\bar{\mathcal{X}}_0)$, which is ample over $\mathbb{D}$ after shrinking $\mathbb{D}$, since $L|_{\bar{\mathcal{X}}_0}=-K_{\bar{X}}$ is ample.
As $\bar{\mathcal{X}}_0\sim_{\mathbb{D}}0$ we have $L-\bar{\mathcal{X}}_0\sim_{\mathbb{D}}K_{\bar{\mathcal{X}}}-2K_{\bar{\mathcal{X}}}$ with $-2K_{\bar{\mathcal{X}}}$ ample over $\mathbb{D}$, so Kawamata--Viehweg vanishing gives $R^1\pi_*\mathcal{O}_{\bar{\mathcal{X}}}(L-\bar{\mathcal{X}}_0)=0$.
The restriction sequence associated to $\bar{\mathcal{X}}_0$ therefore makes
\[
H^0(\bar{\mathcal{X}},L)\longrightarrow H^0(\bar{X},-K_{\bar{X}})
\]
surjective, so $\psi_*B$ lifts to a $1$-complement $\overline{\mathcal{B}}$ of the pair $(\bar{\mathcal{X}},\bar{\mathcal{X}}_0)$ dominating $\mathbb{D}$; see also \cite[Proposition~3.7]{PS09}.
For $t\neq 0$ the fiber $\bar{\mathcal{X}}_t$ is not contained in the non-klt locus of $(\bar{\mathcal{X}},\overline{\mathcal{B}})$, and $\bar{\mathcal{X}}_t\sim\bar{\mathcal{X}}_0$, so adjunction along $\bar{\mathcal{X}}_t$ yields $K_{\bar{\mathcal{X}}_t}+\overline{\mathcal{B}}_t\sim 0$ with $(\bar{\mathcal{X}}_t,\overline{\mathcal{B}}_t)$ log canonical.

Let $\psi_t\colon X^\prime\to\bar{\mathcal{X}}_t$ be the extraction of $\Gamma_1^\prime:=\mathcal{D}_t$, where $\mathcal{D}$ denotes the exceptional divisor of $\mathcal{X}\to\bar{\mathcal{X}}$, and consider the log pullback
\[
K_{X^\prime}+B^\prime=\psi_t^*\bigl(K_{\bar{\mathcal{X}}_t}+\overline{\mathcal{B}}_t\bigr)\sim 0 .
\]
The coefficient of $\Gamma_1^\prime$ in $B^\prime$ is $1-a_{\Gamma_1^\prime}(\bar{\mathcal{X}}_t,\overline{\mathcal{B}}_t)$, which is at most $1$ because the pair downstairs is log canonical, and non-negative because $a_{\Gamma_1^\prime}(\bar{\mathcal{X}}_t)<1$ as observed above.
Hence $B^\prime\in|-K_{X^\prime}|$ is effective and $(X^\prime,B^\prime)$ is log canonical by \cite[Lemma~2.30]{KM98}, so $B^\prime$ is a $1$-complement of $X^\prime$. In particular $h^0(X^\prime,-K_{X^\prime})\geq 1$, so by the same argument of Proposition \ref{prop:markedunobstructed}, it follows that $X'$ smooths to a Hirzebruch surface.

Let $\phi^\prime\colon Y^\prime\to X^\prime$ be the minimal resolution, with fibration $\mu^\prime\colon Y^\prime\to\pp^1$.
By Lemma~\ref{lem:Hz-minres} applied to $X^\prime$, the fiber $f_1^\prime$ has exactly two non-exceptional components, namely the strict transform of $\Gamma_0$ and the exceptional curve $\hat{\Gamma}_1^\prime$ of $\psi_t$, and $f_2^\prime$, if present, has exactly one, a $(-1)$-curve $\Gamma^\prime$.
Since the singularity supported in $f_2$ has been smoothed, after a slide as in Lemma~\ref{defslides} the curve $\Gamma^\prime$ meets the Wahl chain of $P_0$ at an interior component transversally at one point, and the same applies over $f_1$ to $\Gamma_0$.
Hence $X^\prime$ is of marked type, with $X^\prime=X_{**}$ when $\Gamma_1^\prime$ meets two Wahl chains and $X^\prime=X_*$ when there is a unique singular point.

\medskip
For $\Gamma_1\subset X^0$, we have that $X$ belongs to the class of Definition \ref{def:Xkxy}. The first step of the deformation presented in \cite[Example 6.10]{UZ25} gives the deformation $X_*\rightsquigarrow X$ where $X_*$ is the marked surface associated to the Wahl chains of Example \ref{ex:primitive-Hz}.
\end{proof}
\begin{remark}\label{rem:not-dagger}
For the surfaces of $(\dagger)$ type, the statement does not hold. By construction, such a surface $X$ satisfies that $\Gamma_0\subset (X)^0$ with $\Gamma_0^2=-1$. Hence, $H^1(\Gamma_0, \operatorname{N}_{\Gamma_0|X})=0$ and consequently $\Gamma_0$ is unobstructed. This contradicts the possibility of having a marked surface $X_{\bullet}$ in the general fiber, since every deformation of $X$ contains a $(-1)$-curve.
\end{remark}

\subsection{Bounds on the smoothing index}\label{ss:sbounds}
As mentioned in the introduction, for an explicit classification of the orbifold degenerations of Hirzebruch surfaces, for a given $\ff_k\rightsquigarrow X$ it is desirable to effectively compute the smoothing index $m(X)$. In the following proposition we compute a non-optimal bound for $m(X)$. For instance, for the surfaces of Definition \ref{def:Xkxy}, it follows that $m(X_{k,x,y})=k$ 

\begin{proposition}\label{prop:zariski-decomposition}
Let $X$ be a normal projective surface with klt singularities admitting a $\qq$-Gorenstein smoothing to a Hirzebruch surface. Assume $\rho(X)=2$, and let $c:=-\widehat{\Gamma}_1^{2}$.
Suppose that $-K_X$ is not nef. Then the Zariski decomposition of $-K_X$ is
\[
-K_X=P+\frac{K_X\cdot\Gamma_1}{-\Gamma_1^{2}}\,\Gamma_1,
\]
where $P$ is nef, $P\cdot\Gamma_1=0$ and
\[
P^{2}=8+\frac{(K_X\cdot\Gamma_1)^{2}}{-\Gamma_1^{2}}.
\]
In particular $\{P,\Gamma_1\}$ is an orthogonal basis of $N^1(X)_\qq$.
Moreover, if $m(X)\geq 2$, then
\begin{equation}\label{eq:bound}
(m(X)-2)^{2}<c\,m(X).
\end{equation}
Consequently $m(X)\leq c+3$, and $m(X)\leq 3$ when $c=1$.
\end{proposition}

\begin{proof}
Let $\pi:(X\subset \mathcal{X})\to (0\in\mathbb{D})$ be the corresponding $W$-surface, and write $k:=m(X)$, so that $\mathcal{X}_t\cong \ff_k$ for $t\neq 0$.
Upper semicontinuity applied to $\mathcal{O}_{\mathcal{X}}(-mK_{\mathcal{X}/\mathbb{D}})$ for $m$ divisible by the Gorenstein index of $\mathcal{X}$ gives $h^0(X,-mK_X)\geq h^0(\ff_k,-mK_{\ff_k})\geq 4m^2+O(m)$, so $-K_X$ is big.
By \cite[Corollary 7.5]{S84} there is a unique decomposition
\[
-K_X=P+N
\]
with $P$ nef, $N\geq 0$ with negative definite support, and $P$ orthogonal to every component of $N$.
If $\operatorname{Supp}N$ had two components, their classes would span $N^1(X)_\qq$ because $\rho(X)=2$, so the intersection form would be negative definite on all of $N^1(X)_\qq$, contradicting $(-K_X)^2=8>0$.
Moreover $N\neq 0$, since otherwise $-K_X$ would be nef while $-K_X\cdot\Gamma_1<0$.
Writing $N=n\,C$ with $C$ irreducible and $n>0$, orthogonality gives $-K_X\cdot C=nC^2<0$, hence $K_X\cdot C>0$ and therefore $C=\Gamma_1$ by uniqueness of the $K$-positive curve.
Then $P\cdot\Gamma_1=0$ forces $-K_X\cdot\Gamma_1=n\Gamma_1^{2}$, that is $n=(K_X\cdot\Gamma_1)/(-\Gamma_1^{2})$, and
\[
P^{2}=P\cdot(-K_X)=(-K_X)^{2}+n\,(K_X\cdot\Gamma_1)=8+\frac{(K_X\cdot\Gamma_1)^{2}}{-\Gamma_1^{2}}.
\]
Since $P^{2}>\Gamma_1^{2}$ and $P\cdot\Gamma_1=0$, the classes $P$ and $\Gamma_1$ are linearly independent, hence form an orthogonal basis of $N^1(X)_\qq$.

We now prove the numerical bound.
Set
\[
u:=K_X\cdot\Gamma_1>0,\qquad v:=-\Gamma_1^{2}>0,\qquad \theta:=\frac{u^{2}}{v},
\]
so that the decomposition above reads $-K_X=P+\frac{u}{v}\Gamma_1$ with $P^{2}=8+\theta$.

Now, we assume $k\geq 2$. Let $D\in\Cl(X)$ be the specialization in $X$ of $\Delta_0$, an effective divisor with $K_X\cdot D=k-2$ and $D^{2}=-k$.
Write $D=\lambda P+\mu\Gamma_1$ in the orthogonal basis above.
Since $P$ is nef and $D$ is effective, $\lambda P^{2}=P\cdot D\geq 0$, hence $\lambda\geq 0$.
From $K_X=-P-\frac{u}{v}\Gamma_1$ we obtain $k-2=-\lambda P^{2}+\mu u$ and $k=\mu^{2}v-\lambda^{2}P^{2}$, and eliminating $\mu$ gives
\[
\theta k-(k-2)^{2}=\lambda P^{2}\bigl(2(k-2)+8\lambda\bigr)\geq 0,
\]
so that $(k-2)^{2}\leq\theta k$.

It remains to prove $\theta<c$.
Let $n_0,n_1$ be the indices of the singular points of $X$ lying on $\Gamma_1$, with $n_i=1$ at a smooth point.
By adjunction, as in Proposition \ref{prop:moriray}, we have
\[
u=v-\frac{1}{n_0^{2}}-\frac{1}{n_1^{2}}<v .
\]
Writing $\phi^*\Gamma_1=\widehat{\Gamma}_1+\sum_i\alpha_iE_i$, the coefficients $\alpha_i$ are determined by $\phi^*\Gamma_1\cdot E_i=0$; since the intersection matrix $(E_i\cdot E_j)$ is negative definite and $\widehat{\Gamma}_1\cdot E_i\geq 0$, all $\alpha_i\geq 0$.
Intersecting with $\widehat{\Gamma}_1$ therefore gives
\[
v=-\phi^*\Gamma_1\cdot\widehat{\Gamma}_1=c-\sum_i\alpha_i\,(E_i\cdot\widehat{\Gamma}_1)\leq c .
\]
Hence $\theta=u^{2}/v<uv/v=u<v\leq c$, and combining, $(k-2)^{2}<ck$.

This inequality is equivalent to $k^{2}-(c+4)k+4<0$, which forces $k <\tfrac12\bigl(c+4+\sqrt{c^{2}+8c}\bigr)$. Since $\sqrt{c^{2}+8c}<c+4$, the larger root is less than $c+4$. Consequently, $k\leq c+3$.
For $c=1$ the larger root equals $4$, so $k\leq 3$.
\end{proof}
For the case of marked surfaces of Definition \ref{def:marked-surface} the bound $m(X_\bullet)$ can be further sharpened. However, it is still not optimal.

\begin{corollary}\label{cor:marked-bound}
Let $X_\bullet$ be a marked surface of index $k\geq 2$.
Then $m(X_\bullet)\ \leq\ k+i$,
where $i=1$ if $X_\bullet$ is primitive and $i=2$ if $X_\bullet$ is non-primitive.
\end{corollary}

\begin{proof}
Write $m=m(X_\bullet)$ and we follow the same notation of Proposition \ref{prop:zariski-decomposition}. First, assume $X_\bullet=X_*$.
By construction $\Gamma_1$ meets the unique singularity at an end of its chain and $X_*$ is smooth along $\Gamma_1$ elsewhere, so adjunction gives $v>u+1$ and the log-discrepancy of the component met by $\widehat{\Gamma}_1$ lies in $(0,1)$, so $u<k-1$.
Assume for contradiction that $m\geq k+2$, then the expression $(m-2)^{2}\leq\theta m$ in Proposition \ref{prop:zariski-decomposition} along with the monotonicity of the function $f(x)=\frac{(x-2)^2}{x}$ on the interval $[2,\infty)$ provide the inequality $$(u+1)k^{2}<(k+2)u^2.$$ 
By defining $h(x)=(k+2)x^2-k^2x-k^2$, we observe that $h(u)>0$. However, $h$ is convex with $h(0)=-k^{2}<0$ and $h(k-1)=2-3k<0$, so $h<0$ on $[0,k-1]$, a contradiction.
Therefore, we conclude $m\leq k+1$.

Assume $X_\bullet=X_{**}$.
By Remark \ref{rem:Hz-index} the central mark is $k+5$ in type $\mathrm{(I)}$ and $k+8$ in type $\mathrm{(II)}$.
We run the $K$-MMP of $(X_{**}\subset\mathcal{X})\to(0\in\mathbb{D})$. Assume that it only consists of flips, otherwise, $m=1$. As in \cite[\S 2]{Urz16a}, this process terminates on a smooth deformation $\ff_m\rightsquigarrow\ff_{k+4}$, respectively $\ff_m\rightsquigarrow\ff_{k+6}$.
In either case $m\equiv k\pmod 2$ gives $m\leq k+2$.
\end{proof}
The inequality $m(X_\bullet)\leq m(X)$ established by Corollary \ref{cor:marked-bound} for the $\qq$-Gorenstein deformation $X_\bullet\rightsquigarrow X$ belongs to a more general framework, where the smoothing index function is upper semi-continuous.

\begin{theorem}\label{thm:usc-smoothing-index}
Let $X$ be a normal projective surface with only Wahl singularities admitting a $\qq$-Gorenstein smoothing to some Hirzebruch surface, and let $\pi:(X\subset\mathcal{X})\to(0\in\mathbb{D})$ be a $\qq$-Gorenstein deformation.
Then $m(\mathcal{X}_t)\leq m(X)$ for $t\neq 0$.
\end{theorem}
\begin{proof}
Fix $t\neq 0$ and write $X^\prime:=\mathcal{X}_t$.
By Proposition \ref{prop:Noether} it then follows that $\rho(X)=2$. By upper semicontinuity applied to the given smoothing of $X$ we have $h^0(X,-K_X)>0$, so $H^2(X,T_X)=0$ by Proposition~\ref{prop:no-obstruction}.
As in the proof of Theorem~\ref{thm:toric-degeneration}, the miniversal space $S:=\operatorname{Def}^{\qq\mathrm{G}}(X)$ is then smooth with coordinates $(t_1,\ldots,t_r,s_1,\ldots,s_h)$ for which $\{t_i=0\}$ is the locus where $P_i\in \operatorname{Sing}(X)$ remain in the deformation.
Let $\phi:(\mathbb{D},0)\to(S,0)$ be a morphism of germs inducing $\pi$, i.e., $\mathcal{X}\cong\mathcal{X}^{\mathrm{univ}}\times_S\mathbb{D}$, and put $q:=\phi(u)$ for $0<|u|\ll 1$, so that $X^\prime\cong\mathcal{X}^{\mathrm{univ}}_q$ and after reordering, $t_i(q)=0$ exactly for $i\leq l$.
The germ $\ell:(\cc,0)\to(S,q)$, $$\ell(v):=(v,\ldots,v,t_{l+1}(q),\ldots,t_r(q),s(q)),$$ pulls the miniversal family back to a $\qq$-Gorenstein smoothing of $X^\prime$, whose general fiber is some $\ff_{k^\prime}$ by Proposition~\ref{prop:markedunobstructed}. So $X'$ is unobstructed and consequently $m(X')$ is well defined.

Set $k^\prime:=m(X^\prime)$ and fix a $\qq$-Gorenstein smoothing $\ff_{k^\prime}\rightsquigarrow X^\prime$.
By openness of versality \cite[Theorem~1.5]{A74} we may shrink $S$ so that $\mathcal{X}^{\mathrm{univ}}\to S$ is versal at every point.
Since $\mathcal{X}^{\mathrm{univ}}_{\phi(u)}\cong X^\prime$ for every $u\neq 0$ small, versality at $\phi(u)$ produces $\psi_u:(\mathbb{D},0)\to(S,\phi(u))$ with $\mathcal{X}^{\mathrm{univ}}_{\psi_u(\tau)}\cong\ff_{k^\prime}$ for $\tau\neq 0$.
Choosing $u_n\to 0$ and then $\tau_n\neq 0$ with $\lvert\psi_{u_n}(\tau_n)-\phi(u_n)\rvert<\lvert u_n\rvert$, the points $p_n:=\psi_{u_n}(\tau_n)$ satisfy $p_n\to 0$ and $\mathcal{X}^{\mathrm{univ}}_{p_n}\cong\ff_{k^\prime}$.
Hence $0$ lies in the closure of $C:=\{p\in S:\mathcal{X}^{\mathrm{univ}}_p\cong\ff_{k^\prime}\}$, which is constructible by the argument of Theorem~\ref{thm:consc-def}, and the formal curve selection lemma gives a $\qq$-Gorenstein smoothing $\ff_{k^\prime}\rightsquigarrow X$.
Therefore $k^\prime\leq m(X)$.
\end{proof}

\section{Degenerations of $\ff_k$ with $k\leq1$}\label{sec:degen-k01}
Throughout this section we consider $W$-surfaces $\pi:(X\subset\mathcal{X})\to(0\in\mathbb{D})$ with general fiber $\mathcal{X}_t\cong\ff_k$ for $k\leq 1$, so that the general fiber is a del Pezzo surface. We assume that $\rho(X)=2$, unless otherwise stated. The cases with $\rho(X)=1$ are studied separately, see subsection~\ref{subsec:F0}.
As in the previous section we write $\phi:Y\to X$ for the minimal resolution and $\mu:Y\to\pp^1$ for the ruling of Lemma~\ref{lem:Hz-minres}, and we let $f_1$ be the degenerate fiber containing the strict transforms $\hat{\Gamma}_0,\hat{\Gamma}_1$ of the generators of $\operatorname{NE}(X)$ given by Proposition~\ref{prop:moriray}. If $-K_X$ is not ample, we assume that $\Gamma_1\cdot K_X\geq 0$.
We denote by $P_{\ff_k}$ the Fano polygon whose spanning fan defines $\ff_k$, and by $[P_{\ff_k}]$ its mutation-equivalence class in the sense of Definition~\ref{def:mut}.
In contrast with Section~\ref{sec:degen-k3}, here both the case $K_X\cdot\Gamma_1>0$ and the case $-K_X$ nef occur, and the first subsection shows that the former reduces to the latter.
\subsection{Antiflips and the $K$-MMP for degenerations of $\ff_0$ and $\ff_1$}\label{subsec:antiflips}
This subsection proves that a degeneration $X$ of $\ff_k$, with $k\leq 1$, carrying the extremal curve $\Gamma_1$ with $K_X\cdot\Gamma_1>0$ contracting to a cyclic quotient singularity admits a $\qq$-Gorenstein smoothing with prescribed axial multiplicities at the two Wahl singularities lying on $\Gamma_1$. This yields a terminal $K$-antiflip at $\Gamma_1$, by \cite[Corollary~3.23]{HTU17}.
By running a determined MMP on the total space, we show it terminates in a model whose central fiber $X_N$ satisfies that $-K_{X_N}$ is nef (Corollary~\ref{cor:antiflip-termination}); for $k=1$ this condition may be strengthened to $-K_{X_N}$ ample.

\begin{theorem}\label{thm:admissible-smoothing}
Let $X$ be a projective surface with $\rho(X)=2$ admitting a $\qq$-Gorenstein smoothing to $\ff_k$  with $k=0,1$.
Let $\Gamma_1\subset X$ be an extremal curve with $K_X\cdot\Gamma_1> 0$, contracting to a cyclic quotient singularity $\frac{1}{\Delta}(1,\Omega)$, and let $\delta$ be the invariant of the associated extremal neighborhood.
Then for every pair $(\alpha_0,\alpha_1)\in\nn_{>0}^2$ satisfying $\alpha_0^2-\delta\alpha_0\alpha_1+\alpha_1^2>0$, there exists a $\qq$-Gorenstein smoothing $(X\subset\mathcal{X})\to(0\in\mathbb{D})$ such that:
\begin{enumerate}
\item the general fiber is $\mathcal{X}_t\cong\ff_{l_0}$, where $l_0=k\bmod 2$;
\item A toric structure $B\subset X$ near $\Gamma_1$ lifts to the total space, i.e., there exists $D\in |-K_{\mathcal{X}}|$ such that $D|_X=B$; and
\item the axial multiplicities at the two Wahl singularities contained in $\Gamma_1$ are $\alpha_0$ and $\alpha_1$.
\end{enumerate}
In particular, the smoothing admits a terminal $K$-antiflip at $\Gamma_1$.
\end{theorem}

\begin{proof}
Let $r=|\operatorname{Sing}(X)|$, $h=h^1(X,T_X)$, and $N=r+h$. Since $H^2(X,T_X)=0$, the $\qq$-Gorenstein local-to-global exact sequence~\eqref{eq:loc-glob-QG} then yields a smooth miniversal base $S\cong(\cc^N,0)$ with coordinates $(t_1,\ldots,t_r,u_1,\ldots,u_h)$, where $t_i$ is the versal $\qq$-Gorenstein deformation parameter of the $i$-th Wahl singularity. After possibly shrinking $S$, the smooth locus is $S^{\mathrm{sm}}=\{t_1\cdots t_r\neq 0\}$, every fiber $\mathcal{X}_s$ for $s\in S^{\mathrm{sm}}$ is a smooth rational surface with $b_2=2$, hence isomorphic to $\ff_{m(s)}$ for a unique $m(s)\geq 0$ with $m(s)\equiv k\pmod{2}$. Set $C_{l_0}:=\{s\in S^{\mathrm{sm}}:m(s)=l_0\}$ and $Z:=S^{\mathrm{sm}}\setminus C_{l_0}$. We divide the proof into four steps.

\medskip
\emph{Step 1: $C_{l_0}$ is open and nonempty.}
By Theorem~\ref{thm:consc-def}, $X$ admits a $\qq$-Gorenstein smoothing to $\ff_{l_0}$, so $C_{l_0}\neq\emptyset$.
Since $l_0\in\{0,1\}$, the surface $\ff_{l_0}$ satisfies $h^1(\ff_{l_0},T_{\ff_{l_0}})=0$ and is therefore rigid, hence every fiber near a point of $C_{l_0}$ in $S^{\mathrm{sm}}$ is isomorphic to $\ff_{l_0}$.
This implies that $C_{l_0}$ is open and $Z$ is closed in $S^{\mathrm{sm}}$.

\medskip
\emph{Step 2: construction of the smoothing.}
Fix $(\alpha_0,\alpha_1)\in\nn_{>0}^2$ with $\alpha_0^2-\delta\alpha_0\alpha_1+\alpha_1^2>0$.
Label the Wahl singularities so that $P_1,P_2$ are contained in $\Gamma_1$, and define the morphism $\gamma\colon\mathbb{A}^1\to S$ by
\[
\gamma(s):=\bigl(\underbrace{s^{\alpha_0}\,c_1}_{t_1},\;\underbrace{s^{\alpha_1}\,c_2}_{t_2},\,\;\ldots,\;\underbrace{s\,c_r}_{t_r},\;\underbrace{s\,c_{r+1}}_{u_1},\;\ldots,\;\underbrace{s\,c_N}_{u_h}\bigr),
\]
where the constants $c_j\in\cc^*$ are to be determined.
By construction $\gamma(0)=0$, $\gamma(s)\in S^{\mathrm{sm}}$ for $s\neq 0$, and the axial multiplicities at $P_1,P_2$ are $\operatorname{ord}_s(t_1\circ\gamma)=\alpha_0$ and $\operatorname{ord}_s(t_2\circ\gamma)=\alpha_1$. We now choose the constants $c_j$.
By Theorem~\ref{thm:consc-def}, there exists $p\in C_{l_0}\subset S^{\mathrm{sm}}$.
Fix any $s_0\in\cc^*$ and define
\[
c_1:=\frac{p_1}{s_0^{\alpha_0}},\quad c_2:=\frac{p_2}{s_0^{\alpha_1}},\quad c_j:=\frac{p_j}{s_0}\;\;\text{for }j\geq 3,
\]
where $p=(p_1,\ldots,p_N)$.
Then $\gamma(s_0)=p\in C_{l_0}$ by construction, and  $c_1,\dots,c_r$ are nonzero since $p\in S^{\mathrm{sm}}$.

\medskip
\emph{Step 3: finiteness of the non-$\ff_{l_0}$ locus.}
Pull back the universal family via $\gamma$ to obtain a flat projective family $\mathcal{Y}:=\gamma^*\mathcal{X}^{\mathrm{univ}}\to\mathbb{A}^1$.
On $\mathcal{Y}$, the relative anticanonical divisor $-K_{\mathcal{Y}/\mathbb{A}^1}$ is $\qq$-Cartier. For $s\neq 0$, the fiber $\mathcal{Y}_s\cong \ff_{m(s)}$ satisfies that $m(s)=l_0$ if and only if $-K_{\mathcal{Y}_s}$ is ample, since $m(s)\equiv k \pmod{2}$ and $-K_{\ff_m}$ is ample exactly for $m\leq 1$. Thus
\[
\gamma^{-1}(C_{l_0})=\{s\in \mathbb{A}^1\setminus\{0\} : -K_{\mathcal{Y}_s}\ \text{is ample}\}.
\]
By openness of ampleness \cite[Proposition~1.41]{KM98}, the set $A:=\{s\in\mathbb{A}^1 : -K_{\mathcal{Y}_s}\ \text{is ample}\}$ is Zariski open. Since $s_0\in A$ and $0\notin A$, its complement is a nonempty proper closed subset of $\mathbb{A}^1$, hence finite. Therefore $\gamma^{-1}(Z)=(\mathbb{A}^1\setminus A)\setminus\{0\}$ is finite, and $\mathcal{Y}_s\cong \ff_{l_0}$ for all $0<|s|<\epsilon$ for some $\epsilon>0$. Restricting to $\mathbb{D}_\epsilon$, the family $\mathcal{Y}|_{\mathbb{D}_\epsilon}\to\mathbb{D}_\epsilon$ is a $\qq$-Gorenstein smoothing of $X$ with general fiber $\ff_{l_0}$ and axial multiplicities $(\alpha_0,\alpha_1)$.

\medskip
\emph{Step 4: the boundary condition.}
Let $U$ be the germ of $X$ along $\Gamma_1$ with boundary germ $B_U$.
By \cite[Lemma~3.2]{TU22} there are no local-to-global obstructions to $\qq$-Gorenstein deformations of $U$ nor of $(U,B_U)$, i.e. both maps
\[
\operatorname{Def}^{\qq\mathrm{G}}(U,B_U)\to\prod_i\operatorname{Def}^{\qq\mathrm{G}}_{P_i\in(U,B_U)},
\qquad
\operatorname{Def}^{\qq\mathrm{G}}(U)\to\prod_i\operatorname{Def}^{\qq\mathrm{G}}_{P_i}
\]
are smooth. It follows that $H^1(U,T_U)=H^1(\Gamma_1,N_{\Gamma_1/U})=0$, since $\deg N_{\Gamma_1/U}=-1$. So the second map above is injective, hence an isomorphism.
Therefore $\operatorname{Def}^{\qq\mathrm{G}}(U,B_U)\to\operatorname{Def}^{\qq\mathrm{G}}(U)$ is smooth and surjective, and $B_U$ lifts to $\mathcal X'$ near $\Gamma_1$. By \cite[Corollary 3.23]{HTU17} the statement follows.
\end{proof}

\begin{corollary}\label{cor:antiflip-termination}
Let $X$ be a normal projective surface with $\rho(X)=2$ admitting a $\qq$-Gorenstein smoothing to $\ff_k$ and set $l_0=k \bmod 2$. Let $\Gamma_1\subset X$ be an extremal curve with $K_X\cdot\Gamma_1> 0$, contracting to a cyclic quotient singularity $\frac{1}{\Delta}(1,\Omega)$.

Then there exists a $\qq$-Gorenstein smoothing $\pi\colon \mathcal{X}_0 \to \mathbb{D}$ with central fiber $\mathcal{X}_{0,0}\cong X$ and general fiber $\mathcal{X}_{0,t}\cong \ff_{l_0}$, together with a finite sequence of terminal $K$-antiflips
\[
\mathcal{X}_0 \dashrightarrow \mathcal{X}_1 \dashrightarrow \cdots \dashrightarrow \mathcal{X}_N,
\]
such that:
\begin{enumerate}
    \item each $\pi_i\colon \mathcal{X}_i \to \mathbb{D}$ is a $\qq$-Gorenstein deformation with general fiber $\mathcal{X}_{i,t}\cong \ff_{l_0}$;
    \item if $X_i:=\mathcal{X}_{i,0}$ denotes the central fiber, then $X_0=X$ and $X_{i+1}$ is obtained from $X_i$ by a $K$-antiflip;
    \item the final central fiber $X_N$ satisfies that $-K_{X_N}$ is nef.
\end{enumerate}
Moreover, the toric boundary lifts to each $\mathcal{X}_i$.
\end{corollary}

\begin{proof}
By Theorem~\ref{thm:admissible-smoothing}, there exists a $\qq$-Gorenstein smoothing $\pi\colon\mathcal{X}\to\mathbb{D}$ of $X$ to $\ff_{l_0}$ with axial multiplicities $(\alpha_0,\alpha_1)$ at the singularities belonging to $\Gamma_1$.
The total space $\mathcal{X}$ is terminal and $\qq$-factorial \cite[Corollary~3.6]{KSB88}.

We reduce termination to the $3$-fold log MMP. Write $-K_X=P+\nu\Gamma_1$ as in Proposition~\ref{prop:zariski-decomposition}, and let $B_{\mathrm{nt}}=\Gamma_1+A+B$ be the near-toric boundary around $\Gamma_1$. Choose sufficiently small positive rational numbers $a_0,a,b$, with $\nu+a_0<1$, and set
$$B_X:=(\nu+a_0)\Gamma_1+aA+bB.$$
Since $P\cdot\Gamma_1=0$, $\Gamma_1^2<0$, and $P\cdot\Gamma_0>0$, for
sufficiently small coefficients the pair $(X,B_X)$ is klt and
$-(K_X+B_X)$ is ample. 

Let $\mathcal{B}$ be the corresponding boundary supported on the lifted near-toric boundary, so that $\operatorname{Diff}_X(\mathcal B)=B_X$. After shrinking $\mathbb{D}$, inversion of adjunction gives that $(\mathcal X,X+\mathcal B)$ is plt and $-(K_{\mathcal X}+X+\mathcal B)$ is $\pi$-ample.

Thus, for $m\gg0$ sufficiently divisible and a general
$M\in |-m(K_X+B_X)|$, the pair $(X,B_X^+:=B_X+\frac{1}{m}M)$ is klt and satisfies $m(K_X+B_X^+)\sim0$. 
By \cite[Proposition~3.7]{PS09}, this complement extends to a $m$-complement $K_{\mathcal X}+X+\mathcal B^+$ of $(\mathcal{X},X+\mathcal{B})$ with $\operatorname{Diff}_X(\mathcal B^+)=B_X^+$. 
Since $X\sim_{\mathbb{D}}0$, we have $\mathcal B^+\sim_{\qq}-K_{\mathcal X}$.

For $0<\epsilon\ll 1$, set $\Delta:=(1+\epsilon)\mathcal B^+$. Then:
\begin{enumerate}
\item $(\mathcal X,X+\Delta)$ is plt, and in particular $(\mathcal X,\Delta)$ is klt. Indeed, since $(X,B_X^+)$ is klt, inversion of adjunction \cite[Theorem~5.50]{KM98} gives that $(\mathcal X,X+\mathcal B^+)$ is plt near $X$. Hence, after shrinking
$\mathbb D$ and taking $\epsilon>0$ sufficiently small, $(\mathcal X,X+(1+\epsilon)\mathcal B^+)$ is plt.

\item
$K_{\mathcal{X}}+\Delta\sim_{\qq}
K_{\mathcal{X}}-(1+\epsilon)K_{\mathcal{X}} =-\epsilon K_{\mathcal X}$. Thus, a curve $\Gamma$ is $(K_{\mathcal X}+\Delta)$-negative if and only if $K_{\mathcal X}\cdot\Gamma>0$.
\end{enumerate}

Therefore every $K$-antiflip on the central fiber is a $(K_\mathcal{X}+\Delta)$-flip of the klt $3$-fold $(\mathcal{X},\Delta)$. By \cite[Theorem~6.17]{KM98}, any sequence of log flips of a $3$-dimensional klt pair terminates.
Starting from $\mathcal X_0=\mathcal X$, perform such an antiflip whenever $-K_{X_i}$ is not nef, and let $\Delta_i$ be the strict transform of
$\Delta$.
Since $X_i\sim_{\mathbb{D}}0$, each $(K_{\mathcal{X}_i}+\Delta_i)$-flip is also a $(K_{\mathcal{X}_i}+X_i+\Delta_i)$-flip. Hence
$(\mathcal{X}_i,X_i+\Delta_i)$ remains plt at every step. 
By adjunction formula, $(X_i,\operatorname{Diff}_{X_i}(\Delta_i))$ is klt, and therefore $X_i$ is klt.
Since $\mathcal X_i\to\mathbb D$ is again a $\qq$-Gorenstein deformation
with smooth general fiber, $\mathcal X_i$ is terminal by
\cite[Corollary~3.6]{KSB88}. Thus all the antiflips are terminal.
Hence for the sequence $$\mathcal{X}=\mathcal{X}_0\dashrightarrow\mathcal{X}_1\dashrightarrow\cdots,$$ there exists $N$ such that $K_{\mathcal{X}_N}+\Delta_N$ is nef over $\mathbb{D}$, or equivalently $-K_{X_N}$ is nef. Because each step of this process consists of a $K$-antiflip, it cannot end in a MFS.
Each $X_i$ admits a $\qq$-Gorenstein smoothing to $\ff_{l_0}$ and the toric structure near the anti-flipping curve lifts by Step~4 of Theorem~\ref{thm:admissible-smoothing}.
\end{proof}
\begin{remark}
Since every degeneration $\ff_k \rightsquigarrow X$ with $k\leq 1$ is obtained by a sequence of slides by Theorem~\ref{thm:toric-degeneration} (excluding those of Type ($\dagger$)), the $K$-MMP for a non-toric surface is controlled by its toric degeneration. In the case $k=1$, no degeneration $\ff_1 \rightsquigarrow X$ admits a curve $\Gamma_i$ with $K_X \cdot \Gamma_i = 0$ for $i=0,1$. It follows that the nef condition is strengthened to $-K$ ample.
\end{remark}

\subsection{Toric degenerations}\label{sub:toric-deg}
By Corollary~\ref{cor:antiflip-termination}, in this subsection we assume that $X$ is toric and $-K_X$ is ample. So $X=X_P$ for a Fano polygon $P$ with T-singularities and consequently $P\in[P_{\ff_k}]$ by \cite[Theorem 6]{KNP17}. If $\rho(X)=2$, by Proposition~\ref{prop:moriray} we have $\operatorname{NE}(X)=\rr_{\geq0}[\Gamma_0]\oplus\rr_{\geq0}[\Gamma_1]$ with $\Gamma_j^2<0$, so both extremal contractions $F_j\colon(\Gamma_j\subset X)\to Y_j$ are birational and $\rho(Y_j)=1$.
We first identify the surfaces $Y_j$.

\medskip
\begin{lemma}\label{lem:toric-contractions}
Let $\pi:(X\subset \mathcal{X})\to(0\in\mathbb{D})$ a $W$-surface with $X$ toric with $\rho(X)=2$ and general fiber $\mathcal{X}_t\cong\ff_k$. Suppose its singularity data is of the form
$$\Big[\binom{n_{34}}{a_{34}}\Big]-\Gamma_1-\Big[\binom{n_{14}}{a_{14}}\Big]-\Gamma_0-\Big[\binom{n_{12}}{a_{12}}\Big]-(1)-\Big[\binom{n_{23}}{a_{23}}\Big]$$
with corresponding contractions $F_i:(\Gamma_i\subset X)\to (\frac{1}{\Delta_i}(1,\Omega_i)\in Y_i)$. Then, the surfaces $Y_i$ are fake weighted projective planes of the form
\begin{itemize}
    \item $Y_0=\pp(\Delta_0,n_{34}^2,n_{23}^2)/(\zz/n^2)$ where $n=\operatorname{gcd}(n_{34},n_{23})$.

    \item $Y_1 = \pp(\Delta_1, n_{12}^2, n_{23}^2) / (\mathbb{Z} / {n^\prime}^2)$, where $n^\prime = \gcd(n_{12}, n_{23})$.
\end{itemize}
\end{lemma}
\begin{proof}
Since the contractions $F_i$ are toric morphisms, we observe that $Y_i$ is a toric surface with Picard rank $\rho(Y_i)=1$, hence a fake weighted projective plane. Without loss of generality, we assume that $n_{34},n_{23}>1$. By Lemma \ref{lem:Hz-minres}, the minimal resolution of $Y_i$ given by $Y_i^\prime\to \ff_{d_i}$ has two degenerate fibers $f_1,f_2$ consisting of only one $(-1)$ curve each. Then, we obtain the equations
$\Delta_0=n_{23}^2(n_{34}^2d_0-n_{34}(n_{34}-a_{34})-1)-n_{34}^2(n_{23}a_{23}+1)$ and $n_{12}^2=n_{23}^2(\Delta_1d_1-\Omega_1^*)-\Delta_1(n_{23}a_{23}+1)$, where $\Omega_1\Omega_1^*\equiv 1 \pmod {\Delta_1}$. For the $i=0$ case we note that $n^2|\Delta_0$ and consequently $\operatorname{gcd}(\Delta_0,n_{34}^2,n_{23}^2)=n^2$. This implies the sublattice $N_0$ generated by the vectors $\{v_2,v_3,v_4\}$ has index $[N:N_0]=n^2$. For $i=1$, we observe that ${n^\prime}^2|\Delta_1(n_{23}a_{23}+1)$. But since $\operatorname{gcd}(n_{23}^2,n_{23}a_{23}+1)=1$, it follows that ${n^\prime}^2|\Delta_1$, the claim follows similarly.
\end{proof}
By Theorem~\ref{Ildef} every mutation of $P$ induces a $\qq$-Gorenstein pencil between the associated toric surfaces; in order to describe the singularities occurring along $[P_{\ff_k}]$ we use the following explicit form of that pencil.

\begin{theorem}[{\cite[Theorem 7.3]{P21}}] 
 Let $P \subseteq N_{\mathbb{R}}$ be a Fano polygon, let $(w, F)$ be a mutation datum for $P$, and let $P^{\prime}=\operatorname{mut}_{w, F}(P)$ be the mutated polygon. Let $X_P$ (resp. $X_{P^\prime}$) be the Fano toric variety associated to the spanning fan of $P$ (resp. $P^\prime$) and let $\partial X_P$ (resp. $\partial X_{P^{\prime}}$ ) be the toric boundary of $X_P$ (resp. $X_{P^{\prime}}$ ). Set
$$
\begin{aligned}
\mathcal{V}(P)^{\geq 0} & =\mathcal{V}(P) \cap\{v \in N \mid\langle w, v\rangle \geq 0\}, \\
\mathcal{V}\left(P^{\prime}\right)^{<0} & =\mathcal{V}\left(P^{\prime}\right) \cap\{v \in N \mid\langle w, v\rangle<0\} .
\end{aligned}
$$

Consider the lattice $\tilde{N}=N \oplus \mathbb{Z} e_1$ and the polyhedron $\tilde{Q} \subseteq \tilde{M}_{\mathbb{R}}$ defined by
$$
\tilde{Q}=\left\{\begin{array}{l|l}
u+k e_1^* \in \tilde{M}_{\mathbb{R}} & \begin{array}{c}
\forall p \in \mathcal{V}(P) \geq 0, \quad\langle u, p\rangle+1 \geq 0 \\
\forall p^{\prime} \in \mathcal{V}\left(P^{\prime}\right)^{<0}, \quad\left\langle u, p^{\prime}\right\rangle+1+k\left\langle w, p^{\prime}\right\rangle \geq 0 \\
\forall f \in \mathcal{V}(F), \quad\langle u, f\rangle+k \geq 0
\end{array}
\end{array}\right\} .
$$

Then $\tilde{Q}$ is a full dimensional rational polytope and the primitive generators of the rays of the normal fan $\tilde{\Sigma}$ of $\tilde{Q}$ are

\begin{itemize}
    \item $p$ for $p \in \mathcal{V}(P)^{\geq 0}$ 
    \item $p^{\prime}+\left\langle w, p^{\prime}\right\rangle e_1$ for $p^{\prime} \in \mathcal{V}\left(P^{\prime}\right)^{<0}$
    \item $f+e_1$ for $f \in \mathcal{V}(F)$.
\end{itemize}

Let $\tilde{X}=X(\tilde{\Sigma})$ be the toric variety associated to $\tilde{\Sigma}$. Consider the reduced divisor $\mathcal{D}$ on $\tilde{X}$ defined by the following monomial in the Cox coordinates of $\tilde{X}$ :

\begin{equation}
\prod_{p \in \mathcal{V}(P) \geq 0} x_p \prod_{p^{\prime} \in \mathcal{V}\left(P^{\prime}\right)^{<0}} x_{p^{\prime}}
\end{equation}

Set $V=\pp_{\mathbb{C}}^2 \setminus\{[1: 0: 0],[0: 1: 0]\}$. Consider the closed subscheme $\mathcal{X}$ of $\tilde{X} \times_{\text {Spec } \mathbb{C}} V$ defined by the vanishing of the trinomial obtained by varying the three coefficients of

\begin{equation}
\prod_{p \in \mathcal{V}(P) \geq 0} x_p^{\langle w, p\rangle}+\prod_{p^{\prime} \in \mathcal{V}\left(P^{\prime}\right)^{<0}} x_{p^{\prime}}^{-\left\langle w, p^{\prime}\right\rangle}+\prod_{f \in \mathcal{V}(F)} x_f . 
\end{equation}

Then in the diagram

\begin{center}
\begin{tikzcd}
\mathcal{X}\cap (\mathcal{D}\times V) \arrow[rd] \arrow[r, hook] & \mathcal{X} \arrow[d] \\
                            & V                       
\end{tikzcd}
\end{center}
the two morphisms with target $V$ are flat and the base change of the diagram to the points $[0: 1:-1]$ and $[1: 0:-1]$ of $V$ are the closed embeddings $\partial X_P \hookrightarrow X_P$ and $\partial X_{P^{\prime}} \hookrightarrow X_{P^{\prime}}$ over $\operatorname{Spec} \mathbb{C}$
respectively.
\label{thm:petrdeform}
\end{theorem}

\subsection{Degenerations of $\ff_1$}

Let $P\in[P_{\ff_k}]$ for $k\leq 1$ a quadrilateral. We say $P$ does not have a pair of parallel edges, if for every edge $E$ of $P$ with associated mutation data $(w,f)$, the vertices outside $E$ have different height. For convenience, we apply a transformation $M\in GL_2(\mathbb{Z})$ such that the vertices of $P$ are ordered counterclockwise, satisfying $w(v_1)=h_{max}$ and $w(v_2)=w(v_3)=h_{min}$. Since $X_P$ contains only Wahl singularities, it follows that $|E\cap N|-1=-h_{min}$. This restricts the mutation data to $v_3-v_2=(-h_{min})f$.

In particular, any $P\in [P_{\ff_1}]$ does not have an edge $E^\prime$ parallel to $E$. If such edges existed, we would obtain $w(v_1)=w(v_4)=h_{max}$ and consequently $\operatorname{mut}_w(P,F)=\operatorname{conv}\{v_1,v_2,v_4+h_{max}f\}$. This contradicts the fact that degenerations of $\ff_1$ always satisfy $\rho=2$ by \cite[Theorem 4.1]{HP10}.

\begin{convention}\label{conv:orientation}
Let $P\in [P_{\ff_k}]$ for $k\leq 1$ be a quadrilateral with no pair of parallel edges. We pick a counterclockwise ordering, such that for the mutation data $(w,f)$ it satisfies:
\begin{itemize}
    \item $v_1$, such that $w(v_1)=h_{max}$.
    \item $v_2$ and $v_3$, such that $w(v_2)=w(v_3)=h_{min}$.
    \item $v_4$, such that $h_{min}<w(v_4)<h_{max}$.
\end{itemize}
 We denote the mutating edge $E:=\operatorname{conv}\{v_2,v_3\}$, and more generally for $P$ following an ordering as above, we define $E_{ij}:=\operatorname{conv}\{v_i,v_j\}$ for $j\equiv i+1 \pmod{4}$.

 For each edge $E_{ij}$, we compute mutations following the setting of \cite[Corollary 3]{KNP17}. Indeed, we denote the associated mutation data by $(w_{ij},f_{ij})$, where $f_{ij}$ is the primitive direction in which $E_{ij}$ is traversed by the ordering above, that is
$$f_{ij}=\frac{1}{-h_{ij}}(v_j-v_i) \quad\text{if that ordering is counterclockwise},\qquad
f_{ij}=\frac{1}{-h_{ij}}(v_i-v_j) \quad\text{if it is clockwise}.$$
Here $h_{ij}:=\min\{w_{ij}(v)\mid v\in P\}$ denotes the minimum height with respect to $w_{ij}$.
With this convention $f_{ij}$ is the vector along which the mutation $\operatorname{mut}_{w_{ij}}(P,F_{ij})$ is performed in all cases.
\end{convention}
Now, we apply the mutation described in Definition \ref{def:mut} and compute the vertices for $\operatorname{mut}_w(P,F)$. It follows that: 
\begin{equation}\label{mutpol}
    \operatorname{mut}_w(P,F)=\operatorname{conv}\{v_1,v_2,v_4+w(v_4)f,v_1+h_{max}f\}.
\end{equation}

Since $w(v_4+w(v_4)f)=w(v_4)$, this allows us to compute the equations defining the divisor $\mathcal{D}$ of Theorem \ref{thm:petrdeform}. In Cox coordinates of $\tilde{X}$, we make the $x_1$ corresponding to the vertex $v_1$ of $P$ and $x_2$ to the vertex $v_2$ of $P^\prime$. Also, $x_4$ corresponds to the vertex $v_4$ of $P$ if $w(v_4)\geq 0$, or to $v_4+w(v_4)f$ of $P^\prime$ if $w(v_4)<0$. So the divisor $\mathcal{D}$ is given by the equation $x_1x_2x_4=0$. 
In addition, the subscheme $\mathcal{X}$ of $\tilde{X} \times V$ of Theorem \ref{thm:petrdeform} defines the deformation established in Theorem \ref{Ildef} and it is given by the vanishing locus of the trinomial:

\medskip
If $w(v_4)\geq 0$,
\begin{equation}
 x_1^{h_{max}}x_4^{w(v_4)}+x_2^{-h_{min}}+z_0z_1
\end{equation}
and if $w(v_4)<0$,
\begin{equation}
x_1^{h_{max}}+x_2^{-h_{min}}x_4^{-w(v_4)}+z_0z_1,
\end{equation}
where the coordinates $z_i$ are given by the end points of the segment $F=\operatorname{conv}\{\mathbf{0},f\}$. In the condition $w(v_4)\geq 0$, the restriction of $V$ onto $\pp^1_{s,t}=\{[s:t:-1]\in \pp^2_{\mathbb{C}}\}\setminus\{[0:0:1]\}$ gives us the family 
\begin{equation}\label{eq:defpos}
\mathcal{X}_{s,t}=\{sx_1^{h_{max}}x_4^{w(v_4)}+tx_2^{-h_{min}}-z_0z_1=0\}\subseteq \tilde{X} \times \pp^1_{s,t}.
\end{equation}
Following Theorem \ref{thm:petrdeform}, the fibers over $[0:1]$ and $[1:0]$ are isomorphic to $X_P$ and $X_{P^\prime}$ respectively. Also, the subscheme $\mathcal{B}_{s,t}=\mathcal{X}_{s,t}\cap \{x_1x_2x_4=0\}$
gives the deformation of the toric boundaries $\partial X_P$ onto $\partial X_{P^\prime}$.
Similarly holds for the case $w(v_4)<0$, we obtain the family
\begin{equation}\label{eq:defneg}
\mathcal{X}_{s,t}=\{sx_1^{h_{max}}+tx_2^{-h_{min}}x_4^{-w(v_4)}-z_0z_1=0\}\subseteq \tilde{X} \times \pp^1_{s,t},
\end{equation}
and the deformation of the toric boundaries $\mathcal{B}_{s,t}$ as above.

\medskip
When $P\in [\ff_k]$ is a quadrilateral with no parallel edges. The identification of the edges of  $P^\prime=\operatorname{mut}_w(P,F)$ is straightforward by means of polyhedral defining inequalities. From now on, we establish the notation for the half space $K_{a,t}:=\{v\in \mathbb{R}^n:\langle a,v\rangle \geq t\}$. We leave to the reader the details of the following lemma.

\begin{lemma}
Let $P\in[P_{\ff_k}]$ for $k\leq 1$ with no pair of parallel edges and mutation data $(w,f)$. If $P^\prime=\operatorname{mut}(P,F)$, the list of its edges is
\begin{itemize}
    \item $\operatorname{conv}\{v_1,v_1+h_{max}f\}$
    \item $\operatorname{conv}\{v_1,v_2\}$
    \item $\operatorname{conv}\{v_2,v_4+w(v_4)f\}$
    \item $\operatorname{conv}\{v_1+h_{max}f,v_4+w(v_4)f\}.$
\end{itemize}
\label{pollist}
\end{lemma}

\begin{remark}
Since mutation is invertible i.e, $\operatorname{mut}_{-w}(P^\prime,F)=P$, the mutation data $(-w,f)$ for $P^\prime$ satisfies $-w(v_4+w(v_4)f)<-h_{min}$ and $-w(v_4+w(v_4)f)\leq 0$ if and only if $w(v_4)\geq 0$. 
\label{invertingsign}
\end{remark}
Our main goal is to study the deformations of the T-invariant curves of $X$ under the pencil $\mathcal{X}_{s,t}$. For this end, we first analyze the birational geometry of the toric threefold $\tilde{X}$ from Theorem \ref{thm:petrdeform}.

We compute the vectors defining the normal fan $\tilde{\Sigma}$ associated to $\tilde{X}$. In our specific context where we assume that $w(v_4) \geq 0$, Theorem \ref{thm:petrdeform} establishes that $\tilde{\Sigma}$ is the spanning fan induced by the rational polytope
$$ \tilde{P} = \operatorname{conv}\{(v_1,0), (v_4,0), (v_2,h_{\min}),(\mathbf{0},1),(f,1)\}\subseteq (N\oplus \mathbb{Z})_\mathbb{R}.$$
We label these vertices $X_1,\dots,X_5$ in that order, and write $F_{ijk}:=\operatorname{conv}\{X_i,X_j,X_k\}$ with $i<j<k$. 

\begin{proposition}
Let $P\in[P_{\ff_k}]$ for $k\leq 1$ with no pair of parallel edges, with mutation data $(w,f)$ and $w(v_4)\geq 0$. Assume that $P$ is oriented as in Convention \ref{conv:orientation}. Then, the spanning fan $\tilde{\Sigma}$ is simplicial with maximal cones:

\begin{itemize}

    \item $\sigma_{145}=\operatorname{cone}\{X_1,X_4,X_5\}$, \hspace{0.45cm} $\sigma_{134}=\operatorname{cone}\{X_1,X_3,X_4\}$
    
    \item $\sigma_{125}=\operatorname{cone}\{X_1,X_2,X_5\}$, \hspace{0.45cm} $\sigma_{123}=\operatorname{cone}\{X_1,X_2,X_3\}$
    
    \item $\sigma_{235}=\operatorname{cone}\{X_2,X_3,X_5\}$, \hspace{0.45cm} $\sigma_{345}=\operatorname{cone}\{X_3,X_4,X_5\}.$
\end{itemize}
\label{conelist}
\end{proposition}

\begin{proof}
We show that $F_{145}$, $F_{123}$, $F_{125}$ and $F_{345}$ are faces of $\tilde{P}$ by providing explicit half-spaces in $(M\oplus \mathbb{Z})_{\mathbb{R}}$. Indeed, by straightforward computations $F_{145}$ is determined by the half-space $K_{(-w,-h_{max}),-h_{max}}$ and $F_{345}$ by $K_{((1-h_{min})w,h_{min}),h_{min}}$. 

For $F_{125}$, we prove that the half-space $K_{(w_{41},h_{41}-w_{41}(f)),h_{41}}$ determines a face of $\widetilde{P}$. Let $\xi=(w_{41},h_{41}-w_{41}(f))$, then $\langle \xi, X_3\rangle = h_{41}h_{min}+w_{41}(v_3)>h_{41}(1+h_{min})\geq 0$. By the convention \ref{conv:orientation} it follows immediately that $w_{41}(v_3)<w_{41}(v_2)$. Therefore, $\langle \xi, X_4\rangle=h_{41}-w_{41}(f)>h_{41}$. 

For $F_{123}$ we show that $K_{\big(w_{41},\frac{h_{41}-w_{41}(v_2)}{h_{min}}\big),h_{41}}$ determines a face of $\widetilde{P}$. Let $\xi=(w_{41},\frac{h_{41}-w_{41}(v_2)}{h_{min}})$, it follows that $\langle \xi, X_4\rangle=\frac{h_{41}-w_{41}(v_2)}{h_{min}}>h_{41}$ and $\langle \xi, X_5\rangle=\frac{h_{41}-w_{41}(v_3)}{h_{min}}>h_{41}$.

\medskip
These four faces show that no four vertices of $\widetilde{P}$ are coplanar. Specifically, for any four vectors $X_{i_1}, \dots, X_{i_4}$, we can find a face $F$ that contains $X_{i_1}, X_{i_2}, X_{i_3}$ but not $X_{i_4}$. Also, it is easy to see that no three vertices lie on a line.
Then by \cite[Page 8]{Z12}, $\tilde{P}$ must be simplicial, implying all maximal dimensional faces are of the form $F_{ijk}$. Since the edge $\operatorname{conv}\{X_3,X_4\}$ of $\tilde{P}$ must separate two unique faces, this implies that $F_{134}$ is a face of $\tilde{P}$. Similarly, the edge $\operatorname{conv}\{X_3,X_5\}$ induces the face $F_{235}$, which gives the full list of faces of $\tilde{P}$. Indeed, the 3-dimensional cones of $\Sigma$ are the ones listed in the statement.
\end{proof}
\begin{remark}
For $w(v_4)<0$ the vector $(v_4,0)$ is replaced by $(v_4,w(v_4))$, and since $h_{min}<w(v_4)$ the combinatorial structure of $\tilde{P}$, hence the list of maximal cones of $\tilde{\Sigma}$, is unchanged; alternatively, this case is recovered by the inverse mutation of Remark \ref{invertingsign}.
\end{remark}

The origin lies in the strict interior of $\tilde{P}$ and the vertices of $\tilde{P}$ are primitive in $N\oplus\zz$, so $\tilde{X}=X_{\tilde{P}}$ is a projective toric threefold which is $\qq$-factorial, because $\tilde{\Sigma}$ is simplicial, and Fano, because $\mathbf{0}\in\operatorname{int}(\tilde{P})$.
Its Picard number is $\rho(\tilde{X})=2$ and $\operatorname{NE}(\tilde{X})$ is closed, so
\[
\operatorname{NE}(\tilde{X})=\operatorname{NE}(\tilde{X})_{K_{\tilde{X}}<0}=\rr_{\geq 0}[C_{\omega_1}]+\rr_{\geq 0}[C_{\omega_2}],
\]
where the $C_{\omega_i}$ are $T$-invariant curves induced by walls $\omega_i\in\tilde{\Sigma}(2)$, and our task is to determine those walls explicitly.
For this we follow \cite{CvR09}, where extremal rays are read off from primitive relations. We recall only what is used below, in the case of projective $\qq$-factorial toric varieties.

\begin{definition}[{\cite[Section 1.3]{CvR09}}]\label{primrrel}
Let $Y$ be a projective $\qq$-factorial toric variety with fan $\Delta$.
A subset $S=\{\rho_1,\dots,\rho_k\}\subseteq\Delta(1)$ is a \emph{primitive collection} for $\Delta$ if it is not contained in a single cone of $\Delta$ but every proper subset is; writing $\sigma_S\in\Delta$ for the unique cone with $\rho_1+\cdots+\rho_k\in\relint(\sigma_S)$, the resulting expression
$\rho_1+\cdots+\rho_k=\sum_{\rho\in\sigma_S(1)}a_\rho\rho$ with $a_\rho\in\qq^{+}$ is the \emph{primitive relation} $l_S$.
\end{definition}

Recall further that, under the identification of $N_1(Y)_{\rr}$ with the rational relations among the ray generators of $\Delta$, a wall $\omega\in\Delta(n-1)$ determines the wall relation of $[C_\omega]$, namely $\sum_{i=1}^{n+1}a_i\rho_i=0$ with $a_n,a_{n+1}>0$ and $a_i\in\qq$.

\begin{theorem}[{\cite[Proposition 2.1]{CvR09}}]\label{coxprim}
Let $\Delta$ be a projective simplicial fan with convex support of dimension $n$, and let $\omega$ be an extremal wall, so that $[C_\omega]$ generates an extremal ray $l_\omega\in\operatorname{NE}(Y)$.
Then $R=\{\rho_i:a_i>0\}$ is a primitive collection for $\Delta$, and in $\rr^{|\Sigma(1)|}$ the primitive relation $a_R$ of $R$ and the wall relation $a_\tau$ of $\omega$ agree up to a positive constant.
\end{theorem}

From the list of cones of Proposition \ref{conelist}, the fan $\tilde{\Sigma}$ has exactly two primitive collections, $S_1=\{X_1,X_3,X_5\}$ and $S_2=\{X_2,X_4\}$.
Since $\operatorname{NE}(\tilde{X})$ has exactly two extremal rays, any $T$-invariant curve $C_\omega$ whose wall relation is a multiple of some $l_{S_i}$ is extremal.
For notational convenience we write $\omega_{ij}:=\operatorname{cone}\{X_i,X_j\}$.

\begin{proposition}
Let $P\in[P_{\ff_k}]$ as in Convention \ref{conv:orientation} with mutation data $(w,f)$ and $w(v_4)\geq 0$. Then, the ray $R=\mathbb{R}_{\geq 0}[C_{\omega_{15}}]$ is extremal in $\operatorname{NE}(\widetilde{X})$.
\label{flipray}
\end{proposition}
\begin{proof}
We find the primitive relation for $S_2$. We make a change of coordinates $M\in GL_2(\mathbb{Z})$ such that $w=(0,1)$. Applying $w$ to $v_4=av_1+bf$ gives $a=\frac{w(v_4)}{h_{max}}\geq 0$.

Applying $\det(v_1,-)$ to $v_4=av_1+bf$ gives $b\det(v_1,f)=\det(v_1,v_4)$. With $w=(0,1)$
one has $f=(1,0)$, so $\det(v_1,f)=-h_{max}<0$, and $\det(v_1,v_4)<0$ because $v_4,v_1$ are
consecutive in the counterclockwise ordering. Hence $b>0$.
Depending on the value of $b$, we obtain the expression 

\begin{equation}
X_2+X_4 = \begin{cases} 
      aX_1+bX_5+(1-b)X_4 & b\leq 1 \\
      \frac{a}{b}X_1+X_5+(1-\frac{1}{b})X_2 & b>1.
   \end{cases}  
\end{equation}

This is the primitive relation $l_{S_2}$. We observe that the relation $l_{S_2}$ must be a multiple of the wall relation of $[C_{\omega_{15}}]$, where $\omega_{15}$ is the wall separating the cones $\sigma_{145}$ and $\sigma_{125}$. By Theorem \ref{coxprim}, $R$ is therefore an extremal ray of $\operatorname{NE}(\widetilde{X})$.
\end{proof}
We denote this extremal ray by $\tilde{R}_2$ and turn to the remaining ray $\tilde{R}_1$, whose description is more involved: the primitive relation attached to $S_1$ depends on the choice of mutation data.

\begin{definition}\label{def:std-orientation}
A Fano quadrilateral $P\in [P_{\ff_k}]$ with $k\leq 1$ is in \textit{standard orientation} if, after a
$\operatorname{GL}_2(\zz)$-transformation we obtain a counter-clockwise ordering such that for the mutation data $(w_{23},f_{23})$, the divisors $\Gamma_0=V(\mathbb{R}_{\geq 0}v_1)$ and
$\Gamma_1=V(\mathbb{R}_{\geq 0}v_4)$ are the extremal curves of $\operatorname{NE}(X_P)$.
\end{definition}
Every Fano quadrilateral $P\in[P_{\ff_k}]$ can be taken in standard orientation after a $\operatorname{GL}_2(\zz)$-change of coordinates, and we label the divisors $V(\rr_{\geq 0}v_2)$, $V(\rr_{\geq 0}v_3)$ as $A$, $B$ respectively.
From now on we assume that $P$ is in standard orientation and that $P\not\cong P_{\ff_k}$. The excluded case is treated in Example \ref{f1mut}.
The wall relations for $\rho_1=\rr_{\geq 0}v_1$ and $\rho_4=\rr_{\geq 0}v_4$ are
\begin{equation}
 \begin{cases} 
      v_4+t_1v_2=a_1v_1 \\
      v_3+t_2v_1=a_2v_4,
   \end{cases}
\label{curvextrel}
\end{equation}
with $t_i,a_i>0$ given by
$$t_1=\frac{\Gamma_0\cdot A}{\Gamma_0\cdot\Gamma_1}, t_2=\frac{\Gamma_0\cdot \Gamma_1}{\Gamma_1\cdot B}, a_1=-\frac{\Gamma_0^2}{\Gamma_0\cdot\Gamma_1} \hspace{0.1cm}, \hspace{0.1cm} a_2=-\frac{\Gamma_1^2}{\Gamma_1\cdot B}.$$ These curve intersections are computed purely in terms of cone multiplicities of simplicial cones (see for example \cite[Chapter 6]{CLS11}). 
Writing $\tau_{ij}:=\operatorname{cone}\{v_i,v_j\}$ for $j=i+1\pmod 4$, so that $m_{ij}:=\mult(\tau_{ij})=h_{ij}^2$, they are
\begin{itemize}
    \item $\Gamma_0\cdot A=\frac{1}{m_{12}}$, \hspace{0.15cm}  $\Gamma_0\cdot B=0$, \hspace{0.15cm} 
  $\Gamma_0\cdot\Gamma_1=\frac{1}{m_{41}}$, \hspace{0.15cm} $\Gamma_0^2=-\frac{\Delta_0}{m_{41}m_{12}}$,

  \item $\Gamma_1\cdot B=\frac{1}{m_{34}}$, \hspace{0.15cm}  $\Gamma_1\cdot A=0$, \hspace{0.15cm} 
  $A\cdot B=\frac{1}{m_{23}}$, \hspace{0.15cm} $\Gamma_1^2=-\frac{\Delta_1}{m_{34}m_{41}}$,
    \end{itemize}
where $\Delta_i>0$ is the orbifold index of the cyclic quotient singularity obtained by contracting $\Gamma_i$. From these we obtain $m_{23}A\equiv\Delta_1\Gamma_0+m_{34}\Gamma_1$ and $m_{23}B\equiv m_{12}\Gamma_0+\Delta_0\Gamma_1$ in $N_1(X)_{\rr}$, hence the relation
\begin{equation}
    \Delta_0\Delta_1=m_{12}m_{34}-m_{41}m_{23},
\label{singrel}
\end{equation}
for which, we compute $A^2=\frac{\Delta_1}{m_{12}m_{23}}$ and $B^2=\frac{\Delta_0}{m_{23}m_{34}}$.
\begin{remark}
Expanding $K_{X_P}^2$ for $-K_{X_P}\sim \Gamma_0+\Gamma_1+A+B$ gives the Diophantine relation
\begin{multline}\label{generalmarkovian}
8m_{12}m_{23}m_{34}m_{41}=\Delta_0(m_{12}m_{41}-m_{23}m_{34})+\Delta_1(m_{41}m_{34}-m_{12}m_{23}) \\
+2m_{23}m_{34}m_{41}+2m_{12}m_{34}m_{41}+2m_{12}m_{23}m_{41}+2m_{12}m_{23}m_{34}
\end{multline}
Setting $m_{41}=m_{34}^2$ and $\Delta_1=m_{34}$ so that $\Delta_0=m_{12}-m_{23}m_{34}$, Equation \ref{generalmarkovian} becomes Markov's equation $9m_{12}m_{34}m_{23}=(m_{12}+m_{34}+m_{23})^2$. This specialization corresponds to a weighted blow-up on a degeneration $\pp^2\rightsquigarrow\pp(m_{12},m_{34},m_{23})$.
\end{remark}
By Proposition \ref{prop:no-obstruction} we have $H^2(X_P,T_{X_P})=0$. Thus, $X_P$ does not have local-to-global obstructions to deform and there exist a $\qq$-Gorenstein smoothing $(X_P\subset \mathcal{X})\to (0\in\mathbb{D})$ with general fiber $\ff_k$ with $k\leq 1$. In particular, by Proposition \ref{prop:moriray} we have extremal neighborhoods $(\Gamma_i\subset \mathcal{X})\to\bigl(\frac1{\Delta_i}(1,\Omega_i)\in\mathcal Y_i\bigr)$, for which we set  $\delta_0=|h_{12}||h_{41}|\,(-K_{X_P}\cdot\Gamma_0)$ and $\delta_1=|h_{41}||h_{34}|\,(-K_{X_P}\cdot\Gamma_1)$. 

\begin{lemma}\label{lem:F1moritrain}
Let $P\in[P_{\ff_k}]$ have no pair of parallel edges and be in standard orientation.
Then
\begin{equation}\label{eq:KGamma}
m_{12}m_{41}(\Gamma_0\cdot K_{X_P})=-(m_{12}+m_{41}-\Delta_0),\qquad
m_{41}m_{34}(\Gamma_1\cdot K_{X_P})=-(m_{41}+m_{34}-\Delta_1),
\end{equation}
Moreover, $m_{34}\leq m_{41}$, and hence
$t_2\leq1$, $a_2\leq1$,
with equality $t_2=a_2=1$ if and only if $m_{34}=m_{41}$.
\end{lemma}

\begin{proof}
Expanding $-K_{X_P}\sim\Gamma_0+\Gamma_1+A+B$ gives
$$\Gamma_0\cdot(-K_{X_P})=\Gamma_0^2+\Gamma_0\cdot\Gamma_1+\Gamma_0\cdot A=\frac{m_{12}+m_{41}-\Delta_0}{m_{12}m_{41}},$$
and the second identity of \eqref{eq:KGamma} follows in the same way from $\Gamma_1^2$, $\Gamma_0\cdot\Gamma_1$ and $\Gamma_1\cdot B$.

We claim $m_{34}\leq m_{41}$, and argue according to the contraction induced by $\Gamma_0$.
If it is divisorial, the inequality is \cite[Proposition 7.4]{HP10}.
If it is a flipping contraction then $m_{34}\leq\Delta_0$ by Lemma~\ref{lem:Hz-minres}, so $m_{34}\leq\max\{m_{12},m_{41}\}$, and we may assume $m_{41}<m_{12}$.
If $\delta_0|h_{41}|-|h_{12}|>0$, the neighbourhood mutation of $(\Gamma_0\subseteq X_P)$ produces a surface carrying an extremal neighbourhood $(\Gamma_0'\subseteq X')$ in which the Wahl singularity of index $|h_{34}|$ appears, whence $m_{34}\leq m_{41}$.
If $\delta_0|h_{41}|-|h_{12}|<0$ and $m_{34}>1$, then the flip of $\Gamma_0$ yields $X^{+}$ with $(\Gamma_0^{+})^2>0$ in which the singularity of index $|h_{41}|$ contains the negative section contracted by $Y^{+}\to X^{+}$; since $X^{+}$ still carries the Wahl singularity of index $|h_{34}|$, the curve configuration of $Y^{+}$ described in Lemma~\ref{lem:Hz-minres} forces $|h_{34}|\leq|h_{41}|$.

Since \eqref{eq:KGamma} gives $\Delta_1=m_{34}+m_{41}-\delta_1|h_{34}||h_{41}|$, from the previous statement it follows that
$$\Delta_1\leq m_{34}+m_{41}-|h_{34}||h_{41}|=m_{41}-|h_{34}|\bigl(|h_{41}|-|h_{34}|\bigr)\leq m_{41},$$
that is $a_2=\Delta_1/m_{41}\leq1$; equality in either case forces $\delta_1=1$ and $|h_{34}|=|h_{41}|$.
\end{proof}

\begin{remark}
The equality $t_2 = m_{34}/m_{41} = 1$ holds if and only if $m_{34}=m_{41}$, which occurs precisely when $X_P\cong\operatorname{Bl}_p\pp(1,a^2,b^2)$ for a torus-invariant point $p$. In this case $a_2 = \Delta_1/m_{41} = 1$. Otherwise, $\Delta_1<m_{41}$ and $a_2 = \Delta_1/m_{41}<1$. 
\end{remark}

\begin{lemma}\label{lem:sign}
Let $P\in[P_{\ff_k}]$  with no pair of parallel edges and in standard orientation. The following hold:
\begin{enumerate}[label=\roman*)]
\item $w_{23}(v_1)=\frac{\Delta_1+m_{12}}{-h_{23}}$, $w_{23}(v_4)=\frac{\Delta_0+m_{34}}{-h_{23}}$ and $0<w_{23}(v_4)<w_{23}(v_1)$.\\

\item $w_{34}(v_1)=\frac{m_{41}-\Delta_1}{-h_{34}}$, $w_{34}(v_2)=\frac{\Delta_0+m_{23}}{-h_{34}}$. It follows that, $0\leq \min\{w_{34}(v_1),w_{34}(v_2)\}$ and $w_{34}(v_2)<w_{34}(v_1)$ if and only if $m_{23}<m_{41}-\Delta_0-\Delta_1$.\\

\item $w_{41}(v_2)=\frac{m_{12}-\Delta_0}{-h_{41}}$, $w_{41}(v_3)=\frac{m_{34}-\Delta_1}{-h_{41}}$ and $w_{41}(v_3)<w_{41}(v_2)$. \\

\item $w_{12}(v_3)=\frac{m_{23}+\Delta_1}{-h_{12}}$, $w_{12}(v_4)=\frac{m_{41}-\Delta_0}{-h_{12}}$ and $w_{12}(v_3)<w_{12}(v_4)$ if and only if $m_{23}<m_{41}-\Delta_0-\Delta_1$.
\end{enumerate}
\end{lemma}
\begin{proof}
The values $w_{ij}(v_k)$ are obtained by combining the system of equations in \ref{curvextrel} and \ref{singrel}. The assertion in i) $0<w_{23}(v_4)<w_{23}(v_1)$ follows by construction, which implies that $\Delta_0+m_{34}<\Delta_1+m_{12}$. Consequently, $w_{41}(v_3)<w_{41}(v_2)$ in iii) follows directly from this. 

For ii), we observe that if $w_{34}(v_2)<w_{34}(v_1)$, then $h_{34}+t_1w_{34}(v_2)=a_1w_{34}(v_1)$ implies that $0<w_{34}(v_2)$. Similarly, since $w_{34}(v_4)=h_{34}$, applying $w_{34}$ to $v_3+t_2v_1=a_2v_4$ gives $t_2w_{34}(v_1)=-h_{34}(1-a_2)$, so $0\leq w_{34}(v_1)$ because $a_2\leq1$, with $w_{34}(v_1)=0$ if and only if $a_2=1$. The equivalence of $w_{34}(v_2)<w_{34}(v_1)$ follows by direct computation. In the same manner iv) holds.
\end{proof}

The following proposition determines the T-invariant curve inducing $\Tilde{R}_1$. That computation distinguishes the corresponding deformation induced by the mutation data $(w_{ij},f_{ij})$.

\begin{proposition}\label{prop:extremal}
Let $P$ be as above, and let $(w_{ij},f_{ij})$ denote the mutation data with respect to the edge $E_{ij}$, the vectors determining $\tilde{P}$ for the corresponding deformation being denoted $X_i$. Then:
\begin{enumerate}[label=\roman*)]
    \item For $(w_{23},f_{23})$, we have $X_1+X_3+X_5\in\sigma_{123}$, and 
          for $(w_{41},f_{41})$, we have $X_1+X_3+X_5\in\sigma_{134}$.

    \item If $m_{23}<m_{41}-\Delta_0-\Delta_1$, then
          $X_1+X_3+X_5\in\sigma_{134}$ for both $(w_{34},f_{34})$
          and $(w_{12},f_{12})$.

    \item If $m_{41}-\Delta_0-\Delta_1<m_{23}$, then
          $X_1+X_3+X_5\in\sigma_{123}$ for both $(w_{34},f_{34})$
          and $(w_{12},f_{12})$.
\end{enumerate}
If $X_1+X_3+X_5\in\sigma_{123}$, the ray $\widetilde{R}_1$ is induced by any wall $\omega_{k2}\in\widetilde{\Sigma}(2)$ with $k\in\{1,3,5\}$.
If $X_1+X_3+X_5\in\sigma_{134}$, it is induced by any wall $\omega_{k4}\in\widetilde{\Sigma}(2)$ with $k\in\{1,3,5\}$.
\end{proposition}

\begin{proof}
In each case we express $X_1+X_3+X_5$ as a non-negative rational
combination of the $X_i$, possibly with some coefficients equal to zero, and identify the cone it belongs to.
Write $u_1,u_2,u_3,u_4$ for the ordering used for $E_{ij}$, so that $w_{ij}(u_1)$ is maximal and $E_{ij}=\operatorname{conv}\{u_2,u_3\}$. By Proposition~\ref{conelist} we have $X_1=(u_1,0)$, $X_3=(u_2,h_{ij})$ and $X_5=(f_{ij},1)$, while $X_4=(\mathbf{0},1)$ and $X_2$ is the ray attached to $u_4$. Hence in every case
$$X_1+X_3+X_5=(u_1+u_2+f_{ij},\,h_{ij}+1).$$

\medskip
\noindent\textit{i) Mutation $E_{23}$.}
The ordering is $v_1,v_2,v_3,v_4$, which is counterclockwise, so $u_1=v_1$ and $u_2=v_2$, and $X_1+X_3+X_5=(v_1+v_2+f_{23},\,h_{23}+1)$; moreover $X_2=(v_4,0)$ because $w_{23}(v_4)>0$ by Lemma~\ref{lem:sign}. Using $f_{23}=\frac{1}{-h_{23}}(v_3-v_2)$ and the relation $v_3+t_2v_1=a_2v_4$ from~\eqref{curvextrel}, we obtain
$$
X_1+X_3+X_5
  =\Bigl(1+\tfrac{t_2}{h_{23}}\Bigr)X_1
  +\Bigl(\tfrac{a_2}{-h_{23}}\Bigr)X_2
  +\Bigl(1+\tfrac{1}{h_{23}}\Bigr)X_3.
$$
Since $0<t_2\leq 1$, all coefficients are non-negative, so $X_1+X_3+X_5\in\sigma_{123}$.

\medskip
\noindent\textit{Mutation $E_{41}$.}
By Lemma~\ref{lem:sign} the maximum of $w_{41}$ on $P$ is attained at $v_2$, so the ordering is $v_2,v_1,v_4,v_3$, which is clockwise, and $u_1=v_2$, $u_2=v_1$; hence $X_1+X_3+X_5=(v_2+v_1+f_{41},\,h_{41}+1)$.
That maximum being a positive integer, Lemma~\ref{lem:sign} gives $w_{41}(v_2)=\frac{m_{12}-\Delta_0}{-h_{41}}\geq 1$, that is $-h_{41}\leq m_{12}-\Delta_0$.
Hence, with $t_1=m_{41}/m_{12}$ and $a_1=\Delta_0/m_{12}$,
$$t_1=\frac{h_{41}^2}{m_{12}}\leq (-h_{41})\,\frac{m_{12}-\Delta_0}{m_{12}}=(-h_{41})(1-a_1)<-h_{41},$$
the last inequality because $1-a_1<1\leq -h_{41}$. Using $f_{41}=\frac{1}{-h_{41}}(v_4-v_1)$ and $v_4+t_1v_2=a_1v_1$ from~\eqref{curvextrel} this yields
$$
X_1+X_3+X_5
  =\Bigl(1+\tfrac{t_1}{h_{41}}\Bigr)X_1
  +\Bigl(1+\tfrac{1-a_1}{h_{41}}\Bigr)X_3
  +a_1 X_4,
$$
with non-negative coefficients, so $X_1+X_3+X_5\in\sigma_{134}$.

\medskip
\noindent\textit{ii) Mutation $E_{34}$, assuming $m_{23}<m_{41}-\Delta_0-\Delta_1$.}
By Lemma~\ref{lem:sign}, $w_{34}(v_2)<w_{34}(v_1)$, so the maximum of $w_{34}$ is attained at $v_1$. The ordering is $v_1,v_4,v_3,v_2$, which is clockwise, and $u_1=v_1$, $u_2=v_4$. Hence, $X_1+X_3+X_5=(v_1+v_4+f_{34},\,h_{34}+1)$. Using~\eqref{curvextrel} we find
$$
X_1+X_3+X_5
  =\Bigl(1+\tfrac{t_2}{h_{34}}\Bigr)X_1
  +\Bigl(1+\tfrac{1-a_2}{h_{34}}\Bigr)X_3
  +a_2 X_4,
$$
which lies in $\sigma_{134}$, since $t_2\leq 1$ by Lemma~\ref{lem:F1moritrain}, $a_2\leq 1$ by the remark following it, and $-h_{34}\geq 1$.

\medskip
\noindent\textit{Mutation $E_{12}$, assuming $m_{23}<m_{41}-\Delta_0-\Delta_1$.}
By Lemma~\ref{lem:sign}, $w_{12}(v_3)<w_{12}(v_4)$, so the maximum of $w_{12}$ is attained at $v_4$. The ordering is $v_4,v_1,v_2,v_3$, which is counterclockwise, and $u_1=v_4$, $u_2=v_1$. Hence $X_1+X_3+X_5=(v_4+v_1+f_{12},\,h_{12}+1)$.
Lemma~\ref{lem:sign} gives $w_{12}(v_4)=\frac{m_{41}-\Delta_0}{-h_{12}}\geq 1$, hence $-h_{12}\leq m_{41}-\Delta_0$ and therefore $t_1(-h_{12})=\frac{m_{41}}{-h_{12}}>1$.
Using~\eqref{curvextrel} to eliminate $v_2$, we obtain
$$
X_1+X_3+X_5
  =\Bigl(1+\tfrac{1}{t_1 h_{12}}\Bigr)X_1
  +\Bigl(1+\tfrac{t_1-a_1}{t_1 h_{12}}\Bigr)X_3
  +\tfrac{a_1}{t_1}X_4,
$$
which lies in $\sigma_{134}$: the first coefficient is non-negative by the inequality just proved, the third is positive, and the second is non-negative because $t_1-a_1\leq t_1\leq t_1(-h_{12})$ if $t_1\geq a_1$, while it exceeds $1$ if $t_1<a_1$.

\medskip
\noindent\textit{iii) Mutation $E_{34}$, assuming $m_{41}-\Delta_0-\Delta_1<m_{23}$.}
By Lemma~\ref{lem:sign}, $w_{34}(v_1)<w_{34}(v_2)$, so the maximum of $w_{34}$ is attained at $v_2$, the ordering is $v_2,v_3,v_4,v_1$, which is counterclockwise, and $u_1=v_2$, $u_2=v_3$. Hence $X_1+X_3+X_5=(v_2+v_3+f_{34},\,h_{34}+1)$, and $X_2=(v_1,0)$ because $0\leq w_{34}(v_1)$ by Lemma~\ref{lem:sign}. By Lemma~\ref{lem:sign} the maximum of $w_{41}$ on $P$ is $w_{41}(v_2)=\frac{m_{12}-\Delta_0}{-h_{41}}$, so $\Delta_0\leq m_{12}-|h_{41}|$ and $a_1<1$.
Applying $w_{34}$ to $v_4+t_1v_2=a_1v_1$  gives $t_1w_{34}(v_2)=a_1w_{34}(v_1)-h_{34}$, where $w_{34}(v_1)<w_{34}(v_2)$ are non-negative integers by Lemma~\ref{lem:sign}. Hence $w_{34}(v_2)\geq1$, $w_{34}(v_1)\leq w_{34}(v_2)-1$ and
$t_1<1+\frac{-h_{34}-1}{w_{34}(v_2)}\leq -h_{34}.$
Using~\eqref{curvextrel} we obtain
$$
X_1+X_3+X_5
  =\Bigl(1+\tfrac{t_1}{h_{34}}\Bigr)X_1
  +\Bigl(\tfrac{a_1}{-h_{34}}\Bigr)X_2
  +\Bigl(1+\tfrac{1}{h_{34}}\Bigr)X_3,
$$
which lies in $\sigma_{123}$, since $t_1<-h_{34}$ and $a_1>0$.

\medskip
\noindent\textit{Mutation $E_{12}$, assuming $m_{41}-\Delta_0-\Delta_1<m_{23}$.}
By Lemma~\ref{lem:sign}, $w_{12}(v_4)<w_{12}(v_3)$, so the maximum of $w_{12}$ is attained at $v_3$, the ordering is $v_3,v_2,v_1,v_4$, which is clockwise, and $u_1=v_3$, $u_2=v_2$. We treat the two subcases separately.
If $0\leq w_{12}(v_4)$, then $X_2=(v_4,0)$ and $X_1+X_3+X_5=(v_3+v_2+f_{12},\,h_{12}+1)$. Applying $w_{12}$ to $v_3+t_2v_1=a_2v_4$ gives the identity $(-h_{12})t_2=w_{12}(v_3)-a_2w_{12}(v_4)$, and since $w_{12}(v_3)\geq 1$, $w_{12}(v_4)\leq w_{12}(v_3)-1$ and $a_2\leq 1$ we get
$$(-h_{12})t_2\geq w_{12}(v_3)\bigl(1-a_2\bigr)+a_2\geq 1.$$
With this, using~\eqref{curvextrel}, we find
$$
X_1+X_3+X_5
  =\Bigl(1+\tfrac{1}{t_2 h_{12}}\Bigr)X_1
  +\Bigl(\tfrac{a_2}{-h_{12}t_2}\Bigr)X_2
  +\Bigl(1+\tfrac{1}{h_{12}}\Bigr)X_3\;\in\;\sigma_{123}.
$$

If $w_{12}(v_4)<0$, the generator $X_2=(v_4+w_{12}(v_4)f_{12},\,w_{12}(v_4))$ has negative height. Using~\eqref{curvextrel} and $t_2\leq 1\leq w_{12}(v_3)$, we derive
$$
X_1+X_3+X_5
  =\Bigl(\tfrac{w_{12}(v_3)-1}{w_{12}(v_3)}\Bigr)X_1
  +\Bigl(\tfrac{a_2}{w_{12}(v_3)}\Bigr)X_2
  +\Bigl(1-\tfrac{t_2}{w_{12}(v_3)}\Bigr)X_3\;\in\;\sigma_{123}.
$$

\medskip
By Theorem~\ref{coxprim}, each primitive relation above is a positive multiple of the wall relation induced by some $[C_{\omega_{ij}}]$.
When $X_1+X_3+X_5\in\sigma_{123}$, the relation takes the form
\begin{equation}
    X_5+(1-a)X_1+(1-c)X_3-bX_2=0,
\end{equation}
and since $\sigma_{123}$ is simplicial this wall relation coincides for every $\omega_{k2}$ with $k\in\{1,3,5\}$, inducing $\widetilde{R}_1$. When $X_1+X_3+X_5\in\sigma_{134}$, the relation takes the form
\begin{equation}
    X_5+(1-a)X_1+(1-b)X_3-cX_4=0,
\end{equation}
and since $\sigma_{134}$ is simplicial it coincides for every $\omega_{k4}$ with $k\in\{1,3,5\}$, inducing $\widetilde{R}_1$.
\end{proof}

We now analyze the extremal contractions $\varphi_{R_\omega}:\tilde{X}\to X(\Sigma_i)$, which in the toric case are read off from wall relations; we follow \cite{R83}.
Let $X$ be a $\qq$-factorial complete toric variety of dimension $n$ and $R=\rr_{\geq 0}[C_\omega]$ an extremal ray, with $\omega=\operatorname{cone}(\rho_1,\dots,\rho_{n-1})\in\Sigma(n-1)$ separating the maximal cones $\operatorname{cone}(\rho_1,\dots,\rho_n)$ and $\operatorname{cone}(\rho_1,\dots,\rho_{n-1},\rho_{n+1})$, so that the $n+1$ primitive vectors satisfy a wall relation
$$\rho_{n+1}+a_n\rho_n+\sum_{i=1}^{n-1}a_i\rho_i=0, \qquad a_n\in\qq^+,\; a_i\in\qq.$$
Reordering so that $a_i<0$ for $1\leq i\leq\alpha$, $a_i=0$ for $\alpha+1\leq i\leq\beta$, and $a_i>0$ for $\beta+1\leq i\leq n+1$, the contraction $\varphi_{R_\omega}$ is of fiber type if $\alpha=0$, divisorial if $\alpha=1$, and small if $\alpha>1$.
In all cases, $\operatorname{codim}\operatorname{Locus}(R_\omega)=\alpha$, $\dim\varphi_{R_\omega}(\operatorname{Locus}(R_\omega))=\beta-1$, and the exceptional locus is $V(\rho_1,\ldots,\rho_\alpha)$. For the fiber-type and divisorial cases $(\alpha\leq 1)$ one has $\dim\varphi_{R_\omega}(\text{Locus}(R_{\omega}))=\beta-1$. For a small contraction $(\alpha>1)$ the locus $V(\rho_1,\dots,\rho_\alpha)$ is contracted to a point.

\medskip
We apply this to the two extremal rays of $\widetilde{X}$.
For $\widetilde{R}_2$, the primitive relation from Proposition~\ref{flipray} is
\[
X_2+X_4 = \begin{cases} aX_1+bX_5+(1-b)X_4 & b\leq 1,\\ \tfrac{a}{b}X_1+X_5+(1-\tfrac{1}{b})X_2 & b>1, \end{cases}
\]
where $a\in\qq^{\geq 0}$, $b\in\qq^+$, and $v_4=av_1+bf$.
The wall $\omega_{15}$ divides $\sigma_{125}$ and $\sigma_{145}$.
When $a=0$ (equivalently $w(v_4)=0$) and $b=1$, the relation becomes $X_2+X_4-X_5=0$, giving $\alpha=1$, $\beta=2$: a divisorial contraction $\varphi_{\widetilde{R}_2}:\widetilde{X}\to X(\Sigma_2)$ with exceptional divisor $V(\mathbb{R}_{\geq 0}X_5)$ mapping onto a curve.
When $a>0$, one has $\alpha=\beta=2$, so the contraction is small with exceptional locus $V(\mathbb{R}_{\geq 0}\{X_1,X_5\})$.

\medskip
We now turn to $\widetilde{R}_1$.
By Proposition~\ref{prop:extremal}, the primitive relation takes one of two forms depending on the mutation edge. This way we identify which ray of $\tilde{\Sigma}$ is contracted by $\varphi_{\widetilde{R}_1}$. Recall that the pencil $\mathcal{X}_{t,s}\to\pp^1$ induced by Theorem \ref{thm:petrdeform} has the explicit forms (\ref{eq:defpos}) and (\ref{eq:defneg}), according to the mutation data $(w_{ij},f_{ij})$.

\begin{theorem}\label{contractionteo}
Let $P\in[P_{\ff_k}]$ for $k\leq 1$ with no pair of parallel edges, in standard orientation with $P\ncong P_{\ff_k}$.
For every edge $E_{ij}$, the contraction $\varphi_{\widetilde{R}_1}:\widetilde{X}\to X(\Sigma_1)$ is divisorial with exceptional locus contracting to a point.
Moreover, the pushforward ${\varphi_{\widetilde{R}_1}}_*(\mathcal{X}_{s,t})$ gives an induced $\pp^1_{s,t}$-equivariant deformation
$${\varphi_{\widetilde{R}_1}}_*(\mathcal{X}_{s,t})\longrightarrow\pp^1_{s,t}.$$
\end{theorem}

\begin{proof}

\noindent We split the proof in cases.

\textit{Case 1: $X_1+X_3+X_5\in\sigma_{123}$.}
Proposition~\ref{prop:extremal} yields the primitive relation
\begin{equation}
  X_5+(1-a)X_1+(1-c)X_3-bX_2=0,
\end{equation}
with $0\leq a,c<1$ and $b>0$, so $\alpha=\beta=1$: a divisorial contraction with exceptional divisor $E=V(\mathbb{R}_{\geq 0}X_2)$ and $\dim(\varphi_{\widetilde{R}_1}(E))=0$.
We focus on $E_{23}$; the other subcases are analogous. Here $h_{max}=w_{23}(v_1)$, $h_{min}=h_{23}$.
The fan $\Sigma_1$ is obtained from $\widetilde{\Sigma}$ by replacing the walls $\omega_{12}$, $\omega_{23}$, $\omega_{25}$ and their incident maximal cones with $\delta=\operatorname{cone}\{X_1,X_3,X_5\}$ and is generated by
$$Q=\operatorname{conv}\{(v_1,0),(v_2,h_{\min}),(\mathbf{0},1),(f_{23},1)\}\subseteq(N\oplus\mathbb{Z})_{\mathbb{R}}.$$
The morphism $\varphi_{\widetilde{R}_1}$ contracts the ray $\mathbb{R}_{\geq 0}X_2$, so the pushforward satisfies 
\[
{\varphi_{\widetilde{R}_1}}_*(\mathcal{X}_{s,t})=\{sx_1^{h_{\max}}+tx_2^{-h_{\min}}-z_0z_1=0\}\subseteq X(\Sigma_1)\times\pp^1_{s,t},
\]
yielding the $\pp^1_{s,t}$-equivariant diagram
\begin{equation}\label{diagramcontrE}
\begin{tikzcd}
  \mathcal{X}_{s,t} \arrow[rr]\arrow[rd] & & {\varphi_{\widetilde{R}_1}}_*(\mathcal{X}_{s,t}) \arrow[ld]\\
  & \pp^1_{s,t} &
\end{tikzcd}
\end{equation}

\medskip
\noindent\textit{Case 2: $X_1+X_3+X_5\in\sigma_{134}$.}
Proposition~\ref{prop:extremal} yields the primitive relation
\begin{equation}
  X_5+(1-a)X_1+(1-b)X_3-cX_4=0,
\end{equation}
with $0\leq a,b<1$ and $c>0$, so again $\alpha=\beta=1$: a divisorial contraction with exceptional divisor $E=V(\mathbb{R}_{\geq 0}X_4)$ and $\dim(\varphi_{\widetilde{R}_1}(E))=0$.
We focus on $E_{41}$; the other subcases are analogous. Here $h_{max}=w_{14}(v_2)$, $h_{min}=h_{41}$. The fan $\Sigma_1$ is obtained from $\widetilde{\Sigma}$ by removing $\omega_{14}$, $\omega_{34}$, $\omega_{45}$ and adding $\delta=\operatorname{cone}\{X_1,X_3,X_5\}$, and is generated by
$$Q=\operatorname{conv}\{(v_2,0),(v_3,0),(v_1,h_{\min}),(f_{41},1)\}$$
if $0\leq w_{41}(v_3)$, or by
$$Q=\operatorname{conv}\{(v_2,0),(v_3+w_{41}(v_3)f_{41},w_{41}(v_3)),(v_1,h_{\min}),(f_{41},1)\}$$
if $w_{41}(v_3)<0$.
The morphism contracts the ray $\mathbb{R}_{\geq 0}X_4$, so in both subcases 
\[
{\varphi_{\widetilde{R}_1}}_*(\mathcal{X}_{s,t})=\{sx_1^{h_{max}}x_4^{w_{41}(v_3)}+tx_2^{-h_{\min}}-z_1=0\}\subseteq X(\Sigma_1)\times\pp^1_{s,t}
\]
if $0\leq w_{41}(v_3)$, or
\[
{\varphi_{\widetilde{R}_1}}_*(\mathcal{X}_{s,t})=\{sx_1^{h_{max}}+tx_2^{-h_{\min}}x_4^{-w_{41}(v_3)}-z_1=0\}\subseteq X(\Sigma_1)\times\pp^1_{s,t}
\]
if $w_{41}(v_3)<0$. This yields an equivariant diagram as \ref{diagramcontrE}.
In both cases the deformation over $\pp^1_{s,t}$ is induced as in Theorem~\ref{thm:petrdeform}.
\end{proof}

Before constructing all surfaces $X_P$, where $P\in [P_{\ff_1}]$, we record two examples. One that begins from the polygon $P_{\ff_1}=\operatorname{conv}\{(1,0),(0,1),(-1,-1),(0,-1)\}$, which is the only case not covered by Proposition \ref{prop:extremal}.
\begin{example}
Following the standard orientation, we establish the ordering $v_1=(1,0)$, $v_2=(0,1)$, $v_3=(-1,-1)$ and $v_4=(0,-1)$ for $P_{\ff_1}$. We describe the mutations in the sense of Proposition \ref{prop:extremal}. 

\begin{itemize}
    \item The mutation with respect the edge $E_{23}=\operatorname{conv}\{v_2,v_3\}$ has mutation data $w_{23}=(2,-1)$ and $f=\frac{1}{-h_{23}}(v_3-v_2)=(-1,-2)$, for which we compute $w_{23}(v_1)=2$, $w_{23}(v_4)=1$ and $h_{23}=-1$. The mutation results in the polygon
$$P^\prime=\operatorname{conv}\{(1,0),(0,1),(-1,-3),(-1,-4)\}.$$
Which by Theorem \ref{thm:petrdeform} gives us the flat family, $$\mathcal{X}_{s,t}=\{sx_1^2x_4+tx_2-z_0z_1=0\}\subseteq \tilde{X} \times \pp^1_{s,t}$$
having $\mathcal{X}_{0,1}=\ff_1$ and $\mathcal{X}_{1,0}=\operatorname{Bl}_p\pp(1,1,2^2)$, with $p$ a smooth toric invariant point. By Proposition \ref{prop:extremal} and Theorem \ref{contractionteo}, the contraction $\varphi_{\widetilde{R}_1}$ has exceptional locus $E=V(X_2)$ or in Cox coordinates $V(x_4)$. Additionally, $\dim\varphi_{\widetilde{R}_1}(E)=0$. So, $\tilde{X}=\operatorname{Bl}_p\pp(1,1,1,2^2)$.
The induced contraction on the divisor $\mathcal{X}_{s,t}$ as seen in Diagram \ref{diagramcontrE} gives us the flat family ${\varphi_{\widetilde{R}_1}}_*(\mathcal{X}_{s,t})\times\pp^1_{s,t}\subseteq\pp(1,1,1,2^2) \times \pp^1_{s,t}$. This is the markovian degeneration $\pp^2\leadsto \pp(1,1,2^2)$ induced by the mutation of the edge $E_{23}$ of the Fano triangle $$P_0=\operatorname{conv}\{(1,0),(0,1),(-1,-1)\}.$$ See  \cite[Example 6.3]{H13} for explicit equations for the families in the deformation class of $\pp^2$.

\item The mutation with respect $E_{34}$ has mutation data $w_{34}=(0,1)$ and $f_{34}=\frac{1}{-h_{34}}(v_4-v_3)=(1,0)$ for which we compute $w_{34}(v_2)=1$, $w_{34}(v_1)=0$ and $h_{34}=-1$. The mutation results in the polygon
$$P^\prime=\operatorname{conv}\{(0,1),(-1,-1),(1,0),(1,1)\}.$$
Which by Theorem \ref{thm:petrdeform} gives us the flat family, $$\mathcal{X}_{s,t}=\{sx_1+tx_2-z_0z_1=0\}\subseteq \tilde{X} \times \pp^1_{s,t}$$
having $\mathcal{X}_{0,1}=\ff_1$. The primitive relation $X_1+X_3+X_5-X_4=0$ holds in $\widetilde{X}$. Then, the contraction $\varphi_{\widetilde{R}_1}$ has exceptional locus $E=V(X_4)$ or in Cox coordinates $V(z_0)$. In this case we obtain $\widetilde{X}=\operatorname{Bl}_p\pp^3$. The induced family ${\varphi_{\widetilde{R}_1}}_*(\mathcal{X}_{s,t})\times\pp^1_{s,t}\subseteq\pp^3 \times \pp^1_{s,t}$ is just a smooth deformation of $\pp^2$. This deformation is not induced by a mutation of Fano polygons. 

\item The mutation with respect to the edge $E_{41}$ is completely analogous to the $E_{34}$ case. Similarly, the mutation with respect to $E_{12}$ is analogous to $E_{23}$.
\end{itemize}
\label{f1mut}
\end{example}

\begin{example}
Let $P=\operatorname{conv}\{(1,0),(0,1),(-1,-3),(-1,-4)\}$. We make the numbering $v_1=(-1,-4)$, $v_2=(1,0)$, $v_3=(0,1)$ and $v_4=(-1,-3)$. Such that $v_1$ and $v_4$ correspond to the extremal curves of $X_P=\operatorname{Bl}_p\pp(1,1,2^2)$.
\begin{itemize}
    \item The mutation with respect the edge $E_{23}$ has mutation data $w_{23}=(-1,-1)$ and $f_{23}=\frac{1}{-h_{23}}(v_3-v_2)=(-1,1)$, for which we compute $w_{23}(v_1)=5$, $w_{23}(v_4)=4$ and $h_{23}=-1$. The mutation results in the polygon
$$P^\prime=\operatorname{conv}\{(-1,-4),(1,0),(-5,1),(-6,1)\}$$ that gives us the flat family, $$\mathcal{X}_{s,t}=\{sx_1^5x_4^4+tx_2-z_0z_1=0\}\subseteq \tilde{X} \times \pp^1_{s,t}$$
having $\mathcal{X}_{0,1}=\operatorname{Bl}_p\pp(1,1,2^2)$ and $\mathcal{X}_{1,0}=\operatorname{Bl}_p\pp(1,2^2,5^2)$. By Theorem \ref{contractionteo}, the contraction of $V(x_4)$ by $\varphi_{\widetilde{R}_1}$ induces the deformation $\pp(1,1,2^2)\leadsto \pp(1,2^2,5^2)$. In this situation the curve $\Gamma_1=V(x_4)\cap \mathcal{X}_{0,1}$ which is the $(-1)$ curve of $\operatorname{Bl}_p\pp(1,1,2^2)$ deforms onto $V(x_4)\cap \mathcal{X}_{1,0}$, the $(-1)$ curve of $\operatorname{Bl}_p\pp(1,2^2,5^2)$.

\item The mutation with respect $E_{34}=\operatorname{conv}\{v_3,v_4\}$ has mutation data $w_{34}=(4,-1)$ and $f_1=\frac{1}{-h_{34}}(v_4-v_3)=(-1,-4)$ for which we compute $w_{34}(v_2)=4$, $w_{34}(v_1)=0$ and $h_{34}=-1$. The mutation results in the polygon
$$P^\prime=\operatorname{conv}\{(1,0),(0,1),(-1,-4),(-3,-16)\}$$
that gives us the flat family, $$\mathcal{X}_{s,t}=\{sx_1^4+tx_2-z_0z_1=0\}\subseteq \tilde{X} \times \pp^1_{s,t}.$$
The surface $\mathcal{X}_{1,0}$ is the W-blowup of $\pp(1,1,2^2)$ over $\frac{1}{4}(1,1)$ with exceptional divisor $\Gamma^\prime_1$ having the form $[\binom{2}{1}]-(1)-[\binom{4}{1}]$. Since, this situation falls on the condition $w_{34}(v_1)<w_{34}(v_2)$ of Lemma \ref{lem:sign}, by Proposition \ref{prop:extremal} and Theorem \ref{contractionteo}, the contraction of $V(x_4)$ by $\varphi_{\widetilde{R}_1}$ induces a deformation of weighted projective planes $\pp(1,1,3)\leadsto \pp(1,3,4^2)$. Indeed, this deformation is reflected by a mutation of Fano triangles. In this situation the curve $\Gamma_0=V(x_4)\cap \mathcal{X}_{0,1}$ which has the structure $(1)-[\binom{2}{1}]$ in $\operatorname{Bl}_p\pp(1,1,2^2)$ deforms onto $V(x_4)\cap \mathcal{X}_{1,0}$, which preserves the same singularities.

\item For the mutation with respect $E_{41}$, we make the change of coordinates $M = (\begin{smallmatrix} -1 & 0 \\ 0 & 1 \end{smallmatrix})$ to make the mutation following Convention \ref{conv:orientation}. For $P=\operatorname{conv}\{(-1,0),(0,1),(1,-3),(1,-4)\}$, the mutation with respect to $E_{41}=\operatorname{conv}\{(1,-3),(1,-4)\}$ has mutation data $w_{41}=(-1,0)$ and $f_{41}=(0,1)$. The mutation in this coordinates is given by 
$$P^\prime=\operatorname{conv}\{(-1,0),(1,-4),(0,1),(-1,1)\}$$
that gives us the flat family, $$\mathcal{X}_{s,t}=\{sx_1+tx_2-z_0z_1=0\}\subseteq \tilde{X} \times \pp^1_{s,t}$$
having $\mathcal{X}_{0,1}=\operatorname{Bl}_p\pp(1,1,2^2)$. This deformation results in just a change of coordinates. But, in general the deformation associated to $E_{41}$ is a change of the extremal neighborhood $\Gamma_0$ to a successive mk2A in a Mori train over some cyclic quotient singularity $\frac{1}{\Delta_i}(1,\Omega_i)$.
\item The mutation with respect $E_{12}$ is the reverse mutation of the edge $E_{23}$ of the Example \ref{f1mut}.
\end{itemize}
\begin{figure}[h]
    \centering
    \includegraphics[height=5cm, width=7cm]{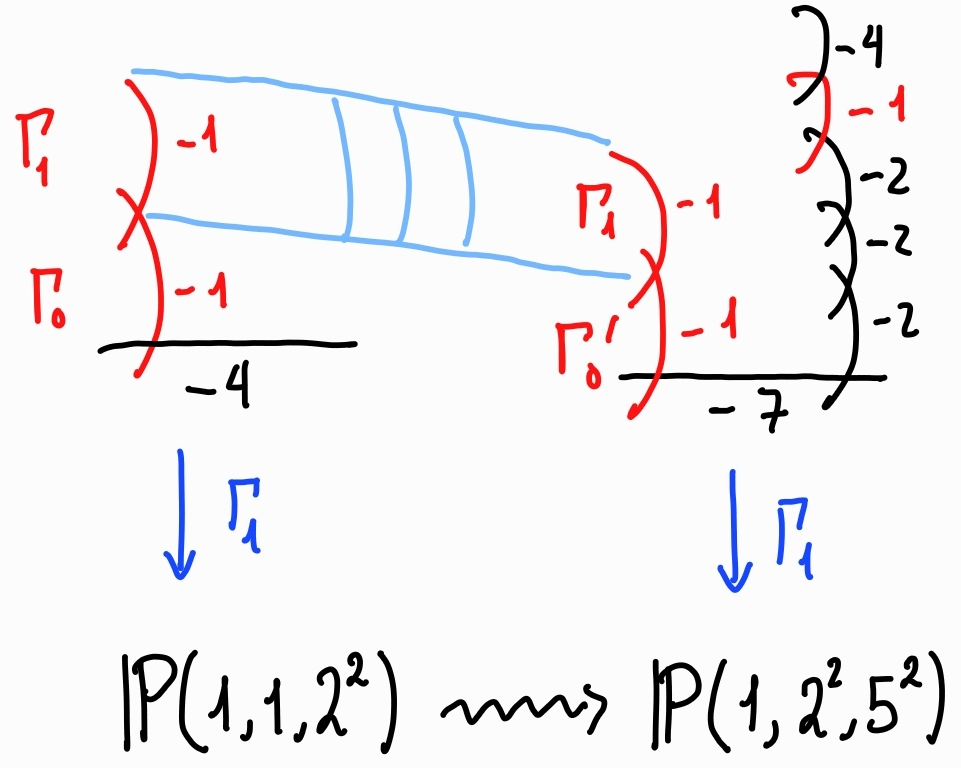}
    \caption{Mutation of $E_{23}$}
    \label{fig:Ex1}
\end{figure}

\label{ex114}
\end{example}
By Proposition \ref{prop:extremal}, the pencil $\mathcal{X}_{s,t}\to\pp^1$ falls into exactly two cases according to whether $X_1+X_3+X_5$ lies in $\sigma_{123}$ or in $\sigma_{134}$.
In the first case, which occurs for $E_{23}$ and for $E_{34}$, $E_{12}$ when $m_{41}-\Delta_0-\Delta_1<m_{23}$, the induced deformation after contracting $\tilde{R}_1$ is controlled by one-step mutations of fake weighted projective planes, and the singularities of $\operatorname{mut}_w(P,F)$ are determined by the following result (Theorem \ref{thm:triangle-mutations}).
In the second case, which occurs for $E_{41}$ and for $E_{34}$, $E_{12}$ when $m_{23}<m_{41}-\Delta_0-\Delta_1$, the pencil is a Mori mutation of an extremal neighborhood (Theorem \ref{thm:Mori-mutations}).

\begin{theorem}[{\cite[Proposition 1.1]{AK16}}]
Let $X=\pp\left(\lambda_0, \lambda_1, \lambda_2\right)$ be a weighted projective plane. Up to reordering of the weights, there exists a one-step mutation to a weighted projective plane $Y$ if and only if $1 / \lambda_0\left(\lambda_1, \lambda_2\right)$ is a $T$-singularity. When this is the case,
$
Y=\pp\left(\lambda_1, \lambda_2, \frac{\left(\lambda_1+\lambda_2\right)^2}{\lambda_0}\right)
$
More generally, there exists a one-step mutation from the fake weighted projective plane $X /(\mathbb{Z} / n)$ to the fake weighted projective plane $Y /\left(\mathbb{Z} / n^{\prime}\right)$ only if $n=n^{\prime}$ and $1 / \lambda_0\left(\lambda_1, \lambda_2\right)$ is a $T$-singularity.
\label{fake}
\end{theorem}
Let $g_j:(\Gamma_j\subset X_P)\to Y_j$ be the contractions of the curves $\Gamma_j$ with
$$Y_0=\pp(\Delta_0,m_{34},m_{23})/(\mathbb{Z}/n_0), \qquad Y_1=\pp(\Delta_1,m_{23},m_{12})/(\mathbb{Z}/n_1),$$
where $n_0=\operatorname{gcd}(m_{23},m_{34})$ and $n_1=\operatorname{gcd}(m_{23},m_{12})$ by Lemma \ref{lem:toric-contractions}. 

\begin{theorem}\label{thm:triangle-mutations}
Let $P\in[P_{\ff_k}]$ be a Fano quadrilateral with no pair of parallel edges, in standard orientation with $P\not\cong P_{\ff_k}$.
For each of the following edges, the divisorial contraction $\varphi_{\widetilde{R}_1}:\widetilde{X}\to X(\Sigma_1)$ identifies ${\varphi_{\widetilde{R}_1}}_*(\mathcal{X}_{s,t})$ with a one-step mutation of a Fano triangle.

\medskip
\begin{enumerate}[label=\roman*)]
\item \textit{Edge $E_{23}$.} Set $P_{23}=\operatorname{mut}_{w_{23}}(P,F_{23})$. Then  $\varphi_{\widetilde R_1*}(\mathcal{X}_{s,t})$ is identified with the mutation of $T=\operatorname{conv}\{v_1,v_2,v_3\}$ along $(w_{23},f_{23})$. The rays $v_1+w_{23}(v_1)f_{23}$ and $v_4+w_{23}(v_4)f_{23}$ correspond to extremal curves $\Gamma_0'$ and $\Gamma_1'$ on $X_{P_{23}}$, admitting contractions $F_i:(\Gamma_i'\subset X_{P_{23}})\to Y_i'$ with
\[
Y_1'=\mathbb P\left(\Delta_1,m_{12},\frac{(m_{12}+\Delta_1)^2}{m_{23}}\right)/(\mathbb Z/n_1),\qquad
Y_0'=\mathbb P(\Delta_0',m_{12},m_{34})/(\mathbb Z/n),
\]
with
\[
\Delta_0'=\frac{m_{34}(\Delta_1+m_{12})+m_{12}(\Delta_0+m_{34})}{m_{23}},
\qquad n=\gcd(m_{12},m_{34}).
 \]

\medskip
\item \textit{Edge $E_{34}$, assuming $m_{41}-\Delta_0-\Delta_1<m_{23}$.} Set $P_{34}=\operatorname{mut}_{w_{34}}(P,F_{34})$. Under $\varphi_{\widetilde R_1}$, the image $\varphi_{\widetilde R_1*}(\mathcal{X}_{s,t})$ is identified with the mutation of $T=\operatorname{conv}\{v_2,v_3,v_4\}$ along $(w_{34},f_{34})$.

The rays $v_1+w_{34}(v_1)f_{34}$ and $v_2+w_{34}(v_2)f_{34}$ correspond to extremal curves $\Gamma_0'$ and $\Gamma_1'$ on $X_{P_{34}}$, admitting contractions $F_i:(\Gamma_i'\subset X_{P_{34}})\to Y_i'$ with
\[
Y_0'=\mathbb P\left(\Delta_0,m_{23},\frac{(m_{23}+\Delta_0)^2}{m_{34}}\right)/(\mathbb Z/n_0),\qquad
Y_1'=\mathbb P(\Delta_1',m_{23},m_{41})/(\mathbb Z/n),
\]
with
\[
\Delta_1'=\frac{m_{41}(\Delta_0+m_{23})+m_{23}(m_{41}-\Delta_1)}{m_{34}},
\qquad n=\gcd(m_{41},m_{23}).
\]

\end{enumerate}
\end{theorem}

\begin{proof}
\noindent\textit{Case i).}
By Proposition~\ref{prop:extremal} and Theorem~\ref{contractionteo}, the pushforward ${\varphi_{\widetilde{R}_1}}_*(\mathcal{X}_{s,t})$ is the family induced by mutating the Fano triangle $T=\operatorname{conv}\{v_1,v_2,v_3\}$ with respect to $(w_{23},f_{23})$, as recorded in diagram~\eqref{diagramcontrE}.
In particular, the ray $v_4+w_{23}(v_4)f_{23}$ gives a curve $\Gamma_1^\prime$ of $X_{P_{23}}$ with contraction $F_1:(\Gamma_1^\prime\subset X_{P_{23}})\to Y_1^\prime$.
Since $\frac{1}{m_{23}}(\Delta_1,m_{12})$ is a T-singularity, Theorem~\ref{fake} gives $Y_1^\prime=\pp(\Delta_1,m_{12},(m_{12}+\Delta_1)^2/m_{23})/(\mathbb{Z}/n_1)$.
Setting $\mu_1=(m_{12}+\Delta_1)^2/m_{23}$, the ray $v_1+w_{23}(v_1)f_{23}$ of $\Gamma_0'$ meets the two adjacent cones of multiplicities $\mu_1$ and $m_{41}$. So the wall relation induced by $\Gamma_0^\prime$ gives
$$\mu_1(v_4+w_{23}(v_4)f_{23})+m_{41}v_1=\Delta_0'(v_1+w_{23}(v_1)f_{23}).$$
As expected, applying $w_{23}$ to both sides yields $(\Gamma_0^\prime)^2=-\Delta_0^\prime/(\mu_1m_{41})$. Hence $Y_0^\prime$ is the toric surface with rays $v_1$, $v_2$ and $v_4+w_{23}(v_4)f_{23}$, whose three cones have multiplicities $m_{12}$, $m_{34}$ and $\Delta_0^\prime$, so $Y_0^\prime=\pp(\Delta_0^\prime,m_{12},m_{34})/(\mathbb Z/n)$ by Lemma~\ref{lem:toric-contractions}.

\medskip
\noindent\textit{Case ii).}
This is analogous to Case i). Indeed, ${\varphi_{\widetilde{R}_1}}_*(\mathcal{X}_{s,t})$ is identified with the mutation of $T=\operatorname{conv}\{v_2,v_3,v_4\}$ along $(w_{34},f_{34})$.
The ray $v_1+w_{34}(v_1)f_{34}$ gives a curve $\Gamma_0^\prime$ with contraction $F_0:(\Gamma_0^\prime\subset X_{P_{34}})\to Y_0^\prime$, and Theorem~\ref{fake} gives $Y_0^\prime=\pp(\Delta_0,m_{23},(m_{23}+\Delta_0)^2/m_{34})/(\mathbb{Z}/n_0)$.
Setting $\mu_0=(m_{23}+\Delta_0)^2/m_{34}$, the wall relation induced by $\Gamma_1^\prime$ gives
$$\mu_0(v_1+w_{34}(v_1)f_{34})+m_{12}v_2=\Delta_1'(v_2+w_{34}(v_2)f_{34}).$$
Applying $w_{34}$ to both sides yields $(\Gamma_1^\prime)^2=-\Delta_1^\prime/(\mu_0 m_{12})$, and the structure for $Y_1^\prime$ follows as in Case i).
\end{proof}
\medskip

Now we prove a similar result for the mutation of the edges $E_{41}$ and $E_{34}$ satisfying $m_{23}<m_{41} - \Delta_1-\Delta_0$. Its associated deformation is expressed in the language of universal families of extremal neighborhoods (Mori mutations), as in \cite[Section 3.4]{HTU17} and \cite[Section 5]{UZ24}.

\begin{theorem}\label{thm:Mori-mutations}
Let $P\in[P_{\ff_k}]$ be a Fano quadrilateral with no pair of parallel edges, in standard orientation with $P\not\cong P_{\ff_k}$.
\begin{enumerate}[label=\roman*)]
\item Set $P_{41}=\operatorname{mut}_{w_{41}}(P,F_{41})$. Then the pencil $\mathcal{X}_{s,t}\to\pp^1$ is identified with a Mori mutation of the neighborhood $\Gamma_0$ containing the singularities of indices $|h_{41}|$ and $|h_{12}|$.
\item If $m_{23}<m_{41}-\Delta_0-\Delta_1$, set $P_{34}=\operatorname{mut}_{w_{34}}(P,F_{34})$. Then the pencil $\mathcal{X}_{s,t}\to \pp^1$ is identified with a Mori mutation of the neighborhood $\Gamma_1$ containing the singularities of indices $|h_{34}|$ and $|h_{41}|$.
\end{enumerate}
\end{theorem}
\begin{proof}
For i) write $n_0=|h_{41}|$, $n_1=|h_{12}|$ and $\delta_0=\frac{m_{12}+m_{41}-\Delta_0}{n_0n_1}$ as in
Lemma~\ref{lem:F1moritrain}, and let $\frac1{n_1^2}(1,n_1a_1-1)$ be the singularity of $\tau_{12}$ and
$Y_0=\frac1{\Delta_0}(1,\Omega_0)$ the contraction of $\Gamma_0$. Since
$\Delta_0=n_0^2+n_1^2-\delta_0n_0n_1$, the neighborhood $(\Gamma_0\subset X_P)\to Y_0$ is the k2A
$v^0=\frac1{\Delta_0}(n_1^2,n_0^2)$ of \cite[Proposition~2.4]{HTU17}, in the notation of which
$(m_0,m_1,\delta,\Delta)=(n_0,n_1,\delta_0,\Delta_0)$.

By~\eqref{mutpol}, it follows that $$\mathcal{V}(P_{41})=\{v_1,\,v_2,\,v_2+w_{41}(v_2)f_{41},\,v_3+w_{41}(v_3)f_{41}\},$$
so $\tau_{12}$ is unchanged and $\tau_{41}$ is replaced by
$\tau_{41}'=\operatorname{cone}\bigl(v_2,v_2+w_{41}(v_2)f_{41}\bigr)$, both adjacent to
$\Gamma_0'=V(\rr_{\geq0}v_2)$. As $f_{41}$ is primitive with $w_{41}(f_{41})=0$ we have
$|\det(u,f_{41})|=|w_{41}(u)|$ for all $u\in N$, so by Lemma~\ref{lem:sign} and $n_0^2=m_{41}$
$$\mult(\tau_{41}')=w_{41}(v_2)^2=\Bigl(\tfrac{m_{12}-\Delta_0}{n_0}\Bigr)^2=n_2^2,
\qquad n_2:=\delta_0n_1-n_0 .$$
Hence $(\Gamma_0'\subset X_{P_{41}})\to Y_0$ is the k2A with indices $(n_1,n_2)$, that is
$v^1=\frac1{\Delta_0}(n_2^2,n_1^2)$.

Finally, $\tau_{12}$ is untouched by the mutation, so $\mathcal X_{s,t}$ is $\qq$-Gorenstein trivial at
$\frac1{n_1^2}(1,n_1a_1-1)$ and smooths the remaining singularity, $Y_0$ being fixed. Since the
$\qq$-Gorenstein deformation space of a Wahl singularity is smooth of dimension one, near $[0:1]$ and
near $[1:0]$ it is therefore the one-parameter deformation of the last assertion of
\cite[Proposition~2.4]{HTU17} for $v^0$ and for $v^1$, whose general fiber is in both cases the k1A
over $Y_0$ with singularity $\frac1{n_1^2}(1,n_1a_1-1)$. Thus $\mathcal X_{s,t}\to\pp^1_{s,t}$ is the
degeneration of that k1A into the two k2A $v^0$ and $v^1$ in its universal family $\mathbb{U}\to M$.

For ii) replace $(\Gamma_0,\tau_{41},\tau_{12},\Delta_0,m_{12},m_{41})$ by
$(\Gamma_1,\tau_{34},\tau_{41},\Delta_1,m_{41},m_{34})$; the hypothesis
$m_{23}<m_{41}-\Delta_0-\Delta_1$ places $E_{34}$ in the case $X_1+X_3+X_5\in\sigma_{134}$ of
Proposition~\ref{prop:extremal}, so that the contracted ray is again $\rr_{\geq0}X_4$.
\end{proof}

\begin{remark}
The mutation of $E_{12}$ behaves analogously to that of $E_{34}$. By Proposition~\ref{prop:extremal} and Lemma~\ref{lem:sign}, the condition $w_{12}(v_3)<w_{12}(v_4)$ is equivalent to $m_{23}<m_{41}-\Delta_0-\Delta_1$.
When this holds, the family $\mathcal{X}_{s,t}\to\pp^1$ is a Mori mutation of $\Gamma_0$ with singularities of indices $|h_{12}|$ and $|h_{41}|$. Otherwise, it is a mutation of Fano triangles as in Theorem~\ref{thm:triangle-mutations}.
In general, the mutation of $E_{12}$ is obtained as the reverse mutation of some $P'\in[P_{\ff_k}]$ along an edge $E_{jk}'\neq E_{12}'$.
\end{remark}

\subsection{Mutations on M-resolutions}\label{sub:M-res}
A non-Du Val T-singularity of the form $\frac{1}{dn^2}(1,dna-1)$ carries the crepant M-resolution 
\begin{equation}\label{eq:canonical}
\Big[\binom{n}{a}\Big]-(1)_0-\cdots-(1)_0-\Big[\binom{n}{a}\Big]\to \frac{1}{dn^2}(1,dna-1). 
\end{equation}
From the toric perspective, suppose the singularity germ of $\frac{1}{dn^2}(1,dna-1)$ is $X_\sigma=\operatorname{Spec}(\mathbb{C}[M\cap \sigma^\vee])$ for a rational polyhedral cone $\sigma$. The previous M-resolution is determined by a canonical refinement of $\sigma$.

\begin{lemma}\label{lem:canonical}
Let $\sigma\subset N_\mathbb{R}$ a rational polyhedral cone defining the cyclic quotient singularity $\frac{1}{dn^2}(1,dna-1)$. Suppose $\sigma=\operatorname{cone}(v,w)$ with $v,w$ primitive vectors in $N$ and define $v_k:=v+\frac{k}{d}(w-v)$ for $0\leq k\leq d$. Then, the fan induced by the cones $\sigma_k:=\operatorname{cone}(v_k,v_{k+1})$ for $k\leq d-1$ determines the map $X(\Sigma)\to X_\sigma$ which represents the chain \ref{eq:canonical}. 
\end{lemma}
\begin{proof}
First, we begin noting that the cones $\sigma_k$ are rational. From \cite[Proposition 3.9]{AK16} if $E=\operatorname{conv}(v,w)$, then $|E\cap N|-1=dn$ implying that $\frac{1}{d}(w-v)\in N$. Moreover, from \cite[Proposition 3.2]{AK14}, it follows that the vectors $v_k$ are primitive in $N$ and the singularity type of $\sigma_k$ is $\frac{1}{n^2}(1,na-1)$. The last assertion follows directly by making the change of coordinates $v=(0,1)$ and $w=(dn^2,1-dna)$.

Now, let $\Gamma_k$ be the T-invariant curve associated to $\mathbb{R}_{\geq 0}v_k$. From the wall relation: $$\frac{1}{n^2}v_{k-1}-\frac{2}{n^2}v_k+\frac{1}{n^2}v_{k+1}=0$$ for $1\leq k\leq d-1$, we derive that $\Gamma_k^2=-\frac{2}{n^2}$. This implies $\Gamma_k\cdot K_X=0$ for a projective surface $X$ such that $X(\Sigma)\subset X$.
\end{proof}

For a toric del Pezzo surface $X_P$ with T-singularities, we consider the map $\bar{X}_P\to X_P$ arising after partially resolving the singularities of the form $\frac{1}{dn^2}(1,dna-1)$ with $d>1$. This construction follows the gluing of toric maps as described in Lemma \ref{lem:canonical}. We denote by $\overline{\Sigma}_P$ as the spanning fan of $\bar{X}_P$ (resp $\overline{\Sigma}_{Q}$ for $\bar{X}_Q$). Given mutation data $(w,f)$ and $Q=\operatorname{mut}_w(P,F)$, there exists a deformation $\pi:\mathcal{X}\to \pp^1$ such that $\pi^{-1}(0)=X_P$ and $\pi^{-1}(\infty)=X_Q$. 
We extend that construction to a deformation $\overline{\pi}:\mathcal{X^\prime}\to \pp^1$ such that ${\overline{\pi}}^{-1}(0)=\bar{X}_P$ and ${\overline{\pi}}^{-1}(\infty)=\bar{X}_Q$ and a morphism $f:\mathcal{X^\prime}\to \mathcal{X}$ satisfying $\overline{\pi}=\pi\circ f$ and the restrictions $f|_{\bar{X}_P}, f|_{\bar{X}_Q}$ coincide with the corresponding M-resolutions. We follow again the construction outlined in Theorem \ref{thm:petrdeform}. 

\begin{lemma}\label{lem:mut-boundary}
Let $P$ be a Fano polygon with $T$-singularities, let $(w,f)$ be mutation data with $F=\operatorname{conv}\{0,f\}$, and let $Q=\operatorname{mut}_w(P,F)$.
Write $P_h$ for the slice of $P$ at height $h=\langle w,\cdot\rangle$, and $P_h=\operatorname{conv}\{l_h,r_h\}$ with $r_h$ the endpoint in the direction of $f$.
Let $\theta\colon N\to N$ be given by $\theta(v)=v+\langle w,v\rangle f$, which is a lattice automorphism since $\langle w,f\rangle=0$.
Then $Q_h=\operatorname{conv}\{l_h,\theta(r_h)\}$ for every $h$.
Consequently every $v\in\overline{\Sigma}_P(1)$ satisfies $v\in\overline{\Sigma}_Q(1)$ or $\theta(v)\in\overline{\Sigma}_Q(1)$, and every $v'\in\overline{\Sigma}_Q(1)$ satisfies $v'\in\overline{\Sigma}_P(1)$ or $\theta^{-1}(v')\in\overline{\Sigma}_P(1)$.
\end{lemma}

\begin{proof}
For $h\geq0$, we have
\[
Q_h=P_h+hF=\operatorname{conv}\{l_h,r_h+hf\}.
\]
For $h<0$ when $G_h$ is non-empty, the factor condition gives
\[
Q_h=G_h=\operatorname{conv}\{l_h,r_h-|h|f\}
=\operatorname{conv}\{l_h,r_h+hf\}.
\]
Since $\langle w,r_h\rangle=h$, in both cases $r_h+hf=\theta(r_h)$, proving the first assertion.
Hence $\partial Q=L\cup\theta(R)$, where $L$ and $R$ are the chains of $\partial P$ determined by $l_h$ and by $r_h$. These two descriptions agree at height $0$ because $\theta$ restricts to the identity there.
As $\theta\in\operatorname{GL}(N)$, it maps each cone over an edge of $R$ to the corresponding cone over $\theta(R)$ by a lattice isomorphism, and hence carries the refinement of Lemma \ref{lem:canonical} to the corresponding
refinement.
The only edges of $\partial P$ lying in both chains are the horizontal ones, which occur only at $h_{\min}$ and $h_{\max}$.
An edge at lattice height $r$ and lattice length $\ell$ spans a cone of index $r\ell$, so an edge of type $\frac{1}{dn^2}(1,dna-1)$ has $\ell=dn$ by \cite[Proposition~3.9]{AK16} and therefore $r=n$.
Mutation replaces $\ell$ by $\ell-r$ at $h_{\min}$ and by $\ell+r$ at $h_{\max}$, that is, it replaces $d$ by $d-1$ and by $d+1$ respectively.
The rays of the fan in Lemma \ref{lem:canonical} over such an edge $\operatorname{conv}\{v,v+dnf\}$ are the equally spaced lattice points $v+knf$ with $0\le k\le d$, so the two lists correspond under the identity and under $\theta$.
\end{proof}

\begin{proposition}\label{prop:mres-deformation}
Let $P$ be a Fano polygon with $T$-singularities, let $(w,f)$ be a mutation data, and let $Q=\operatorname{mut}_w(P,F)$. Let $\pi\colon\mathcal{X}\to\pp^1$ be the pencil of Theorem~\ref{thm:petrdeform}.
Then there is a fan $\Lambda$ refining $\widetilde{\Sigma}$ such that the strict transform $\bar{\mathcal{X}}\subseteq X(\Lambda)\times\pp^1$ of $\mathcal{X}$ is flat over $\pp^1$, the induced family $\overline{\pi}\colon\bar{\mathcal{X}}\to\pp^1$ is a $\qq$-Gorenstein deformation with $\overline{\pi}=\pi\circ\overline{g}$, and the induced morphism $\overline{g}\colon\bar{\mathcal{X}}\to\mathcal{X}$ restricts over $0$ and over $\infty$ to the crepant M-resolutions $\bar{X}_P\to X_P$ and $\bar{X}_Q\to X_Q$.
\end{proposition}

\begin{proof}
Let $\iota_P,\iota_Q\colon N\hookrightarrow\widetilde{N}=N\oplus\zz e_1$ be the saturated embeddings $\iota_P(v)=(v,0)$ and $\iota_Q(v)=(v,\langle w,v\rangle)$, dual to $\mu_M$ and to $\mu$, and put $N_P=\iota_P(N)$ and $N_Q=\iota_Q(N)$.
Let $\varphi=\max_{m\in T}\langle m,\cdot\rangle$ be the support function of $T=\operatorname{conv}\{0,e_1^*,e_1^*-w\}\subseteq\widetilde{M}_\rr$, and let $\Psi$ be the piecewise linear function on $\widetilde{\Sigma}$ with $\Psi\equiv 1$ on $\widetilde{\Sigma}(1)$, so that $\Psi=1$ exactly on $\partial\widetilde{P}$.
Evaluating on the generators of Theorem~\ref{thm:petrdeform} gives $\varphi(\iota_P(p))=\varphi(\iota_Q(q))=0$ and $\varphi(e_1)=\varphi(f+e_1)=1$, so that the three monomials defining $\mathcal{X}$ are $$x^{(m)}=\prod_\rho x_\rho^{\varphi(v_\rho)-\langle m,v_\rho\rangle}$$ with $m\in\mathcal{V}(T)$.

We proceed to construct $\Lambda$. For $v\in N$ primitive put $H_v=\rr\,\iota_P(v)+\rr e_1$ and $H'_v=\rr\,\iota_P(v)+\rr(f+e_1)$, so that
\[
H_v\cap N_{P,\rr}=H'_v\cap N_{P,\rr}=\rr\,\iota_P(v),
\qquad
H_v\cap N_{Q,\rr}=\rr\,\iota_Q(v),
\qquad
H'_v\cap N_{Q,\rr}=\rr\,\iota_Q(\theta(v)).
\]
Let $\mathcal{H}$ consist of all $H_v$ with $v\in\overline{\Sigma}_P(1)\cap\overline{\Sigma}_Q(1)$ together with all $H'_v$ with $v\in\overline{\Sigma}_P(1)$ and $\theta(v)\in\overline{\Sigma}_Q(1)$, which by Lemma~\ref{lem:mut-boundary} accounts for every ray of $\overline{\Sigma}_P$ and every ray of $\overline{\Sigma}_Q$.
Let $\Lambda$ be the common refinement of $\widetilde{\Sigma}$ and of the half-spaces bounded by the planes of $\mathcal{H}$, chosen so that
\begin{equation}\label{eq:Lambda-rays}
\Lambda(1)\smallsetminus\widetilde{\Sigma}(1)\subseteq\bigl\{\iota_P(v)\;:\;v\in\partial P,\ \langle w,v\rangle\ge 0\bigr\}\cup\bigl\{\iota_Q(v')\;:\;v'\in\partial Q,\ \langle w,v'\rangle\le 0\bigr\},
\end{equation}
which is the same constraint satisfied by the rays of $\widetilde{\Sigma}$ itself, and which gives $\Psi(v)=1$ and $\varphi(v)=\max\{0,\langle e_1^*,v\rangle\}=\max\{0,\langle e_1^*-w,v\rangle\}$ for every $v\in\Lambda(1)$.
Since intersecting cones with a subspace commutes with intersecting them with half-spaces, $\Lambda|_{N_P}$ is the refinement of $\widetilde{\Sigma}|_{N_P}=\Sigma_P$ by the lines $H\cap N_{P,\rr}$ with $H\in\mathcal{H}$, which are exactly the rays of $\overline{\Sigma}_P$. So, $\Lambda|_{N_P}=\overline{\Sigma}_P$ and symmetrically $\Lambda|_{N_Q}=\overline{\Sigma}_Q$.
Note that the rays of $\overline{\Sigma}_P$ at negative height are produced by walls rather than by rays of $\Lambda$, since $\operatorname{cone}(\iota_Q(v),e_1)\cap N_{P,\rr}=\rr_{\ge 0}\,\iota_P(v)$ whenever $\langle w,v\rangle<0$.

We proceed to prove that $\overline{\pi}$ gives a $\qq$-Gorenstein deformation with special fibers $\bar{X}_P$ and $\bar{X}_Q$. Since $\bar{\mathcal{X}}$ is the strict transform of the normal projective variety $\mathcal{X}$ under a proper birational morphism, it dominates $\pp^1$. A dominant morphism from a normal projective variety to a smooth curve is flat, so $\overline{\pi}$ is flat.

By \eqref{eq:Lambda-rays} the primitive generator of every ray of $\Lambda$ lies on $\partial\widetilde{P}$, so $\overline{g}$ is crepant, and by the previous paragraph the same holds fibrewise for $\bar{X}_P\to X_P$ and $\bar{X}_Q\to X_Q$; by Lemma~\ref{lem:canonical} these are the M-resolutions \eqref{eq:canonical}.
It remains to see that $\overline{\pi}$ is $\qq$-Gorenstein, which is a local condition at the singular points of its fibres.
Let $P_1,\dots,P_d$ be the Wahl points of $\bar{X}_P$ lying over a singularity $\frac{1}{dn^2}(1,dna-1)$ of $X_P$.
By \cite[Section 3]{BC94}, the $\qq$-Gorenstein versal deformation of that singularity is $\{xy=z^{dn}+a_{d-1}z^{(d-1)n}+\dots+a_0\}/\mu_n$, that of $\bar{X}_P$ near  $P_1,\dots, P_d$ is $\prod_i\operatorname{Def}^{\qq G}(P_i)\cong\cc^d$, and blowing down is the finite surjective morphism sending $(t_1,\dots,t_d)$ to the coefficients of $\prod_i(z^n-t_i)$.
The germ of $\overline{\pi}$ at $P_i$ is by construction induced from $\operatorname{Def}^{\qq G}(P_i)$, hence is $\qq$-Gorenstein, and therefore so is $\overline{\pi}$.
Additionally, $\overline{\pi}=\pi\circ\overline{g}$ holds by construction.
\end{proof}

\subsection{Degenerations of $\ff_0$}\label{subsec:F0}
Toric degenerations $\ff_0\rightsquigarrow X$ with $-K_X$ ample and $\rho(X)=1$ are classified by \cite[Theorem~4.1]{HP10}, they are weighted projective planes $\pp(a^2,b^2,2c^2)$, where $(a,b,c)\in\zz_{>0}^3$ is a solution of the diophantine equation
\begin{equation}\label{eq:Markov0}
  x^2+y^2+2z^2=4xyz.
\end{equation}
By \cite[Theorem~6]{KNP17}, every such surface lies in the mutation-equivalence class $[P_{\ff_0}]$ of the Fano polygon $P_{\ff_0}$. We adopt the convention $a\leq b$ throughout.

The T-singularities of $\pp(a^2,b^2,2c^2)$ have weights
\begin{equation}\label{eq:weights}
  w_a\equiv 4b^{-1}c\pmod{a},\quad
  w_b\equiv 2c^{-1}a\pmod{b},\quad
  w_c\equiv 2a^{-1}b\pmod{c}.
\end{equation}
We set $c':=2ab-c$ and $w_{c'}:=2aw_b-w_c$. The mutations of $\pp(a^2,b^2,2c^2)$ fall into two families.
Mutations fixing $c$: for the T-singularity $\frac{1}{b^2}(a^2,2c^2)$, a one-step mutation gives $\pp(a^2,(4ac-b)^2,2c^2)$; for $\frac{1}{a^2}(b^2,2c^2)$, one gets $\pp(b^2,(4bc-a)^2,2c^2)$.
For the mutations of $c$ (fixing $a$ and $b$), a one-step mutation does not give a deformation $\pp(a^2,b^2,2c^2)$ to $\pp(a^2,b^2,2(c')^2)$ with indecomposable mutation factor, see \cite[Remark 2]{KNP17}. We proceed to describe the geometric process that connects both. 

\begin{lemma}\label{lem:f0small}
Let $(a,b,c)$ be a solution of~\eqref{eq:Markov0} with $a<b$. Then $c'=2ab-c$ satisfies $c'<b$ if and only if $b<c$.
\end{lemma}
\begin{proof}
Rewrite~\eqref{eq:Markov0} as $2z^2-4abz+(a^2+b^2)=0$, a quadratic in $z$ with roots $c$ and $c'=2ab-c$.
By Vieta's formulas, it follows that $c+c'=2ab$ and $cc'=\frac{a^2+b^2}{2}$. Since $a<b$, we have $cc'=\frac{a^2+b^2}{2}<b^2$. Assume $b<c$, then $c'=\frac{cc'}{c}<\frac{b^2}{c}<b$. On the other hand, if $c'<b$. Then $c=2ab-c'>2ab-b=(2a-1)b\geq b$.
\end{proof}
We compute the M-resolution of $X_T=\pp(a^2,b^2,2c^2)$ over the T-singularity $\frac{1}{2c^2}(1,2cw_c-1)$ with $c>1$.  This is the crepant morphism $\bar{X}_T\to \pp(a^2,b^2,2c^2)$ defined at $\ref{eq:canonical}$. Since $d=2$, it has boundary data:
\begin{itemize}
\item if $b<c$:
\begin{equation}\label{eq:mres1}
  \Bigl[\tbinom{a}{w_a}\Bigr]-(1)-
  \Bigl[\tbinom{c}{w_c}\Bigr]-(1)_0-
  \Bigl[\tbinom{c}{w_c}\Bigr]-(1)_--
  \Bigl[\tbinom{b}{w_b}\Bigr];
\end{equation}
\item if $c<b$:
\begin{equation}\label{eq:mres2}
  \Bigl[\tbinom{a}{w_a}\Bigr]-(1)-
  \Bigl[\tbinom{b}{w_b}\Bigr]-(1)_--
  \Bigl[\tbinom{c}{w_c}\Bigr]-(1)_0-
  \Bigl[\tbinom{c}{w_c}\Bigr].
\end{equation}
\end{itemize}
Following Mori's recursion (see \cite[Section 5]{UZ24} for the current notation) to the extremal neighborhoods we produce toric del Pezzo surfaces in both situations.
In chain~\eqref{eq:mres1}, the $(-1)$-curve $\hat\Gamma_0$ forms a k2A extremal neighborhood with Wahl indices $n_0=c$ and $n_1=b$. The number $\delta=n_0n_1|K_{\bar{X}_T}\cdot \Gamma_0|$ satisfies $\delta=cw_b-bw_c=2a$. The Mori recursion gives $n_2=\delta n_1-n_0=2ab-c=c'$, and the corresponding weight is $w_{c'}=2aw_b-w_c$.
The Mori train modification of $\hat\Gamma_0$ produces the toric del Pezzo surface $X_{a,b,c}$ with boundary data
\begin{equation}\label{eq:mut1}
  \Bigl[\tbinom{a}{w_a}\Bigr]-(1)-
  \Bigl[\tbinom{c'}{w_{c'}}\Bigr]-(1)_--
  \Bigl[\tbinom{b}{w_b}\Bigr]-(1)_--
  \Bigl[\tbinom{c}{w_c}\Bigr].
\end{equation}
Performing the same modification on the remaining extremal neighborhood contracts the M-resolution of $\frac{1}{2(c')^2}(1,2c'w_{c'}-1)$ and yields the weighted projective plane $\pp(a^2,b^2,2(c')^2)$.
Analogously, starting from chain~\eqref{eq:mres2}, the same two-step process produces
\begin{equation}\label{eq:mut3}
  \Bigl[\tbinom{a}{w_a}\Bigr]-(1)-
  \Bigl[\tbinom{c}{w_c}\Bigr]-(1)_--
  \Bigl[\tbinom{b}{w_b}\Bigr]-(1)_--
  \Bigl[\tbinom{c'}{w_{c'}}\Bigr],
\end{equation}
and then $\pp(a^2,b^2,2(c')^2)$.

\medskip
The following result identifies the polygon mutation of the cone determining the $\frac{1}{2c^2}(1,2cw_c-1)$ of $\pp(a^2,b^2,2c^2)$ with Mori mutations.

\begin{theorem}\label{thm:F0-Mori-mutation}
Let $(a,b,c)$ be a solution of~\eqref{eq:Markov0} with $a\leq b<c$. Let $T=\operatorname{conv}\{v_1,v_2,v_3\}$ be the triangle of $\pp(a^2,b^2,2c^2)$, oriented so that $\operatorname{cone}(v_2,v_3)$ defines $\tfrac{1}{2c^2}(1,2cw_c-1)$.
Let $w_{23}$ be the primitive inner normal to $E_{23}=\operatorname{conv}\{v_2,v_3\}$ with mutation factor $f_{23}$.
\begin{enumerate}[label=\roman*)]
\item The polygon $Q=\operatorname{mut}_{w_{23}}(T,F_{23})$ is a Fano quadrilateral with toric surface $X_Q=X_{a,b,c}$ and boundary data~\eqref{eq:mut1}.
The mutations along the edges $E_{12}$ and $E_{13}$ of $T$ produce Fano triangles.
\item The $\qq$-Gorenstein deformation $\pi\colon\mathcal{X}\to\pp^1$ induced by the mutation, with $\pi^{-1}(0)=\pp(a^2,b^2,2c^2)$ and $\pi^{-1}(\infty)=X_{a,b,c}$, is a Mori mutation of the k2A neighborhood $\hat\Gamma_0$ in $\bar{X}_T$.
\end{enumerate}
\end{theorem}

\begin{proof}
We orientate $T$ counterclockwise, so that $\det(v_1,v_2)=a^2$,
$\det(v_2,v_3)=2c^2$, $\det(v_3,v_1)=b^2$ and $2c^2v_1+b^2v_2+a^2v_3=0$; all cones are read in this order. By \cite[Corollary~3]{KNP17}, the vertices of $Q$ are
\[
  \mathcal{V}(Q)=\bigl\{v_1,\;v_2,\;v_2+cf_{23},\;v_1+c'f_{23}\bigr\},
\]
ordered counterclockwise. The cones $\operatorname{cone}(v_1,v_2)$,
$\operatorname{cone}(v_2,v_2+cf_{23})$, $\operatorname{cone}(v_2+cf_{23},v_1+c'f_{23})$ and
$\operatorname{cone}(v_1+c'f_{23},v_1)$ have multiplicities $a^2$, $c^2$, $b^2$ and $(c')^2$
respectively. The first is a cone of $T$ and the third is its image under the map
$u\mapsto u+w_{23}(u)f_{23}\in\operatorname{SL}(N)$, so they define the singularities $\frac1{a^2}(1,aw_a-1)$ and $\frac1{b^2}(1,bw_b-1)$ respectively. For the second, we apply
Lemma~\ref{lem:canonical}. Since $v_3-v_2=2cf_{23}$, the refining ray is $v_2+\tfrac{1}{2}(v_3-v_2)=v_2+cf_{23}$ and both cones of the refinement are $\frac1{c^2}(1,cw_c-1)$; the second cone of $Q$ is one of them. A complete fan with four rays is determined up to $\operatorname{GL}_2(\zz)$ by three
consecutive cones. So $X_Q=X_{a,b,c}$ and the fourth cone is
$\frac1{(c')^2}(1,c'w_{c'}-1)$. 

The cones over $E_{12}$ and $E_{13}$ are Wahl, so these mutations have factors of unit length and coincide with the one-step mutations of Theorem~\ref{fake}.
By Equation \eqref{eq:Markov0} they are the Fano triangles of $\pp(b^2,(4bc-a)^2,2c^2)$ and $\pp(a^2,(4ac-b)^2,2c^2)$.

\medskip\textit{Part~ii).}
By i) all singularities of $X_Q$ are Wahl, so $\bar{X}_Q=X_Q$, and
Proposition~\ref{prop:mres-deformation} applied to $T$ and $Q$ gives a $\qq$-Gorenstein deformation $\bar\pi\colon\bar{\mathcal X}\to\pp^1$ with $\bar\pi^{-1}(0)=\bar X_T$ and
$\bar\pi^{-1}(\infty)=X_Q$, together with $\bar g\colon\bar{\mathcal X}\to\mathcal X$ satisfying $\bar\pi=\pi\circ\bar g$. Over $0$ the morphism $\bar g$ contracts the curve labelled $(1)_0$ in \eqref{eq:mres1}, and over $\infty$ it is an isomorphism.

Write $n_0=c$, $n_1=b$, $n_2=c'$ and $\delta=2a$, and let
$\bar{\Gamma}_0=V(\rr_{\geq0}v_3)\subset\bar X_T$ be the curve labelled $(1)_-$ in \eqref{eq:mres1} and $\Gamma_0'=V(\rr_{\geq0}(v_1+c'f_{23}))\subset X_Q$. 
Deleting $\rr_{\geq0}v_3$ from $\overline\Sigma_T$ and $\rr_{\geq0}(v_1+c'f_{23})$ from $\Sigma_Q$
produces the same fan, with rays $v_1,v_2,v_2+cf_{23}$. Hence both curves are contracted to the same cyclic quotient singularity $Y_0=\tfrac1{\Delta_0}(1,\Omega_0)$, that of
$\operatorname{cone}(v_2+cf_{23},v_1)$, where
\[
  \Delta_0=\det(v_2+cf_{23},v_1)=b^2-cc'=n_0^2+n_1^2-\delta n_0n_1 .
\]
Since $c+c'=2ab$ we have $n_2=\delta n_1-n_0$, and
\[
K_{X_Q}\cdot\Gamma_0'=-\frac{\delta}{n_1n_2},
\qquad
K_{\bar{X}_T}\cdot\bar{\Gamma}_0=\frac{\Delta_0-n_0^2-n_1^2}{(n_0n_1)^2}=-\frac{\delta}{n_0n_1}.
\]
Therefore $(\bar{\Gamma}_0\subset\bar{X}_T)\to Y_0$ is the k2A $v^0=\frac{1}{\Delta_0}(n_1^2,n_0^2)$ of
\cite[Proposition~2.4]{HTU17}, in the notation of which
$(m_0,m_1,\delta,\Delta)=(n_0,n_1,\delta,\Delta_0)$, and $(\Gamma_0'\subset X_Q)\to Y_0$ is the k2A $v^1=\frac{1}{\Delta_0}(n_2^2,n_1^2)$, consecutive to it because $n_2=\delta n_1-n_0$.
\end{proof}
\begin{remark}\label{rem:next-mut}
For $c<b$ the boundary data of $\bar X_T$ is \eqref{eq:mres2} and the mutation of $E_{23}$ produces \eqref{eq:mut3}; Part~ii) holds verbatim, the k2A being again the one with $n_0=c$, $n_1=b$, $\delta=2a$ and $n_2=\delta n_1-n_0=c'$.

We now mutate $Q$ along the edge $E=\operatorname{conv}\{v_2,\,v_2+cf_{23}\}$, whose mutation data is again $(w_{23},f_{23})$: the edge $E$ has height $-c$, while the two
remaining vertices $v_1$ and $v_1+c'f_{23}$ both lie at height $c'$. Since $\operatorname{cone}(E)$ is Wahl of index $c$, the edge $E$ collapses and
\cite[Corollary~3]{KNP17} gives
\[
  T':=\operatorname{mut}_{w_{23}}(Q,F_{23})
  =\operatorname{conv}\bigl\{v_1,\;v_2,\;v_1+2c'f_{23}\bigr\},
\]
which induces $\pp(a^2,b^2,2(c')^2)$.
By Lemma \ref{lem:f0small}, $b<c$ implies $c'<b$, so the $E_{23}$ mutation of $T'$ exists and the deformation $\operatorname{mut}_{w_{23}}(T',F_{23}')=Q$ is precisely the inverse of the $E_{23}$ mutation of $T$.
\end{remark}

Theorems \ref{thm:triangle-mutations}, \ref{thm:Mori-mutations} and \ref{thm:F0-Mori-mutation} establish Theorem \ref{ithm:mutations}. We conclude the paper, by describing the mutations of the quadrilateral $Q$.After a $\operatorname{GL}_2(\zz)$-transformation, relabel $v_1^\prime=v_2+cf_{23}$, $v_2^\prime=v_2$, $v_3^\prime=v_1$, and $v_4^\prime=v_1+c'f_{23}$.
Then $Q$ is in standard orientation for $E_{23}^\prime=\operatorname{conv}\{v_2^\prime,v_3^\prime\}$, with $m_{12}=c^2$, $m_{23}=a^2$, $m_{34}=(c')^2$, $m_{41}=b^2$, and $\Delta_0=\Delta_1=cc'-a^2$.
By Lemma~\ref{lem:sign}, $w_{23}^\prime(v_1^\prime)>w_{23}^\prime(v_4^\prime)>-a=w_{23}^\prime(v_2^\prime)=w_{23}^\prime(v_3^\prime)$.
Set $E_{41}^\prime=\operatorname{conv}\{v_4^\prime,v_1^\prime\}$.

\begin{proposition}\label{prop:Q-mutations}
Let $(a,b,c)$ be a solution of~\eqref{eq:Markov0} with $a\leq b<c$, and let $Q$ be the Fano quadrilateral of Theorem~\ref{thm:F0-Mori-mutation},i).
\begin{enumerate}[label=\roman*)]
\item The mutation of $Q$ along $E_{23}^\prime$ is a mutation as in Theorem~\ref{thm:triangle-mutations},i).
\item The mutation of $Q$ along $E_{41}^\prime$ is a Mori mutation of the k2A extremal neighborhood of $\Gamma_0^\prime=V(\rr_{\geq0}v_1^\prime)$, as in Theorem~\ref{thm:Mori-mutations},i).
\item Neither $Q_1=\operatorname{mut}_{w_{23}^\prime}(Q,F_{23}^\prime)$ nor $Q_2=\operatorname{mut}_{w_{41}^\prime}(Q,F_{41}^\prime)$ has a pair of parallel edges.
\end{enumerate}
\end{proposition}

\begin{proof}
Parts i) and ii) follow from Theorem~\ref{thm:triangle-mutations},i) and Theorem~\ref{thm:Mori-mutations},i), because mutation computation of these edges do not rely on non-parallel conditions.

For iii), Lemma~\ref{pollist} gives $\mathcal{V}(Q_1)=\{v_1^\prime,v_2^\prime,v_4^\prime+h(v_4^\prime)f,v_1^\prime+h_{\max}f\}$, where $f=f_{23}^\prime$ and $h=w_{23}^\prime$.
The first pair of opposite edges is not parallel since $\det(f,v_4^\prime-v_2^\prime)=h(v_4^\prime)-h_{\min}\ne0$.
The second is parallel only if $c^2-(c')^2=a^2$, equivalently $2b(c-c')=a$, which contradicts $c'<b<c$.
Thus $Q_1$ has no parallel edges.

For $Q_2$, put $f=f_{41}^\prime$ and $h=w_{41}^\prime$.
As in Theorem~\ref{thm:Mori-mutations}, $\mathcal{V}(Q_2)=\{v_1^\prime,v_2^\prime,v_2^\prime+h(v_2^\prime)f,v_3^\prime+h(v_3^\prime)f\}$.
One pair of opposite edges is parallel only if $\Delta_1=b^2+(c')^2$, which is excluded by part ii).
The other is parallel only if $c^2-(c')^2=b^2$, equivalently $2a(c-c')=b$.
Since $\gcd(a,b)=1$, this gives $a=1$ and $4c=5b$, contradicting \eqref{eq:Markov0}.
Hence $Q_2$ also has no parallel edges.
\end{proof}

By part~iii), neither $Q_1$ nor $Q_2$ has a pair of parallel edges.
Therefore Convention~\ref{conv:orientation} applies to both, and all results of the preceding subsection, in particular Lemma~\ref{lem:F1moritrain}, Lemma~\ref{lem:sign}, Propositions~\ref{flipray} and~\ref{prop:extremal}, and Theorems~\ref{contractionteo}, \ref{thm:triangle-mutations}, and~\ref{thm:Mori-mutations} describe their deformation theory.
Parts~i) and~ii) of Proposition~\ref{prop:Q-mutations} identify which case of that theory applies to each mutation of $Q$.

\bibliographystyle{habbvr}
\bibliography{bib}

\end{document}